\documentclass[11pt,times,twoside]{article}
\usepackage{graphicx,latexsym,euscript,makeidx,color,bm}
\usepackage{amsmath,amsfonts,amssymb,amsthm,thmtools,mathtools,mathrsfs,enumerate}
\usepackage[colorlinks,linkcolor=blue,anchorcolor=green,citecolor=red]{hyperref}

\usepackage{geometry}
\def\5n{\negthinspace \negthinspace \negthinspace \negthinspace \negthinspace }
\def\4n{\negthinspace \negthinspace \negthinspace \negthinspace }
\def\3n{\negthinspace \negthinspace \negthinspace }
\def\2n{\negthinspace \negthinspace }
\def\1n{\negthinspace }
 \def\sA{\mathscr{A}}    
 \def\sB{\mathscr{B}}

\def\dbE{\mathbb{E}} \def\sE{\mathscr{E}}    
\def\dbF{\mathbb{F}} \def\sF{\mathscr{F}}    
 \def\sG{\mathscr{G}}        
\def\dbH{\mathbb{H}}   \def\cH{{\cal H}}

   \def\cL{{\cal L}}  
     
\def\dbN{\mathbb{N}}     
     
\def\dbP{\mathbb{P}}   \def\cP{{\cal P}}  
\def\dbQ{\mathbb{Q}}     
\def\dbR{\mathbb{R}}     
     
 \def\sT{\mathscr{T}}

   \def\cW{{\cal W}}     \def\wW{\widetilde{W}}
     
         \def\by{\bar{y}}
   \def\cZ{{\cal Z}}      \def\bz{\bar{z}}
\def\Om{\Omega}           \def\Th{\Theta} \def\sgn{\mathop{\rm sgn}}

\def\ss{\smallskip}                
\def\ms{\medskip}                
               
\def\ds{\displaystyle}           
\def\ra{\rightarrow}      
 
\def\no{\noindent}        \def\q{\quad}                      
\def\ns{\noalign{\ss}}    \def\qq{\qquad}                    
    \def\hb{\hbox}                     
                   
         \def\rf{\eqref}                    
  \def\deq{\triangleq}               
            \def\({\Big (}
\def\les{\leqslant}                  \def\){\Big )}
\def\leq{\leqslant}       \def\geq{\geqslant}
\def\ges{\geqslant}       \def\esssup{\mathop{\rm esssup}}   \def\[{\Big[}
           \def\]{\Big]}
                   \def\cd{\cdot}
\def\wt{\widetilde}              \def\cds{\cdots}
                            
\def\nn{\nonumber}        \def\ts{\times}

\def\a{\alpha}        \def\G{\Gamma}   \def\g{\gamma}   \def\O{\Omega}   \def\o{\omega}
\def\b{\beta}         \def\D{\Delta}   \def\d{\delta}        
         \def\Th{\Theta}  \def\th{\theta}    
\def\e{\varepsilon}   \def\L{\Lambda}  \def\l{\lambda}  \def\m{\mu}      
    \def\t{\tau}     \def\f{\varphi}  \def\i{\infty}   

\def\bde{\begin{definition}\label}    \def\ede{\end{definition}}
\def\be{\begin{equation}}
\def\bel{\begin{equation}\label}      \def\ee{\end{equation}}
\def\bt{\begin{theorem}\label}        \def\et{\end{theorem}}
\def\bc{\begin{corollary}\label}      \def\ec{\end{corollary}}
\def\bl{\begin{lemma}\label}          \def\el{\end{lemma}}
\def\bp{\begin{proposition}\label}    \def\ep{\end{proposition}}
\def\bas{\begin{assumption}\label}    \def\eas{\end{assumption}}
\def\br{\begin{remark}\label}         \def\er{\end{remark}}
\def\bex{\begin{example}\label}       \def\ex{\end{example}}
\def\ba{\begin{array}}                \def\ea{\end{array}}
\def\ben{\begin{enumerate}}           \def\een{\end{enumerate}}

\newtheorem{theorem}{Theorem}[section]
\newtheorem{definition}[theorem]{Definition}
\newtheorem{proposition}[theorem]{Proposition}
\newtheorem{corollary}[theorem]{Corollary}
\newtheorem{lemma}[theorem]{Lemma}
\newtheorem{remark}[theorem]{Remark}
\newtheorem{example}[theorem]{Example}

\newtheorem{assumption}{Assumption}

\makeatletter
   
   \@addtoreset{equation}{section}
\makeatother

\allowdisplaybreaks[4]

\begin{document}

\title{\bf Mean-field backward stochastic differential equations and nonlocal PDEs with quadratic growth}

\author{Tao Hao\thanks{School of Statistics and Mathematics, Shandong University of Finance and Economics, Jinan 250014, China (Email: {\tt taohao@sdufe.edu.cn}).
TH is partially supported by Natural Science Foundation of Shandong Province (Grant No. ZR2020MA032),
National Natural Science Foundation of China (Grant Nos. 11871037, 72171133).}~,~~
Ying Hu\thanks{Univ. Rennes, CNRS, IRMAR - UMR 6625, F-35000, Rennes, France (Email: {\tt ying.hu@univ-rennes1.fr}). YH is partially supported by Lebesgue Center of Mathematics ``Investissements d'avenir'' program-ANR-11-LABX-0020-01, by CAESARS-ANR-15-CE05-0024 and by MFG-ANR16-CE40-0015-01.}~,~~
Shanjian Tang\thanks{Department of Finance and Control Sciences, School of Mathematical Sciences, Fudan University, Shanghai 200433, China (Email: {\tt sjtang@fudan.edu.cn}). ST is partially supported by National Science Foundation of China (Grant No. 12031009) and
National Key R/D Program of China (Grant No. 2018YFA0703900).
}~,~~
Jiaqiang Wen\thanks{Corresponding Author. Department of Mathematics, Southern University of Science and Technology, Shenzhen 518055, China (Email: {\tt wenjq@sustech.edu.cn}). JW is partially supported by National Natural Science Foundation of China (grant No. 12101291) and Guangdong Basic and Applied Basic Research Foundation (grant No. 2022A1515012017), and Shenzhen Fundamental Research General Program (Grant No. JCYJ20230807093309021).}}
%
\maketitle

\no\bf Abstract: \rm
In this paper, we study general mean-field backward stochastic differential equations (BSDEs, for short) with quadratic growth. First, the existence and uniqueness of local and global solutions are proved with some new ideas for a one-dimensional mean-field BSDE when the generator $g\big(t, Y, Z, \mathbb{P}_{Y}, \mathbb{P}_{Z}\big)$ has a quadratic growth in $Z$ and the terminal value is bounded. Second, a comparison theorem for the general mean-field BSDEs is obtained with the Girsanov transform. Third, we prove the convergence of the particle systems to the mean-field BSDEs with quadratic growth, and the convergence rate is also given. Finally, in this framework, we use the mean-field BSDE to provide a probabilistic representation for the viscosity solution of a nonlocal partial differential equation (PDE, for short) as an extended nonlinear Feynman-Kac formula, which yields the existence and uniqueness of the solution to the PDE.

\ms

\no\bf Key words: \rm Backward stochastic differential equation, mean-field, quadratic growth, partial differential equation, McKean-Vlasov equation.

\ms

\no\bf AMS subject classifications. \rm 60H10, 60H30

\section{Introduction}

Mean-field stochastic differential equations (SDEs, for short), also called McKean-Vlasov equations, can be traced back to the work of Kac \cite{Kac-1956} in the 1950s.
Recently, inspired by particle systems, the mean-field backward stochastic differential equations (BSDEs, for short) were introduced by  Buckdahn, Djehiche, Li, and Peng \cite{Buckdahn-D-Li-Peng-09}  and Buckdahn, Li, and Peng \cite{BLP}. Since then, mean-field BSDEs and the related nonlocal partial differential equations (PDEs, for short) have received an intensive attention.
However, up to now, most works on mean-field BSDEs assume  the linear growth of the generators, which essentially  hinders the theory's further development and application. This paper will comprehensively study mean-field BSDEs with quadratic growth, including the existence and uniqueness result, the comparison theorem, the particle systems, and their applications to PDEs.
To be more precise, we describe the problem in detail.

Let $(\Om,\sF,\dbF,\dbP)$ be a complete filtered probability space on which a $d$-dimensional standard Brownian motion $\{W_t\;;0\les t<\i\}$ is defined, where $\dbF=\{\sF_t\}_{t\ges0}$ is the natural filtration of $W$ augmented by all the $\dbP$-null sets in $\sF$. For a finite time $T>0$,  consider the following general mean-field BSDE:
\bel{MFBSDE}
Y_t=\eta+\int_t^Tg\big(s,Y_s,Z_s, \dbP_{Y_s},\dbP_{Z_s}\big)ds-\int_t^TZ_sdW_s,\q t\in[0,T],
\ee
where the random variable $\eta$ is called the {\it terminal value} and the coefficient $g$ is called the {\it generator}.
The unknown processes, called an adapted solution of \rf{MFBSDE}, are the pair $(Y,Z)$ of $\dbF$-adapted processes, with
$\mathbb{P}_{Y_s}$ and $\mathbb{P}_{Z_s}$ being the laws (or  distributions) of $Y_s$ and $Z_s$, respectively.
Here, by the word ``general'', we mean that the generator $g$ depends on the marginal distributions of the unknown processes rather than merely their expectations.
In what follows,  BSDE~\rf{MFBSDE}  is called a {\it quadratic} mean-field BSDE or a mean-field BSDE {\it with quadratic growth} if  the generator $g$ has a  quadratic growth  in the  second unknown argument $Z$, and  the terminal value $\eta$ is called {\it bounded} if it is essentially bounded.

When the  marginal laws of both unknown processes $Y$ and $Z$ appear merely via their expectations, the mean-field BSDE \rf{MFBSDE} was studied by Buckdahn, Djehiche, Li, and Peng \cite{Buckdahn-D-Li-Peng-09}  and Buckdahn, Li, and Peng \cite{BLP}, where the existence, uniqueness, a comparison theorem, and the relation with a nonlocal PDE are given when the coefficients are  uniformly Lipschitz continuous.
Recently, the derivative of a functional $\varphi: \mathcal{P}_{2}(\mathbb{R}^{d}) \rightarrow \mathbb{R}$ with respect to the measure argument is introduced by Lions \cite{Lions-13} at {\it Coll\`{e}ge de France}. Since then, this definition has been  adopted in many works.
%
%
For instance, Chassagneux, Crisan, and Delarue \cite{CC-Delarue-15} (see also Carmona and Delarue \cite{Carmona-Delarue-18}) studied the general mean-field BSDE \rf{MFBSDE} coupled with a McKean-Vlasov forward equation, and proved that this class of equations admit  unique adapted solutions under uniformly Lipschitz continuous coefficients.  Buckdahn, Li, Peng, and Rainer \cite{Buckdahn-Li-Peng-Rainer-17} studied the general forward mean-field stochastic differential equations and the associated PDEs. Li \cite{Li-18} studied the general mean-field forward-backward SDEs with jumps and associated nonlocal quasi-linear integral-PDEs.
%
%
%
%
Besides, for the applications of the mean-field framework in stochastic control problems,
Yong \cite{Yong-13} studied a linear-quadratic optimal control problem of mean-field SDEs,
and Buckdahn, Chen, and Li \cite{Buckdahn-Chen-Li-21} studied the partial derivative with respect to the measure and its application to general controlled mean-field systems.

On the one hand, when the generator $g$ is independent of  $(\mathbb{P}_{Y},\mathbb{P}_{Z})$, the general mean-field BSDE \rf{MFBSDE} is reduced to the following  BSDE:
\begin{align}\label{BSDE}
	Y_t=\eta+\int_t^Tg(s,Y_s,Z_s)ds - \int_t^TZ_{s}dW_{s},\q~0\les t\les T,
\end{align}
which were introduced by Pardoux and Peng \cite{Pardoux-Peng-90}, where the existence and uniqueness were obtained for the case of  uniformly Lipschitz continuous coefficients.
From then on,  BSDEs have received an extensive attention in various fields of
partial differential equations (see Pardoux and Peng \cite{Pardoux-Peng-92}), mathematical finance (see El Karoui, Peng, and Quenez \cite{Karoui-Peng-Quenez-97}), and stochastic optimal control (see Yong and Zhou \cite{Yong-Zhou-99}), to mention a few.
At the same time, due to various  applications as well as the list of open problems proposed by Peng \cite{Peng-98}, many efforts have been made to relax the existing conditions on the generator $g$ of BSDE \rf{BSDE} for the existence and/or uniqueness of adapted solutions.
For instance, Lepeltier and San Martin \cite{Lepeltier-Martin-97} obtained the existence of adapted solutions to BSDEs when the generator $g$ is continuous and of linear growth in $(Y,Z)$.
In 2000, Kobylanski \cite{Kobylanski00} proved the existence and uniqueness result on adapted solution of a one-dimensional BSDE \rf{BSDE} when the generator $g$ is of quadratic growth in $Z$ and the terminal value $\eta$ is bounded.
Along this way,  the existence and uniqueness result on BSDEs \rf{BSDE} with quadratic growth, in the one-dimensional situation with unbounded terminal values,  was obtained by Briand and Hu \cite{Briand-Hu-06,Briand-Hu-08} and Bahlali, Eddahbi, and Ouknine \cite{Bahali-Eddahbi-Ouknine-17};
the multi-dimensional situation with bounded terminal values was studied by Hu and Tang \cite{Hu-Tang-16} and Xing and Zitkovic \cite{Xing-Zitkovic};
and the multi-dimensional situation with unbounded terminal values was investigated by Fan, Hu and Tang \cite{Fan--Hu--Tang-20}, under various conditions and with different methods.
Some other recent developments concerning the quadratic BSDEs can be found in
Barrieu and El Karoui \cite{Barrieu-Karoui-13}, Delbaen, Hu, and Bao \cite{Delbaen-Hu-Bao11}, Fan, Hu, and Tang \cite{FHT}, Hu, Li, and Wen \cite{Hu-Li-Wen-JDE2021}, and so on.

On the other hand,  the last two decades, stimulated by the broad applications and the open problem proposed by Peng \cite{Peng-98}, also see a lot of efforts at relaxing the conditions on the generator $g$ of the mean-field BSDEs \rf{MFBSDE}.
When the generator $g$ depends on the expectation of the unknown processes $(Y,Z)$,
Cheridito and Nam \cite{Cheridito-Nam-17} discussed the existence of adapted solutions to a class of mean-field BSDEs with quadratic growth;
Hibon, Hu, and Tang \cite{Hibon-Hu-Tang-17} studied the existence and uniqueness of adapted solutions to one-dimensional mean-field BSDEs with quadratic growth, and Hao, Wen, and Xiong \cite{Hao-Wen-Xiong-22} studied a class of multi-dimensional quadratic mean-field BSDEs with small terminal values.
However, to our best knowledge, there are few works concerning the quadratic mean-field BSDEs when the generator $g$ depends on the laws of unknown processes $(Y,Z)$, let alone the particle systems.

In this paper, by introducing some new ideas, we study the one-dimensional quadratic mean-field BSDE \rf{MFBSDE} with bounded terminal values.
Firstly, we construct a local solution for the general mean-field BSDE \rf{MFBSDE} (see \autoref{th 3.2}) by borrowing some ideas of Hu and Tang \cite{Hu-Tang-16} and  using the fixed-point principle. It should be noted that the generator $g$ has a general growth with respect to $Y$ (see \autoref{Example1}), and the choice of the radius of  the centered ball is more straightforward  compared to  that  of Hu and Tang \cite[Theorem 2.2]{Hu-Tang-16}.
Secondly, using different methods, we prove the existence and uniqueness of the global solutions of mean-field BSDE \rf{MFBSDE} under the following four cases:
(i)  the generator $g$ has a quadratic growth with respect to $Z$ and is bounded with respect to $\dbP_Z$ (see \autoref{th 3.5}) ;
(ii)   the generator $g$ has a strictly quadratic growth with respect to $Z$ and   has a sub-quadratic growth with respect to $\dbP_Z$   (see  \autoref{th 3.5-1});
(iii)  the generator $g$ has  a quadratic and sub-quadratic growth with respect to $Z$ and $\dbP_Z$, respectively,  with  the growth coefficient   of $g$ in  $\dbP_Z$  being sufficiently small (see \autoref{th 3.5-12});
(iv)  the generator $g(t,Y,Z, \dbP_{Y},\dbP_{Z}):=g_1(t,Z)+g_2(t,\dbP_{Z})$, where   both functionals $g_1$  and $g_2$ have  a quadratic growth in their last arguments  (see \autoref{24.1.13}).
Thirdly, to receive a wider application, we give a comparison theorem for this class of BSDEs  (see \autoref{th 4.1}) with the Girsanov transform.
Fourthly, we consider the following  system of $N$ particles: for $t\in[0,T]$,
\begin{equation}\label{22.9.25.1}
	Y^i_t=\eta^i
	+\int_t^T\!\!\! g\Big(s,Y^i_s,Z^{i,i}_s,\frac{1}{N}\sum\limits_{i=1}^N\delta_{Y^i_s},
	\frac{1}{N}\sum\limits_{i=1}^N\d_{Z^{i,i}_s}\Big)ds
	-\int_t^T\sum\limits_{j=1}^NZ^{i,j}_sdW^j_s,
\end{equation}
where $\d$ is the Dirac measure, $\{\eta^i, 1\leq i\leq N\}$ are $N$ independent copies of $\eta$, and
$\{W^j, 1\leq j\leq N\}$ are $N$ independent $d$-dimensional Brownian motions.
With Lions' idea and the law of large numbers,  we prove that the mean-field limit of the  $N$-particle system \rf{22.9.25.1} converges to the mean-field BSDE \rf{MFBSDE} (see \autoref{th 8.1}) when $N$ tends to infinity. Moreover,  we obtain the  rate of convergence  (see \autoref{th 8.2-1}) when the generator $g$ does not depend on the law of $Z$.
Finally, we use the mean-field BSDE \rf{MFBSDE} with quadratic growth to prove the existence and uniqueness of the viscosity solution of a nonlocal PDE, and thus extend the nonlinear Feynman-Kac formula of Buckdahn, Li, and Peng \cite{BLP} with linear growth to the case of quadratic growth (see \autoref{th 5.7}).
In our nonlinear Feynman-Kac formula, the generator $g$ depends on the  distribution only via the expectation of the state process $(Y, Z)$. Since Wasserstein space has no local compactness, it remains   a challenging topic in the theory of viscosity solutions to allow the generator to depend on the distribution in a general way (see Wu and Zhang \cite{WZ}).

The rest of this paper is organized as follows. In Section \ref{Sec2}, we present some preliminary notations and results. In Section \ref{Sec3}, we prove the existence and uniqueness of the local and global solutions to the general mean-field BSDE \rf{MFBSDE}. In Section \ref{Sec4}, a comparison theorem is proved.
We study the particle systems for  mean-field BSDEs in Section \ref{Sec7}, where the convergence and its rate  are given for the particle systems. In Section \ref{Sec5}, we provide the relationship between the solution of the mean-field BSDE with quadratic growth and the viscosity solution of the related nonlocal PDE. Section \ref{Sec6} concludes the results.

\section{Preliminaries} \label{Sec2}

Let $(\Om,\sF,\dbF,\dbP)$ be a complete filtered probability space on which a $d$-dimensional standard Brownian motion $\{W_t\ ;0\les t<\i\}$ is defined, where $\dbF=\{\sF_t;0\les t<\i\}$ is the natural filtration of $W$ augmented by all the $\dbP$-null sets in $\sF$.
The notion $\dbR^{m\ts d}$ is the space of $m\ts d$-matrix $C$ with Euclidean norm $|C|=\sqrt{tr(CC^\intercal)}$.
For two real numbers $a$ and $b$, denote by $a \wedge b$ and $a \vee b$ the minimum and maximum of them, respectively. Set $a^{+}=a \vee 0$ and $a^{-}=-(a \wedge 0)$. Denote by $\chi_{A}$ the indicator of set $A$, and $\operatorname{sgn}(x)=$ $\chi_{\{x>0\}}-\chi_{\{x < 0\}}$ with $\operatorname{sgn}(0)=0$. For some positive real number $b$, by $[b]$, we denote the largest integer not exceeding $b$.
Let $M$ be a continuous local martingale, and denote 
$$\mathcal{E}(M)^t_0:=\exp(M_t-\frac{1}{2}\langle M \rangle_t),\quad 0\leq t<\infty.$$
The notation $\d_{\{a\}}$ denotes the Dirac measure at $a$.

In addition, for any $p\ges1$, $t\in[0,T)$, and Euclidean space $\dbH$, we introduce the following spaces:
\begin{align*}
	L^p_{\sF_t}(\Om;\dbH)=\Big\{& \xi:\Om\to\dbH\bigm|\xi\hb{ is $\sF_t$-measurable, }\|\xi\|_{L^p(\Om)}\deq\big(\dbE|\xi|^p\big)^{1\over p}<\i\Big\},  \\
	L_{\sF_t}^\i(\Om;\dbH)=\Big\{& \xi:\Om\to\dbH\bigm|\xi\hb{ is $\sF_t$-measurable, }
	\|\xi\|_{\i}\triangleq\esssup_{\o\in\Om}|\xi(\omega)|<\i\Big\},\\
	L_\dbF^p(t,T;\dbH)=\Big\{& \f:\Om\ts[t,T]\to\dbH\bigm|\f\hb{ is
		$\dbF$-progressively measurable, }\\
	&\qq\qq\qq
	\|\f\|_{L_\dbF^p(t,T)}\deq\dbE\[\(\int^T_t|\f_s|^2ds\)^{p\over2}\]^{\frac{1}{p}}<\i\Big\},\\
	S_\dbF^p(t,T;\dbH)=\Big\{&\f:\Om\ts[t,T]\to\dbH\bigm|\f\hb{ is
		$\dbF$-adapted, continuous, }\\
	&\qq\qq\qq
	\|\f\|_{S_\dbF^p(t,T)}\deq\Big\{\dbE\(\sup_{s\in[t,T]}|\f_s|^p\)\Big\}^{\frac{1}{p}}<\i\Big\},\\
	S_\dbF^\infty(t,T;\dbH)=\Big\{&\f:\Om\ts[t,T]\to\dbH\bigm|\f\hb{ is
		$\dbF$-adapted, continuous,  }\\
	&\qq\qq\qq
	\|\f\|_{S_\dbF^\infty(t,T)}\deq\esssup_{(s,\o)\in[t,T]\times\Om}|\f_s(\o)|<\i\Big\},\\
	\cZ^2_\dbF(t,T;\dbH)=\Big\{& Z\in L^2_\dbF(t,T;\dbH)\Bigm|\|Z\|_{\cZ^2_\dbF(t,T)}\deq\sup_{\t\in\sT[t,T]}\Big\|
	\dbE_\t\[\int_\t^T|Z_s|^2ds\] \Big\|_\i^{1\over2}<\i\Big\},
\end{align*}
where $\sT[t,T]$ is the set of all $\dbF$-stopping times $\t$ valued in $[t,T]$ and $\dbE_\t$ is the  expectation conditioned on the $\sigma$-field $\sF_{\t}$.
We denote by $\dbE^\dbQ$ and $\dbE^\dbQ_\t$ the expectation operator and the conditional expectation operator with respect to the probability measure $\dbQ$, respectively.
Moreover, let $M=(M_t,\sF_t)$ be a uniformly integrable martingale with $M_0=0$, and we set
$$\|M\|_{BMO_p(\dbP)}\deq
\sup_\t\bigg\|\dbE_\t\Big[\( \langle M\rangle_{\i} -\langle M\rangle_{\t} \)^{\frac{p}{2}}\Big]^{\frac{1}{p}}\bigg\|_{\i}.$$
The class $\big\{M: \|M\|_{BMO_p(\dbP)}<\i\big\}$ is denoted by $BMO_p(\dbP)$. Note that $BMO_p(\dbP)$ is a Banach space under the norm $\| \cd \|_{BMO_p(\dbP)}$. In the sequel,  we write $BMO(\dbP)$ for the space $BMO_2(\dbP)$.
Note that the process $t\mapsto\int_0^tZ_sdW_s$ on $[0,T]$ (denoted by $Z\cd W$) belongs to $BMO(\dbP)$ if and only if $Z\in \cZ^2_\dbF(0,T;\dbH)$, i.e.,
\begin{equation}\label{4.3.1}
	\|Z\cd W\|_{BMO(\dbP)}\equiv \|Z\|_{\cZ^2_\dbF(0,T)}.
\end{equation}

For $p\geq2$, let $\mathcal{P}_{p}(\mathbb{R}^{d})$ be the set of all probability measures $\mu$ on $(\mathbb{R}^{d}, \sB(\mathbb{R}^{d}))$ with finite $p$-th moment, i.e., $\int_{\mathbb{R}^{d}}|x|^{p} \mu(d x)<\infty$. Here $\sB(\mathbb{R}^{d})$ denotes the Borel $\sigma$-field over $\mathbb{R}^{d}$.
In addition, the set $\mathcal{P}_{p}(\mathbb{R}^{d})$ is endowed with the following $p$-Wasserstein metric: for $\mu, \nu \in \mathcal{P}_{p}(\mathbb{R}^{d})$,
\begin{align*}
	\cW_{p}(\mu, \nu)\deq\inf \bigg\{ & \(\int_{\mathbb{R}^{d} \times \mathbb{R}^{d}}|x-y|^{p}  \rho(d x d y)\)^{\frac{1}{p}}\Bigm| \\
	& \rho \in \mathcal{P}_{p}(\mathbb{R}^{2 d}),\  \rho(. \times \mathbb{R}^{d})=\mu,\ \rho(\mathbb{R}^{d} \times .)=\nu\bigg\}.
\end{align*}
Now we let $p=2$ and suppose that there exists a sub-$\sigma$-algebra $\sG$ of $\sF$ which is independent of $\sF_{\infty}$ and will be assumed ``rich enough'', as explained below:
for every $\mu \in \mathcal{P}_{2}(\mathbb{R}^{d})$ there is a random variable $\vartheta \in L^{2}_{\sG}(\Om; \mathbb{R}^{d})$ such that $\mathbb{P}_{\vartheta}=\mu$. It is well known that the probability space $([0,1], \sB([0,1]), d x)$ has this property.
Then we call that a function $h:\mathcal{P}_2(\mathbb{R}^d)\rightarrow \mathbb{R}$ is differentiable in $\mu_0\in\mathcal{P}_2(\mathbb{R}^d)$, if there exists a $\xi_0\in L^2_{\sG}(\Om;\mathbb{R}^d)$ with $\mu_0=\mathbb{P}_{\xi_0}$, such that the lifted function $\bar{h}: L^2_{\sG}(\Om;\mathbb{R}^d)\rightarrow \mathbb{R}$ defined by $\bar{h}(\xi)\deq h(\mathbb{P}_\xi)$ has {\it Fr\'{e}chet} derivative at $\xi_0$.
In other words, there exists a continuous linear functional $D\bar{h}(\xi_0): L^2_{\sG}(\Om;\mathbb{R}^d)\rightarrow \mathbb{R}$, such that
for any $\eta\in L^2_{\sG}(\Om;\mathbb{R}^d)$,
\begin{equation}\label{equ 2.3}
	\bar{h}(\xi_0+\eta)-\bar{h}(\xi_0)=D\bar{h}(\xi_0)(\eta)+o\big(||\eta||_{L^2(\Om)}\big)\q\hb{with}\q ||\eta||_{L^2(\Om)}\rightarrow0.
\end{equation}
Riesz representation theorem and the argument of Cardaliaguet \cite{Car} show that
there exists a Borel measurable function $\psi:\mathbb{R}^d\rightarrow\mathbb{R}^d$ depending only on the law of $\xi_0$, but not on the random variable $\xi_0$ itself, such that the forward term of (\ref{equ 2.3}) can be written as
\begin{equation}\label{equ 2.4}
	h(\mathbb{P}_{\xi_0+\eta})-h(\mathbb{P}_{\xi_0})=
	\mathbb{E}[\psi(\xi_0)\cdot\eta]+o\big(||\eta||_{L^2(\Om)}\big).
\end{equation}
Then, according to (\ref{equ 2.4}), Buckdahn, Li, Peng, and Rainer \cite{Buckdahn-Li-Peng-Rainer-17} first define $\partial_\mu h(\mathbb{P}_{\xi_0};a)\deq \psi(a)$ for every $a\in\mathbb{R}^d$, which is called the derivative of $h$ at $\mathbb{P}_{\xi_0}$.
Remark that
the function $\partial_\mu h(\mathbb{P}_{\xi_0};a)$ is only $\mathbb{P}_{\xi_0}(da)$-a.e. uniquely determined.
%
\begin{definition}\label{Solution}\rm
	A pair of processes $(Y,Z)\in S^2_{\dbF}(0,T;\dbR^m)\times
	L^2_\dbF(0,T;\dbR^{m\ts d})$ is called an {\it adapted solution} of mean-field BSDE \rf{MFBSDE} if it satisfies \rf{MFBSDE} $\dbP$-almost surely, and  a bounded adapted solution if it further belongs to  $S^\infty_\dbF(0,T;\dbR^m)\times\cZ^2_\dbF(0,T;\dbR^{m\ts d})$.
\end{definition}
Now, we recall the following propositions concerning BMO-martingales, which are slightly different from  those of Kazamaki 
\cite[Chapters 2 and 3]{Kazamaki-06}.
\begin{proposition}[The Reverse H\"{o}lder Inequality]\label{22.9.1.1} \it 
	Let $p \in(1, \infty)$ and $M$ be a one-dimensional continuous $B M O$ martingale. If $\|M\|_{B M O(\dbP)}<\Phi(p)$, then $\sE(M)$ satisfies the reverse H\"{o}lder inequality:
	$$
	\dbE_\tau\left[\mathscr{E}(M)_\tau^{\infty}\right]^p \leq c_p,
	$$
	for any stopping time $\tau$, with a positive constant $c_p$ depending only on $p$.
\end{proposition}
\begin{proposition}\label{2.1.11} \it 
	For $\widetilde{K}>0$, there are constants $c_1>0$ and  $c_2>0$ depending on $\widetilde{K}$ such that for any
	BMO-martingale $M$ and any one-dimensional BMO-martingale $N$ such that $\|N\|_{BMO(\dbP)}\les \widetilde{K}$, we have
	$$c_1\|M\|_{BMO(\dbP)}\les\|\widetilde{M}\|_{BMO(\widetilde{\dbP})}\les c_2\|M\|_{BMO(\dbP)},$$
	where $\widetilde{M}\triangleq M-\langle M,N\rangle$ and $d\widetilde{\dbP}\triangleq\sE(N)_0^{\i}d\dbP$.
\end{proposition}

\section{Existence and Uniqueness}\label{Sec3}

In this section, we study the existence and uniqueness of an adapted solution to a mean-field BSDE \rf{MFBSDE} with quadratic growth. In particular, we consider the  one-dimensional situation, i.e., $m=1$. We shall study the local and global solutions and divide this section into two parts: Subsection \ref{Sec3.1} for the local solution and Subsection \ref{Sec3.2} for the global solution.

In the following, suppose that $\theta:\Om\ts[0,T]\rightarrow\dbR^+$ is an $\sF_t$-progressively measurable process, $\phi,\phi_0:[0,+\infty)\rightarrow[0,+\infty)$ are two nondecreasing continuous functions, and $\b$, $\b_0$, $\g$, $\g_0$, $\tilde{\g}$, and $\alpha\in[0,1)$ are all positive constants.

\subsection{Local solution}\label{Sec3.1}

In this subsection, we prove the existence and uniqueness of local  solutions of mean-field BSDE \rf{MFBSDE} with quadratic growth. 

\bas{ass 3.1}\rm  The terminal value $\eta:\Om\rightarrow\dbR$ and the generator $g: \Omega\times[0,T]\times\dbR\times\dbR^d\times \cP_2(\dbR)\ts \cP_2(\dbR^{d}) \rightarrow \dbR$ satisfy the following conditions:
\begin{enumerate}[~~\,\rm (i)]
	\item $d\dbP\times dt$-a.e. $(\o,t)\in\Om\ts[0,T]$, for $(y, z, \mu_1,\mu_2)\in \mathbb{R}\times\mathbb{R}^d\times \mathcal{P}_2(\mathbb{R})\times \mathcal{P}_2(\mathbb{R}^{d})$,
	$$ |g(w,t,y,z,\mu_1,\m_2)|\les \theta_t(\omega)+\phi(|y|)+\frac{\gamma}{2}|z|^2
	+\phi_0\big(\cW_2(\mu_1,\d_{\{0\}})\big)+\gamma_0\cW_2(\mu_2,\d_{\{0\}})^{1+\alpha}.$$
	\item $d\dbP\times dt$-a.e. $(\o,t)\in\Om\ts[0,T]$,
	$$\ba{ll}
	& |g(\o,t,y,z,\m_1,\m_2)-g(\o,t,\by,\bz,\bar{\m}_1,\bar{\m}_2)| \\
	&\les \phi\(|y|\vee|\bar{y}|\vee\cW_2(\mu_1,\d_{\{0\}})\vee  \cW_2(\bar{\mu}_1,\d_{\{0\}})\)\\
	&\q \times  \bigg[\Big(
	1+|z|+|\bar{z}|+\cW_2(\mu_2,\d_{\{0\}})+\cW_2(\bar{\mu}_2,\d_{\{0\}})\Big)
	\(|y-\bar{y}|+|z-\bar{z}|+\cW_2(\mu_1,\bar{\mu}_1)\)\\
	&\qq   +\Big(1+\cW_2(\mu_2,\d_{\{0\}})^{\alpha}
	+\cW_2(\bar{\mu}_2,\d_{\{0\}})^{\alpha}\Big)\cW_2(\mu_2,\bar{\mu}_2)\bigg]
	\ea$$
	for any $(y, z; \by, \bz; \mu_1,\mu_2; \bar{\m}_1, \bar{\m}_2)\in (\mathbb{R}\times\mathbb{R}^d)^2\times (\mathcal{P}_2(\mathbb{R})\times \mathcal{P}_2(\mathbb{R}^{d}))^2$.
	\item  There are two positive constants $K_1$ and $K_2$ such that
	$$\|\eta\|_\infty\les K_1\q\hb{and}\q \Bigl\|\int_0^T\!\! \theta_t\, dt\Bigr\|_\infty\les K_2.$$
\end{enumerate}\eas

\begin{remark}\label{Example1}\rm
	\autoref{ass 3.1} is more general than the condition ($\sA1$) of Hibon, Hu, and Tang \cite{Hibon-Hu-Tang-17} in that here we relax the growth and continuity of the generator $g$ with respect to the first variable $y$. For example, the generator $g$ below satisfies \autoref{ass 3.1}, but fails to satisfy the condition ($\sA1$) of \cite{Hibon-Hu-Tang-17}:
	\begin{align*}
		& g(t,y,z,\mu_1,\mu_2)
		:=|y|^2|z|+|z|^2+\cW_2(\mu_1,\d_{\{0\}})^3\cos(\cW_2(\mu_2,\d_{\{0\}}))+\cW_2(\mu_2,\d_{\{0\}})^\frac{3}{2}
	\end{align*}
	for every $(t, y, z) \in [0,T]\times \dbR\times \dbR^d$ and $(\m_1, \mu_2)\in\cP_2(\dbR)\times \cP_2(\dbR^{d})$. 
\end{remark}

The following proposition is an essential extension of  Hu and Tang \cite[Lemma 2.1]{Hu-Tang-16} to  the case of mean-field-dependent  generators.

\begin{proposition}\label{le 7.1} \it 
	Assume that for any given processes $(P,Q)\in S_\dbF^\infty(0,T;\dbR)\ts \cZ^2_\dbF(0,T;\dbR^d)$,
	the terminal value $\eta:\Om\rightarrow\dbR$ and the generator $g: \Omega\times[0,T]\times\dbR^d\rightarrow \dbR$ satisfy the conditions:
	\begin{enumerate}[~~\,\rm (i)]
		\item $d\dbP\times dt$-a.e. $t\in[0,T]$, for any $z\in \dbR^d$,
		$$|g(\omega,t,z)|\leq \theta_t(\omega)+\phi(|P_t|)+\frac{\gamma}{2}|z|^2
		+\phi_0\big(\|P_t\|_{L^2(\Om)}\big)+\gamma_0\|Q_t\|_{L^2(\Om)}^{1+\alpha}.$$
		\item $d\dbP\times dt$-a.e. $t\in[0,T],$ for every $z,\bar{z}\in \dbR^d$,
		$$|g(\omega,t,z)-g(\omega,t,\bar{z})|\leq
		\phi\big(|P_t|\vee\|P_t\|_{L^2(\Om)}\big)\cd
		\big(1+|z|+|\bar{z}|+ 2\|Q_t\|_{L^2(\Om)}  \big)|z-\bar{z}|.$$
		\item  For two positive constants $K_1$ and $K_2$,
		$$\|\eta\|_\infty\les K_1\q\hb{and}\q \Bigl\|\int_0^T\!\!\! \theta_t\, dt\Bigr\|_\infty\les K_2.$$
	\end{enumerate}
	Then the following backward stochastic differential equation
	\begin{equation}\label{equ 7.1}
		Y_t=\eta+\int_t^T\!\!\! g(s,Z_s)\, ds-\int_t^TZ_s\, dW_s, \quad t\in [0,T],
	\end{equation}
	admits a unique adapted solution $(Y,Z)\in S_\dbF^\infty(0,T;\dbR)\ts \cZ^2_\dbF(0,T;\dbR^d)$. Moreover, for any $t\in[0,T]$ and stopping time $\tau\in \sT[t,T]$, we have
	\begin{equation}\label{equ 7.2}
		\begin{aligned}
			|Y_t|\leq &\ \|\eta\|_\infty+\Bigl\|\int_0^T\!\!\! \theta_s\, ds\Bigr\|_\infty
			+(T-t)\widehat{\phi+\phi_0}^t
			+\gamma_0\|Q\|^{1+\alpha}_{\cZ^2_\dbF(t,T)}(T-t)^{\frac{1-\alpha}{2}}
		\end{aligned}
	\end{equation}
	and
	\begin{equation}\label{equ 7.3}
		\begin{aligned}
			&\dbE_\tau\[\int_\tau^T\!\! |Z_s|^2ds\]
			\leq \frac{1}{\gamma^2}\exp\big(2\gamma\|\eta\|_\infty\big)
			+\frac{2}{\gamma}\exp\big(2\gamma\|Y\|_{S_\dbF^\infty(t,T)}\big)\\
			&\quad\quad \times \bigg(\Bigl\|\int_0^T\!\!\! \theta_s\, ds\Bigr\|_\infty
			+(T-t)\widehat{\phi+\phi_0}^t
			+\gamma_0\|Q\|^{1+\alpha}_{\cZ^2_\dbF(t,T)}(T-t)^{\frac{1-\alpha}{2}}\bigg),
		\end{aligned}
	\end{equation}
	where $$\widehat{\phi+\phi_0}^t:=(\phi+\phi_0)(\|P\|_{S_\dbF^\infty(t,T)}).$$ 
\end{proposition}

\begin{proof}
	Using H\"{o}lder's inequality, we have that  for $t\in[0,T]$,
	\begin{equation*}\label{equ 7.3-111}
		\int_t^T (\dbE[|Q_s|^2])^\frac{1+\alpha}{2}ds
		\leq
		\Big(\int_t^T \dbE[|Q_s|^2]ds \Big)^\frac{1+\alpha}{2}(T-t)^{\frac{1-\alpha}{2}}
		\leq \|Q\|^{1+\alpha}_{\cZ^2_\dbF(t,T)}(T-t)^{\frac{1-\alpha}{2}}.
	\end{equation*}
	Then, for a constant $\ell\geq1$,  we have
	$$ \begin{aligned}
		& \dbE_t\[\exp\Big\{ \ell\gamma \gamma_0\int_t^T\|Q_s\|_{L^2(\Om)}^{1+\alpha}ds\Big\}\]\\
		&=\exp\Big\{ \ell\gamma \gamma_0\int_t^T\big(\dbE[|Q_s|^2]\big)^\frac{1+\alpha}{2}ds\Big\}
		\leq \exp\Big\{\ell\gamma \gamma_0\|Q\|^{1+\alpha}_{\cZ^2_\dbF(t,T)}(T-t)^{\frac{1-\alpha}{2}}\Big\}.
	\end{aligned}
	$$
	Consequently,
	\begin{equation*}\label{4.10.1}
		\begin{aligned}
			&\dbE_t\exp\bigg\{
			l\gamma|\eta|+ \ell \gamma \int_t^T\!\! \[(\theta_s(\omega)+\phi(|P_s|)+\phi_0\big(\|P_s\|_{L^2(\Om)}\big) +\gamma_0\|Q_s\|_{L^2(\Om)}^{1+\alpha}\]ds
			\bigg\}\\
			&\leq\exp^{\ell\gamma}\left\{
			\|\eta\|_\infty+\Bigl\|\int_0^T\!\!\! \theta_s\, ds\Bigr\|_\infty\!\!+(T-t)\widehat{\phi+\phi_0}^t
			+\gamma_0\|Q\|^{1+\alpha}_{\cZ^2_\dbF(t,T)}(T-t)^{\frac{1-\alpha}{2}}
			\right\}<\infty.
		\end{aligned}
	\end{equation*}
	In view of the conditions (i)-(iii),
	we see from  Briand and Hu \cite[Proposition 3  and Corollary 4]{Briand-Hu-08} that BSDE \rf{equ 7.1} admits a solution $(Y,Z)$ such that
	$$
	\dbE\int_0^T|Z_s|^2ds<\infty,
	$$
	and
	$$
	\begin{aligned}
		|Y_t|&\leq \|\eta\|_\infty+\Bigl\|\int_0^T\!\!\! \theta_s\, ds\Bigr\|_\infty
		+(T-t)\widehat{\phi+\phi_0}^t
		+\gamma_0\|Q\|^{1+\alpha}_{\cZ^2_\dbF(t,T)}(T-t)^{\frac{1-\alpha}{2}},
	\end{aligned}
	$$
	which implies that  the inequality \rf{equ 7.2} holds and thus $Y\in S_\dbF^\infty(0,T;\dbR)$.

	We now prove that $Z\cdot W$ is a BMO martingale. Applying It\^{o}-Tanaka's formula to the term $\exp\big\{2\gamma|Y_\tau|\big\}$ with $\tau\in\sT[t,T]$, we have that
	\begin{equation*}\label{equ 7.4}
		\begin{aligned}
			&\exp\big\{2\gamma|Y_\tau|\big\}+2\gamma^2\int_\tau^T\!\!\! \exp\big\{2\gamma|Y_s|\big\}|Z_s|^2\, ds\\
			&= \exp\big\{2\gamma|\eta|\big\}+2\gamma\int_\t^T\!\!\! \exp\big\{2\gamma|Y_s|\big\} g(s,Z_s)ds+2\gamma\int_\tau^T\!\!\! \exp\big\{2\gamma|Y_s|\big\} Z_s\, dW_s.
		\end{aligned}
	\end{equation*}
	In view of the condition (i), we have
	\begin{equation*}\label{equ 5.5}
		\begin{aligned}
			&\exp\big\{2\gamma|Y_\tau|\big\}
			+2\gamma^2\, \dbE_\tau \Big[\int_\tau^T\!\!\! \exp\big\{2\gamma|Y_s|\big\}|Z_s|^2ds\Big]
			\leq \dbE_\tau \big[\exp\big\{2\gamma|\eta|\big\}\big]\\
			&
			+2\gamma\, \dbE_\tau\!\! \int_\t^T\!\!\! \exp\big\{2\gamma|Y_s|\big\}
			\Big[
			\theta_s+\phi(|P_s|)+\frac{\gamma}{2}|Z_s|^2
			+\phi_0\big(\|P_s\|_{L^2(\Om)}\big)+\gamma_0\|Q_s\|_{L^2(\Om)}^{1+\alpha}
			\Big]ds.
		\end{aligned}
	\end{equation*}
	Consequently, we have
	\begin{equation*}\label{equ 7.5}
		\begin{aligned}
			&\g^2\, \dbE_\tau\!\! \int_\tau^T\!\! |Z_s|^2ds
			\les \gamma^2\, \dbE_\tau \!\! \int_\tau^T\!\!\! \exp\big\{2\gamma|Y_s|\big\}|Z_s|^2\, ds\\
			&\leq \exp\big\{2\gamma \|\eta\|_\infty\big\}+2\gamma  \left(\Bigl\|\int_0^T\!\!\! \theta_s\, ds\Bigr\|_\infty
			+(T-t)\widehat{\phi+\phi_0}^t
			+\gamma_0\|Q\|^{1+\alpha}_{\cZ^2_\dbF(t,T)}(T-t)^{\frac{1-\alpha}{2}}
			\right)\\
			&\quad \times \exp\left\{2\gamma\|Y\|_{S^\infty_\dbF(t,T)}\right\}.
		\end{aligned}
	\end{equation*}
	Hence, we have inequality  (\ref{equ 7.3}) and thereby $Z\in\cZ^2_\dbF([0,T];\dbR^d)$. In other words, $Z\cdot W$ is a BMO martingale.

	Finally, we prove the uniqueness. In fact, in view of the condition (ii) with $(U, V) \in S_\dbF^\infty(0,T;\dbR)\ts \cZ^2_\dbF(0,T;\dbR^d)$, similar to  the proof of Hu and Tang \cite[Lemma 2.1]{Hu-Tang-16},  we can use the Girsanov transform to derive a comparison result on the solutions of BSDE \rf{equ 7.1}, which yields the desired result.
\end{proof}

We now prove the existence and uniqueness of mean-field BSDE \rf{MFBSDE} in a subset of $S_\dbF^\infty(T-\e,T;\dbR)\ts \cZ^2_\dbF(T-\e,T;\dbR^d)$, where $\e\in(0,T)$ is some constant. In addition, for positive constants $L_1$ and $L_2$, we define the following Banach space:
$$
\begin{aligned}
	\sB_\e(L_1,L_2)\deq \Big\{ &
	(Y,Z)\in S_\dbF^\infty(T-\e,T;\dbR)\ts \cZ^2_\dbF(T-\e,T;\dbR^d),\\
	& \|Y\|_{S_\dbF^\infty(T-\e,T)}\les L_1\q
	\text{and}\q \|Z\|^2_{\cZ^2_\dbF(T-\e,T)}\les L_2
	\Big\},
\end{aligned}
$$
endowed with the norm
$$
\|(Y,Z)\|_{\sB_\e(L_1,L_2)}=
\Big\{
\|Y\|_{S_\dbF^\infty(T-\e,T)}^2
+\|Z\|^2_{\cZ^2_\dbF(T-\e,T)}
\Big\}^\frac{1}{2}.
$$

The main result of this subsection is stated as follows.

\begin{theorem}\label{th 3.2}\it 
	Under \autoref{ass 3.1}, there exists a positive constant $\e$,  that depends only on the functions  $(\phi(\cd),\phi_0(\cd))$ and constants $(\gamma,\gamma_0,\alpha,K_1,K_2)$, such that on the interval $[T-\e,T]$, the mean-field BSDE (\ref{MFBSDE}) has a unique local  solution $(Y,Z)\in\sB_\e(L_1,L_2)$ with
	\begin{equation*}\label{4.7.2}
		L_1=2(K_1+K_2)\q\hb{and}\q
		L_2=\frac{2}{\gamma^2}e^{2\gamma K_1}+\frac{4K_2}{\gamma}e^{2\gamma L_1}.
	\end{equation*}
\end{theorem}

\begin{proof}
	We construct a contraction mapping to prove the existence and uniqueness of an adapted solution to BSDE~\rf{MFBSDE}.
	For any given pair of processes $(P,Q)\in S_\dbF^\infty(0,T;\dbR)\ts \cZ^2_\dbF(0,T;\dbR^d)$, we consider the following BSDE:
	\begin{equation}\label{equ 3.2}
		Y_t=\eta+\int_t^Tg(s,P_s,Z_s, \dbP_{P_s},\dbP_{Q_s})ds-\int_t^TZ_sdW_s,\q t\in[0,T].
	\end{equation}
	We define the generator $g^{P,Q}$:
	$$g^{P,Q}(t,z)=g(t,P_t,z,\dbP_{P_t},\dbP_{Q_t}),\q (t,z)\in[0,T]\times \mathbb{R}^d.$$
	Under \autoref{ass 3.1}, we can easily verify that the generator $g^{P,Q}$ satisfies  $d\dbP\times dt$-a.e., for every $z,\bar{z}\in \dbR^d$,
	$$|g^{P,Q}(\omega,t,z)|\leq \theta_t(\omega)+\phi(|P_t|)+\frac{\gamma}{2}|z|^2
	+\phi_0\big(\|P_t\|_{L^2(\Om)}\big)+\gamma_0\|Q_t\|_{L^2(\Om)}^{1+\alpha},$$
	and
	$$|g^{P,Q}(\omega,t,z)-g^{P,Q}(\omega,t,\bar{z})|\leq
	\phi\big(|P_t|\vee\|P_t\|_{L^2(\Om)}\big)\cd
	\big(1+|z|+|\bar{z}|+2\|Q_t\|_{L^2(\Om)}\big)|z-\bar{z}|.$$
	Thus according to \autoref{le 7.1}, BSDE (\ref{equ 3.2}) admits a unique solution
	$(Y,Z)\in S_\dbF^\infty(0,T;\dbR)\ts \cZ^2_\dbF(0,T;\dbR^d)$.
	Moreover,  for any $t\in[0,T]$ and stopping time $\tau\in \sT[t,T]$, we have
	\begin{equation}\label{equ 3.3}
		\begin{aligned}
			|Y_t|&\leq \|\eta\|_\infty+\Bigl\|\int_0^T\theta_s\, ds\Bigr\|_\infty
			+(T-t)\widehat{\phi+\phi_0}^t
			+\gamma_0\|Q\|^{1+\alpha}_{\cZ^2_\dbF(t,T)}(T-t)^{\frac{1-\alpha}{2}}
		\end{aligned}
	\end{equation}
	and
	\begin{equation}\label{equ 3.4}
		\begin{aligned}
			\dbE_\tau\int_\tau^T|Z_s|^2ds
			\leq&\,  \frac{1}{\gamma^2}\exp\left(2\gamma\|\eta\|_\infty\right)
			+\frac{2}{\gamma}\exp\left(2\gamma\|Y\|_{S_\dbF^\infty(t,T)}\right)\\
			&\times \left(\ \Bigl\|\int_0^T\!\!\! \theta_s\, ds\Bigr\|_\infty
			+(T-t)\widehat{\phi+\phi_0}^t
			+\gamma_0\|Q\|^{1+\alpha}_{\cZ^2_\dbF(t,T)}(T-t)^{\frac{1-\alpha}{2}}\right).
		\end{aligned}
	\end{equation}
	Then, we define the mapping $\Theta$ on
	$S_\dbF^\infty(0,T;\dbR)\ts \cZ^2_\dbF(0,T;\dbR^d)$ as follows:
	$$\Theta(P,Q)\deq (Y,Z),\q  (P,Q)\in S_\dbF^\infty(0,T;\dbR)\ts \cZ^2_\dbF(0,T;\dbR^d).$$

	Next, we  find a suitable subset $\sB_\e(L_1,L_2)$ of $S_\dbF^\infty(0,T;\dbR)\ts \cZ^2_\dbF(0,T;\dbR^d)$ such that $\Theta$ is stable and contractive on it.
	For this, note  \rf{equ 3.3} and \rf{equ 3.4} and thanks to \autoref{ass 3.1},  we have
	\begin{equation}\label{equ 3.5}
		\begin{aligned}
			\|Y\|_{S^\infty_{\dbF}(t,T)}\leq&\ K_1+K_2
			+(T-t)\widehat{\phi+\phi_0}^t
			+\gamma_0\|Q\|^{1+\alpha}_{\cZ^2_\dbF(t,T)}(T-t)^{\frac{1-\alpha}{2}}
		\end{aligned}
	\end{equation}
	and
	\begin{equation}\label{equ 3.6}
		\begin{aligned}
			\|Z\|_{\cZ^2_\dbF(t,T)}^2
			\leq&\frac{1}{\gamma^2}\exp\big(2\gamma\|\eta\|_\infty\big)
			+\frac{2}{\gamma}\exp\big(2\gamma\|Y\|_{S_\dbF^\infty(t,T)}\big)\\
			&\times \left(\ K_2 +(T-t)\widehat{\phi+\phi_0}^t
			+\gamma_0\|Q\|^{1+\alpha}_{\cZ^2_\dbF(t,T)}(T-t)^{\frac{1-\alpha}{2}}\right).
		\end{aligned}
	\end{equation}
	Now, we define
	\begin{equation}\label{equ 3.6.2}
		L_1=2(K_1+K_2)\q\hbox{and}\q
		L_2=\frac{2}{\gamma^2}e^{2\gamma K_1}+\frac{4K_2}{\gamma}e^{2\gamma L_1}.
	\end{equation}
	By $\e_1$ and $\e_2$,  we denote the unique solutions to the following two equations:
	\begin{equation*}\label{equ 3.7}
		[\phi(L_1)+\phi_0(L_1)]x+\gamma_0L_2^{\frac{1+\alpha}{2}}x^{\frac{1-\alpha}{2}}=\frac{L_1}{2}
	\end{equation*}
	and
	\begin{equation}\label{equ 3.8}
		[2\phi(L_1)+2\phi_0(L_1)]y+2\gamma_0L_2^{\frac{1+\alpha}{2}}y^{\frac{1-\alpha}{2}}=\frac{\gamma L_2}{2}e^{-2\gamma L_1},
	\end{equation}
	respectively. Hence,  combining (\ref{equ 3.5})-(\ref{equ 3.8}), we see that if
	$$\|P\|_{S_\dbF^\infty(T-\e,T)}\leq L_1\q\hb{and}\q \|Q\|^2_{ \cZ^2_\dbF(T-\e,T)}\leq L_2,$$
	then
	$$\|Y\|_{S_\dbF^\infty(T-\e,T)}\leq L_1\q\hb{and}\q \|Z\|^2_{ \cZ^2_\dbF(T-\e,T)}\leq L_2, \q \forall \e\in(0,\e^*],$$
	where $\e^*=\min\{\e_1,\e_2\}>0$.  Since $\e^*$ depends only on $K_1,K_2, \a, \g$ and $\g_0$, we have 
	$$\Theta(P,Q)\in \sB_\e(L_1,L_2), \q \forall (P,Q)\in\sB_\e(L_1,L_2).$$
	Therefore, the mapping $\Theta$ is stable in $\sB_\e(L_1,L_2).$
	
	We now prove that the mapping $\Theta$ is a contraction in $\sB_\e(L_1,L_2).$
	Indeed, for any fixed $\e\in(0,\e^*]$, and any pairs $(P,Q)\in \sB_\e(L_1,L_2)$ and $(\bar P,\bar Q)\in \sB_\e(L_1,L_2)$, we set
	$$
	(Y,Z):=\Theta(P,Q)\quad\hb{and}\q (\bar{Y},\bar{Z}):=\Theta(\bar{P},\bar{Q}).
	$$
	In addition, denote for any $t\in[T-\e,T]$,
	$$\Delta Y_t:=Y_t-\bar{Y}_t,\q  \Delta Z_t:=Z_t-\bar{Z}_t,\q \Delta P_t:=P_t-\bar{P}_t, \q \Delta Q_t:=Q_t-\bar{Q}_t.$$
	Then, we have
	\begin{equation}\label{equ 3.9}
		\Delta Y_t=\int_t^T
		\big(I_{1,s}+I_{2,s}\big)ds-\int_t^T\Delta Z_sdW_s,\q  t\in[T-\e,T],
	\end{equation}
	where
	\begin{equation*}\label{equ 3.11}
		\begin{aligned}
			I_{1,s}&:= g(s,P_s,Z_s, \dbP_{P_s},\dbP_{Q_s})-g(s,P_s,\bar{Z}_s, \dbP_{P_s},\dbP_{Q_s}),\\
			I_{2,s}&:= g(s,P_s,\bar{Z}_s, \dbP_{P_s},\dbP_{Q_s})-g(s,\bar{P}_s,\bar{Z}_s, \dbP_{\bar{P}_s},\dbP_{\bar{Q}_s}),
		\end{aligned}
	\end{equation*}
	and thus
	\begin{equation*}\label{equ 3.10}
		I_{1,s}+I_{2,s}=g(s,P_s,Z_s, \dbP_{P_s},\dbP_{Q_s})-g(s,\bar{P}_s,\bar{Z}_s, \dbP_{\bar{P}_s},\dbP_{\bar{Q}_s}).
	\end{equation*}
	Note that the term $I_{1,s}$ can be written as $I_{1,s}=\Lambda_s(Z_s-\bar{Z}_s),$ where
	$$\Lambda_s:=\left\{\ba{ll}
	\frac{(g(s,P_s,Z_s, \dbP_{P_s},\dbP_{Q_s})
		-g(s,P_s,\bar{Z}_s, \dbP_{P_s},\dbP_{Q_s}))(Z_s-\bar{Z}_s)
	}{|Z_s-\bar{Z}_s|^2},&\quad \text{if}\ Z_s\neq\bar{Z}_s;\\
	0                ,&\quad \text{if}\ Z_s=\bar{Z}_s.
	\ea\right.$$
	Clearly,  the item (ii) of \autoref{ass 3.1} implies that
	$$|\Lambda_s|\leq \phi(|P_s|\vee \|P_s\|_{L^2(\Om)})(1+|Z_s|
	+|\bar{Z}_s|+2\|Q_s\|_{L^2(\Om)}),\q\forall s\in[t,T].$$
	Then, for
	\begin{equation}\label{equ girsanov}
		\widetilde{W}_s:=W_s-\int_0^s\Lambda_rdr\q\hb{and}\q
		d\dbQ:=\sE(\L\cd W)_0^Td\dbP,
	\end{equation}
	%
	%
	the probability measure $\dbQ$ is equivalent to $\mathbb{P}$ and the process  $\widetilde{W}$ is a Brownian motion under $\dbQ$.
	Furthermore, we can rewrite BSDE (\ref{equ 3.9})  as
	\begin{equation}\label{equ 3.12}
		\begin{aligned}
			\Delta Y_t+\int_t^T\Delta Z_sd\widetilde{W}_s =\int_t^TI_{2,s}ds,\q \forall t\in[T-\e,T],
		\end{aligned}
	\end{equation}
	which easily implies that
	$$
	|\Delta Y_t|^2+\dbE_t^\dbQ\int_t^T|\Delta Z_s|^2ds=\dbE_t^\dbQ\bigg[\Big(\int_t^TI_{2,s}ds\Big)^2\bigg].
	$$
	Since 
	$\cW_2(\dbP_{P_s},\d_{\{0\}})\leq \{\dbE[|P_s|^2]\}^\frac{1}{2}= \|P_s\|_{L^2(\Om)}$,
	we have the following estimates  from the item (ii) of \autoref{ass 3.1},
	\begin{align*}
		|I_{2,s}|\leq&\
		\phi\Big(|P_s|\vee|\bar{P}_s|\vee\cW_2(\dbP_{P_s},\d_{\{0\}})\vee
		\cW_2(\dbP_{\bar{P}_s},\d_{\{0\}})\Big)\\
		&\  \times\bigg[\Big(1+2|\bar{Z}_s|+\cW_2(\dbP_{Q_s},\d_{\{0\}})
		+\cW_2(\dbP_{\bar{Q}_s},\d_{\{0\}})\Big)
		\left(|\Delta P_s|+\cW_2(\dbP_{P_s},\dbP_{\bar{P}_s})\right)\\
		&\q\ +\Big(1+\cW_2(\dbP_{Q_s},\d_{\{0\}})^{\alpha}
		+\cW_2(\dbP_{\bar{Q}_s},\d_{\{0\}})^{\alpha}\Big)
		\cW_2(\dbP_{Q_s},\dbP_{\bar{Q}_s})
		\bigg]\\
		\leq&\
		\phi\Big(|P_s|\vee|\bar{P}_s|\vee\|P_s\|_{L^2(\Om)}\vee \|\bar{P}_s\|_{L^2(\Om)}\Big)\\
		&\quad\times
		\bigg[\Big(1+2|\bar{Z}_s|+
		\|Q_s\|_{L^2(\Om)}+ \|\bar{Q}_s\|_{L^2(\Om)}\Big) (|\Delta P_s|+ \|\Delta P_s\|_{L^2(\Om)})\\
		&\q\ +\Big(1+\|Q_s\|^{\alpha}_{L^2(\Om)}+ \|\bar{Q}_s\|^{\alpha}_{L^2(\Om)}\Big)
		\|\Delta Q_s\|_{L^2(\Om)}\bigg]\\
		\leq&\
		\phi\Big(\|P\|_{S_\dbF^\infty(t,T)}\vee \|\bar{P}\|_{S_\dbF^\infty(t,T)}\)
		\bigg[2\Big(1+2|\bar{Z}_s|+
		\|Q_s\|_{L^2(\Om)}+ \|\bar{Q}_s\|_{L^2(\Om)}\Big)\\
		&\quad \times \|\Delta P\|_{S_\dbF^\infty(t,T)}
		+\Big(1+\|Q_s\|^{\alpha}_{L^2(\Om)}+ \|\bar{Q}_s\|^{\alpha}_{L^2(\Om)}\Big)
		\|\Delta Q_s\|_{L^2(\Om)}\bigg]\\
		\leq& \phi\big(L_1\big)
		\bigg[2\Big(1+2|\bar{Z}_s|+\|Q_s\|_{L^2(\Om)}+ \|\bar{Q}_s\|_{L^2(\Om)}\Big ) \|\Delta P\|_{S_\dbF^\infty(t,T)} \\
		&\quad\quad
		+\Big(1+\|Q_s\|_{L^2(\Om)}^{\alpha}+\|\bar{Q}_s\|_{L^2(\Om)}^{\alpha}\Big)
		\|\Delta Q_s\|_{L^2(\Om)}\bigg].
	\end{align*}
	Thus we have
	%
	\begin{align}\label{equ 3.14}
		&|\Delta Y_t|^2+\dbE_t^\dbQ\Big[\int_t^T|\Delta Z_s|^2ds\Big] \nn \\
		&\leq\dbE_t^\dbQ\bigg\{\bigg[\int_t^T
		2\phi\big(L_1\big)
		\Big(1+2|\bar{Z}_s|+\|Q_s\|_{L^2(\Om)}+ \|\bar{Q}_s\|_{L^2(\Om)}\Big ) \|\Delta P\|_{S_\dbF^\infty(t,T)}\nn \\
		&\quad+\phi\big(L_1\big)\Big(1+\|Q_s\|_{L^2(\Om)}^{\alpha}+\|\bar{Q}_s\|_{L^2(\Om)}^{\alpha}\Big)
		\|\Delta Q_s\|_{L^2(\Om)}
		ds\bigg]^2\bigg\}\\
		&\leq
		8\phi\big(L_1\big)^2
		\dbE_t^\dbQ\bigg[
		\Big(\int_t^T(1+2|\bar{Z}_s|+\|Q_s\|_{L^2(\Om)}+\|\bar{Q}_s\|_{L^2(\Om)})\|\Delta P\|_{ S_\dbF^\infty(t,T)}ds\Big)^2
		\bigg]\nn \\
		&\quad+2\phi\big(L_1\big)^2
		\Big(\int_t^T(1+\|Q_s\|_{L^2(\Om)}^{\alpha}+\|\bar{Q}_s\|_{L^2(\Om)}^{\alpha})\|\Delta Q_s\|_{L^2(\Om)}ds\Big)^2. \nn
	\end{align}
	%
	We now estimate the first term in the right hand side of the last inequality. By using H\"{o}lder's inequality and \autoref{2.1.11}, and noting \rf{4.3.1}, we have for  $t\in[T-\e,T]$,
	%
	\begin{align*}
		&8\phi\big(L_1\big)^2
		\dbE_t^\dbQ\bigg[
		\Big(\int_t^T(1+2|\bar{Z}_s|+\|Q_s\|_{L^2(\Om)}+\|\bar{Q}_s\|_{L^2(\Om)})\|\Delta P\|_{ S_\dbF^\infty(t,T)}ds\Big)^2
		\bigg]\\
		&\leq
		32\e\phi\big(L_1\big)^2\|\Delta P\|^2_{ S_\dbF^\infty(t,T)}
		\Big(T+4\dbE_t^\dbQ\Big[\int_t^T|\bar{Z}_s|^2ds\Big]
		+\int_t^T\dbE^{\dbP}|Q_s|^2ds+\int_t^T\dbE^{\dbP}|\bar{Q}_s|^2ds\Big)\\
		&\leq
		32\e\phi\big(L_1\big)^2\|\Delta P\|^2_{S_\dbF^\infty(t,T)}
		\Big(
		T+4\|\bar{Z}\cd \widetilde W\|_{BMO(\dbQ)}^2+\|Q\cd W\|_{BMO(\dbP)}^2+\|\bar{Q}\cd W\|_{BMO(\dbP)}^2\Big)\\
		&\leq
		32\e\phi\big(L_1\big)^2
		\Big(T+4c_2L_2+2L_2\Big)\|\Delta P\|^2_{S_\dbF^\infty(t,T)}.
	\end{align*}
	Since $\a\in[0,1)$,  using H\"{o}lder's inequality and noting \rf{4.3.1} again, we see that for $t\in[T-\e,T]$,
	$$
	\begin{aligned}
		\int_t^T\|Q_s\|_{L^2 (\Om)}^{2\alpha}ds
		&=\int_t^T\big\{\dbE^{\dbP}|Q_s|^2\big\}^{\alpha}ds
		\leq \Big(\int_t^T \dbE^{\dbP}|Q_s|^2ds  \Big)^\alpha(T-t)^{1-\alpha} \\
		&\leq \|Q\|^{2\alpha}_{\cZ^2_\dbF(t,T)}(T-t)^{1-\alpha}
		\leq L_2^{\alpha}(T-t)^{1-\alpha}
		\leq L_2(T-t)^{1-\alpha}.
	\end{aligned}
	$$
	In the above,  $L_2$ can be chosen to be greater than 1, by  a careful choice of  $\g$, $K_1$ or $K_2$ in \rf{equ 3.6.2}.
	Consequently, using H\"{o}lder's inequality, we estimate  the second term in the right hand side of \rf{equ 3.14} as follows:
	\begin{equation}\label{equ 3.16}
		\begin{aligned}
			&2\phi(L_1)^2
			\Big(\int_t^T(1+\|Q_s\|_{L^2(\Om)}^{\alpha}+\|\bar{Q}_s\|_{L^2(\Om)}^{\alpha})\|\Delta Q_s\|_{L^2(\Om)}ds\Big)^2\\
			&\leq2\phi(L_1)^2\int_t^T||\Delta Q_s||^2_{L^2(\Omega)}ds\cdot
			\int_t^T\Big(1+\|Q_s\|_{L^2(\Om)}^{\alpha}+\|\bar{Q}_s\|_{L^2(\Om)}^{\alpha}\Big)^2ds\\
			&\leq6\phi(L_1)^2||\Delta Q||^2_{\cZ^2_\dbF(t,T)}
			\big[T-t+2L_2(T-t)^{1-\alpha}\big]\\
			&\leq
			6\e^{1-\alpha}\phi(L_1)^2(T^\alpha+2L_2)\|\Delta Q\|^2_{\cZ^2_\dbF(t,T)},
		\end{aligned}
	\end{equation}
	where in the last inequality we have used the inequality $\e\les \e^{1-\a}T^\a$.
	Now from (\ref{equ 3.14}) and (\ref{equ 3.16}),  we have
	\begin{align}\label{equ 3.16-1}
		|\Delta Y_t|^2+\dbE_t^\dbQ\int_t^T|\Delta Z_s|^2ds
		\leq&\
		32\e\phi\big(L_1\big)^2
		\big(T+4c_2L_2+2L_2\big)\|\Delta P\|^2_{S_\dbF^\infty(t,T)} \\
		&\
		+6\e^{1-\alpha}\phi\big(L_1\big)^2(T^\alpha+2L_2)
		\|\Delta Q\|^2_{\cZ^2_\dbF(t,T)}. \nn
	\end{align}
	Moreover, thanks to \autoref{2.1.11}, and noting that $\a\in[0,1)$ and $t\in[T-\e,T]$, we have
	\begin{equation*}\label{equ 3.18}
		\begin{aligned}
			\|\Delta Y\|^2_{S_\dbF^\infty(T-\e,T)}+c_1^2\|\Delta Z\|^2_{\cZ^2_\dbF(T-\e,T)}
			&\leq
			32\e\phi\big(L_1\big)^2
			\big(T+4c_2L_2+2L_2\big)\|\Delta P\|^2_{S_\dbF^\infty(T-\e,T)}\\
			&\q +6\e^{1-\alpha}\phi\big(L_1\big)^2(T^\alpha+2L_2)
			\|\Delta Q\|^2_{\cZ^2_\dbF(T-\e,T)}.
		\end{aligned}
	\end{equation*}
	From this, we know that
	$\Theta$ is a contraction in $\sB_\e(L_1,L_2)$ for sufficiently small $\e>0$. The proof is complete.
\end{proof}

\subsection{Global solution}\label{Sec3.2}

Based on the previous result concerning the local solution of mean-field BSDE \rf{MFBSDE}, in this subsection, we are going to study the global adapted solution of mean-field BSDE \rf{MFBSDE} in two situations: (i) the case that the generator is bounded in $\mu_2$; (ii) the case that  the generator is unbounded in $\mu_2$.  We point out that these two situations do not cover each other.

\subsubsection{Global solution:  the generator $g$ has a  bounded growth  in $\mu_2$} 

Consider  the following assumption,  which is  slightly stronger than \autoref{ass 3.1}.

\bas{ass 3.2}\rm  The terminal value $\eta:\Om\rightarrow\dbR$ and the generator $g: \Omega\times[0,T]\times\dbR\times\dbR^d\times \cP_2(\dbR)\ts\cP_2(\dbR^{d})\rightarrow \dbR$ satisfy the following conditions: there exists a positive constant $K$ such that
\begin{enumerate}[~~\,\rm (i)]
	\item $d\dbP\times dt$-a.e. $(\o,t)\in\Om\ts[0,T]$, for every $(y, z, \mu_1, \mu_2)\in\dbR\times\dbR^d\times \cP_2(\dbR)\ts \cP_2(\dbR^d)$,
	\begin{equation}\label{22.5.18.1}
		|g(w,t,y,z,\mu_1,\mu_2)|\les \theta_t(\omega)+K|y|+\frac{\gamma}{2}|z|^2+K\cW_2(\mu_1,\d_{\{0\}}).
	\end{equation}
	\item $d\dbP\times dt$-a.e. $(\o,t)\in\Om\ts[0,T]$, for every $(y, z; \by,\bz;\mu_1, \mu_2; \bar{\m}_1, \bar{\m}_2)\in \left(\dbR\times\dbR^d\right)^2\times \left(\cP_2(\dbR)\ts\cP_2(\dbR^{d})\right)^2$,
	$$\ba{ll}
	& |g(\o,t,y,z,\m_1,\mu_2)-g(\o,t,\bar y,\bz,\bar\mu_1,\bar\mu_2)| \les K\left(|y-\bar{y}|+\cW_2(\mu_1,\bar{\mu}_1)\right)\\
	&\q + \phi\big(|y|\vee|\bar y|\vee\cW_2(\mu_1,\d_{\{0\}})\vee\cW_2(\bar \mu_1,\d_{\{0\}})\big)  \big[(1+|z|+|\bar{z}|)|z-\bar{z}|+\cW_2(\mu_2,\bar{\mu}_2)\big].
	\ea$$
	\item  For two positive constants $K_1$ and $K_3$,
	$$\|\eta\|_\infty\les K_1\q\hb{and}\q \bigg\|\int_0^T|\theta_t(\omega)|^2dt\bigg\|_\infty\les K_3.$$
\end{enumerate}\eas

\begin{remark}\rm
	On the one hand,  by using H\"{o}lder's inequality,  we have
	\begin{equation*}\label{22.9.15}
		\bigg\|\int_0^T|\theta_t(\omega)|dt\bigg\|_\infty
		\les \bigg\| \sqrt{T}  \bigg(\int_0^T|\theta_t(\omega)|^2dt\bigg)^{\frac{1}{2}}\bigg\|_\infty
		\les \sqrt{TK_3}=K_2.
	\end{equation*}
	On the other hand,  \autoref{ass 3.2} is weaker than the conditions $(\sA2)$ and $(\sA3)$ of Hibon, Hu, and Tang \cite{Hibon-Hu-Tang-17}. For example, the following generator $g$:
	\begin{align*}
		g(\omega,t,y,z,\mu_1,\m_2)=|y|+|z|+|z|^2+\mathcal{W}_2(\mu_1,\delta_{\{0\}})\sin(\mathcal{W}_2(\mu_2,\delta_{\{0\}})),
	\end{align*}
	for every $(t, y, z)\in[0,T]\times \dbR\times \dbR^d, \mu_1\in\mathcal{P}_2(\mathbb{R}),\mu_2\in\mathcal{P}_2(\mathbb{R}^{d})$
	satisfies \autoref{ass 3.2}, but it does not satisfy the conditions of \cite{Hibon-Hu-Tang-17}.
\end{remark}

In this subsection, the main result is the following theorem on the global adapted solution of mean-field BSDE \rf{MFBSDE}.
\begin{theorem}\label{th 3.5}\it
	Under \autoref{ass 3.2}, on the whole interval $[0,T]$, BSDE (\ref{MFBSDE}) possesses a unique global solution $(Y,Z)\in S_\dbF^\infty(0,T;\dbR)\ts \cZ^2_\dbF(0,T;\dbR^d)$.
	Furthermore, there exist two positive constants $M_1$ and $M_2$, that depend only on 
	$(K,K_1,K_3,\g,T)$, such that
	\begin{equation}\label{4.7.0}
		\|Y\|_{S_\dbF^\infty(0,T)}\leq  M_1  \q~
		\text{and}\q~  \|Z\|^2_{ \cZ^2_\dbF(0,T)}\leq M_2.
	\end{equation}
\end{theorem}

\begin{proof}
	We will initially analyze the solvability of  BSDE \eqref{MFBSDE} on the interval $[T-\kappa_{\l},T]$, where $\kappa_{\l}$ is a positive constant that will be determined later.
	For this purpose, we define
	\begin{equation*}\label{22.5.17.1}
		\overline{C}=K_1^2+K_3+K^2+2K+2.
	\end{equation*}
	Moreover, we let the function $\Gamma$ be the unique solution of the following ordinary differential equation
	\begin{equation*}
		\Gamma(t)=\overline{C}+\overline{C}\int_t^T\Gamma(s)ds,\q  t\in[0,T].
	\end{equation*}
	In addition, we define
	$$\l=\sup\limits_{t\in[0,T]}\Gamma(t)=\Gamma(0).$$
	Then, the function $\Gamma$ is continuous and decreasing.

	Now, due to that $\|\eta\|^2_\infty\leq \overline{C}=\G(T),$ \autoref{th 3.2} implies that there exists a positive constant $\kappa_{\l}$, depending only on $\l$, such that BSDE (\ref{MFBSDE}) possesses a unique local solution
	$(Y^{(1)},Z^{(1)})$ on the interval $[T-\kappa_{\l},T]$.
	Next, we show   that  $$\|Y^{(1)}\|^2_{S^\i_{\dbF}(t,T)}\les \G(t),\q t\in[T-\kappa_{\l},T].$$ 
	In fact, for $t\in [T-\kappa_{\l},T]$,
	\begin{equation*}
		\begin{aligned}
			Y^{(1)}_t &=\eta+\int_t^T\(g(s,Y^{(1)}_s,Z^{(1)}_s,\dbP_{Y^{(1)}_s},\dbP_{Z^{(1)}_s})-g(s,0,Z^{(1)}_s,0,\dbP_{Z^{(1)}_s})\\
			&\q + g(s,0,Z^{(1)}_s,0,\dbP_{Z^{(1)}_s})-g(s,0,0,0,\dbP_{Z^{(1)}_s})+g(s,0,0,0,\dbP_{Z^{(1)}_s})\)ds-\int_{t}^{T}Z^{(1)}_sdW_s.
		\end{aligned}
	\end{equation*}
	Note that,  by assumptions, 
	\begin{equation*} 
		\left\{
		\begin{aligned}
			&g(s,0,Z^{(1)}_s,0,\dbP_{Z^{(1)}_s})-g(s,0,0,0,\dbP_{Z^{(1)}_s})=\Lambda_sZ^{(1)}_s;\\
			&|\Lambda_s|\leq \phi(0)(1+|Z^{(1)}_s|).
		\end{aligned}
		\right.
	\end{equation*}
	Then the process 
	\begin{equation*}\label{22.5.13.411}
		\widetilde{W}_t\deq W_t-\int_0^t\Lambda_sds, \quad t\in [0,T],
	\end{equation*}
	is a Brownian motion with respect to the equivalent probability measure $\widetilde{\dbP}$ defined by
	$$d\widetilde{\mathbb{P}} \deq\sE(\Lambda \cdot W)_0^Td\mathbb{P},$$
	with respect to which the conditional expectation is denoted by $\widetilde{\dbE}_r[\cd]$.
	Now, applying It\^{o}'s formula to $|Y^{(1)}|^2$ and using \autoref{ass 3.2},   we have that for any $r\in[T-\kappa_\l,t]$,
	\begin{align*}
		&\widetilde{\dbE}_r|Y^{(1)}_t|^2+\widetilde{\dbE}_r\int_t^T|Z^{(1)}_s|^2ds\\
		&=\widetilde{\dbE}_r|\eta|^2
		+\widetilde{\dbE}_r\int_t^T2Y^{(1)}_s \(g(s,Y^{(1)}_s,Z^{(1)}_s,\dbP_{Y^{(1)}_s},\dbP_{Z^{(1)}_s})-g(s,0,Z^{(1)}_s,0,\dbP_{Z^{(1)}_s})\\
		&\q    +g(s,0,0,0,\dbP_{Z^{(1)}_s})\)      ds\\
		&
		\leq \widetilde{\dbE}_r|\eta|^2
		+\widetilde{\dbE}_r\int_t^T2|Y^{(1)}_s| \(K|Y^{(1)}_s|+KW_2(\dbP_{Y^{(1)}_s},\d_{\{0\}})+\th_s(\o)    \)  ds\\
		&\leq K_1^2+K_3+\widetilde{\dbE}_r\int_t^T((K^2+2K+1)|Y^{(1)}_s|^2 +\dbE|Y^{(1)}_s|^2)ds\\
		&\leq K_1^2+K_3+ \int_t^T\((K^2+2K+2) \|Y^{(1)}\|^2_{S^\i_{\dbF}(s,T)}\)ds,
	\end{align*} 
	which implies that
	\begin{align*}
		&\|Y^{(1)}\|^2_{S^\i_{\dbF}(t,T)}\leq K_1^2+K_3+ \int_t^T(K^2+2K+2) \|Y^{(1)}\|^2_{S^\i_{\dbF}(s,T)}ds.
	\end{align*} 
	Hence,    $\|Y^{(1)}\|^2_{S^\i_{\dbF}(t,T)}\les \G(t)$ for $ t\in[T-\kappa_{\l},T].$

		In a similar way, for $t\in [T-2\kappa_{\l},T-\kappa_{\l})$, again,  \autoref{th 3.2} implies that   BSDE (\ref{MFBSDE}) admits a unique solution $(Y^{(2)},Z^{(2)})\in S_\dbF^\infty(T-2\kappa_{\l},T-\kappa_{\l};\dbR)\ts \cZ^2_\dbF(T-2\kappa_{\l},T-\kappa_{\l};\dbR^d)$.
		We define for $t\in [T-2\kappa_{\l},T], $
		$$
		\left(\overline{Y}_t, \overline{Z}_t\right):=\left(Y^{(1)}_t, Z^{(1)}_t\right)\, \chi_{[T-\kappa_{\l}, T]}(t)+\left(Y^{(2)}_t, Z^{(2)}_t\right) \chi_{[T-2\kappa_{\l}, T-\kappa_{\l})}(t).$$
		Here, $\chi_A$ stands for the indicator function of the set $A$. 
		Clearly, $(\overline{Y},\overline{Z})$  is the unique solution to  
		the mean-field BSDE (\ref{MFBSDE}) on the interval $[T-2\kappa_{\l},T]$. Besides,  proceeding identically to  the above, we have
		$$\|Y\|^2_{S^\i_{\dbF}(t,T)}\les \G(t),\q~ t\in[T-2\kappa_{\l},T].$$
		Repeating the preceding process allows us to demonstrate that the mean-field BSDE (\ref{MFBSDE}) has a unique solution $(Y,Z)$  on $[0,T]$, where $Y$ is uniformly bounded by  $\l=\G(0)$. In other words, there
		exists a positive constant $M_1$ depending only on $(K,K_1,K_3,T)$ such that
		\begin{equation*}\label{4.7.1.2}
			\|Y\|_{S_\dbF^\infty(0,T)}\leq  M_1.
		\end{equation*}

		We now show that $Z\cdot W$ is a $BMO(\dbP)$-martingale.
		For this, we define
		\begin{equation}\label{22.8.29.3}
			\Phi(x):=\frac{1}{\gamma^2}\big[\exp(\gamma|x|)-\gamma|x|-1\big],\quad x\in \mathbb{R}, 
		\end{equation}
		where $\g$ is a positive constant given in  \rf{22.5.18.1}. We have
		\begin{equation*}\label{22.5.18.2}
			\Phi'(x)=\frac{1}{\gamma}\big[\exp(\gamma |x|)-1\big]\sgn(x), \quad
			\Phi''(x)=\exp(\gamma|x|), \quad  \Phi''(x)-\gamma |\Phi'(x)|=1.
		\end{equation*}
		Using It\^{o}'s formula and \autoref{ass 3.2},  we have
		\begin{align*}
			\Phi(Y_t)=&\ \Phi(Y_T)
			+\int_{t}^{T}\Phi'(Y_s)g\big(s,Y_s,Z_s, \mathbb{P}_{Y_s}, \mathbb{P}_{Z_s}\big)ds\\   &-\int_{t}^{T}\Phi'(Y_s)Z_sdW_s-\frac{1}{2}\int_{t}^{T}\Phi''(Y_s)|Z_s|^2ds\\
			\les &\ \Phi(Y_T)+\int_{t}^{T}|\Phi'(Y_s)|
			\big\{\theta_s+K\big[|Y_s|+\|Y_s\|_{L^2(\Omega)}\big]\big\}ds
			-\int_{t}^{T}\Phi'(Y_s)Z_sdW_s\\
			& +\frac{1}{2}\int_{t}^{T}\big[\gamma|\Phi'(Y_s)| -\Phi''(Y_s)\big]|Z_s|^2ds.
		\end{align*}
		Then, from \autoref{ass 3.2}, we have
		\begin{equation*}
			\begin{aligned}
				\Phi(Y_t)+\frac{1}{2}\mathbb{E}_t\int_t^T|Z_s|^2ds
				&\leq \Phi(\|\eta\|_\infty)
				+\dbE_t\int_{t}^{T}|\Phi'(Y_s)|\big\{\theta_s+K\big[|Y_s|+\|Y_s\|_{L^2(\Omega)}\big]\big\}ds\\
				&\leq \Phi(K_1)+|\Phi'(M_1)|\dbE_t\int_{t}^{T}(\theta_s+2KM_1)ds\\
				& \leq \Phi(K_1)+|\Phi'(M_1)|\big[\sqrt{K_3 T}+ 2KM_1 T\big].\\
			\end{aligned}
		\end{equation*}
		In other words,
		\begin{equation*}
			\mathbb{E}_t\int_t^T|Z_s|^2ds\les
			2\Phi(K_1)+2|\Phi'(M_1)|\big[\sqrt{K_3 T}+ 2KM_1 T\big].
		\end{equation*}
		Therefore, we have
		\begin{equation*}
			\|Z\|^2_{ \cZ^2_\dbF(0,T)}
			=\|Z\cdot W\|^2_{BMO(\mathbb{P})}\leq2\Phi(K_1)+2|\Phi'(M_1)|\big[\sqrt{K_3 T}+ 2KM_1 T\big]\deq M_2.
		\end{equation*}

		Finally, we prove the uniqueness. Let $(Y,Z)$ and $(\bar{Y},\bar{Z})$
		be two adapted solutions of BSDE (\ref{MFBSDE}).
		We denote
		$$\Delta Y:=Y-\bar{Y},\quad \Delta Z:=Z-\bar{Z}.$$
		Then for arbitrary $\varepsilon>0$ and  $t\in[T-\e,T]$,  we have
		\begin{equation*}
			\begin{aligned}
				\Delta Y_t+\int_t^T\Delta Z_sd\widetilde{W}_s =\int_t^T\[g(s,Y_s,\bar{Z}_s, \dbP_{Y_s}, \dbP_{Z_s})-g(s,\bar{Y}_s,
				\bar{Z}_s, \dbP_{\bar{Y}_s}, \dbP_{\bar{Z}_s})\]ds,
			\end{aligned}
		\end{equation*}
		where $\widetilde{W}$ is given in (\ref{equ girsanov}).
		Now,  similar to (\ref{equ 3.12})-(\ref{equ 3.16-1}), we have that  for $t\in [T-\varepsilon,T],$
		\begin{equation*}
			\begin{aligned}
				&\|\Delta Y\|^2_{S_\dbF^\infty(t,T)}+c_1^2\|\Delta Z\|^2_{\cZ^2_\dbF(t,T)}\\
				&\leq
				32\e\phi\big(M_1\big)^2
				\Big(T+4c_2M_2+2M_2 \Big)\|\Delta Y\|^2_{S_\dbF^\infty(t,T)}\\
				&\q +6\e^{1-\alpha}\phi\big(M_1 \big)^2\|\Delta Z\|^2_{\cZ^2_\dbF(t,T)}(T^\alpha+2M_2),
			\end{aligned}
		\end{equation*}
		where $c_1$ and $c_2$ are given in \autoref{2.1.11}.
		Clearly,  $Y=\overline{Y}$ and $Z=\overline{Z}$ on the interval $[T-\varepsilon,T]$ if $\varepsilon$ is small enough. Repeating this argument in  a finite number of steps, the uniqueness is obtained.
	\end{proof}

\begin{remark}\rm
	Hibon, Hu, and Tang \cite{Hibon-Hu-Tang-17} give  the existence and uniqueness result on quadratic  mean-field BSDEs when the generators $g$ depend on the marginal expectation of the unknown processes $(Y, Z)$. In comparison with their work, our conditions are weaker, our results are more general, and our method for establishing the existence and uniqueness result is more powerful.
\end{remark}

	
	\subsubsection{Global solution:  the generator $g$  has an unbounded growth in $\mu_2$} 
	
	In this subsection, we study the global solution to  BSDE (\ref{MFBSDE}), where the generator $g$ has an unbounded growth in $\mu_2$. We consider the following three different situations: 
	(i)   the generator $g$ has a strictly quadratic growth with respect to $z$ and  a sub-quadratic growth with respect to $\mu_2$;
	(ii)  the generator $g$ has a  quadratic  and  sub-quadratic growth with respect to $z$ and $\mu_2$, respectively,  with  the growth coefficient   of $g$ in  $\mu_2$  being sufficiently small;
	(iii)  the generator $g(t,y,z, \mu_1,\mu_2)=g_1(t,z)+g_2(t,\mu_2)$, where  both functionals  $g_1$ and $g_2$ have a quadratic growth with respect to their last arguments.

	Recall  that $\theta:\Om\ts[0,T]\rightarrow\dbR^+$ is an $\sF_t$-progressively measurable process,  and $\b$, $\b_0$, $\g$, $\g_0$, $\tilde{\g}$, and $\alpha\in[0,1)$ are all positive constants.

	\bas{ass 3.2-11}\rm
	Assume that the generator 
	$$g: \Omega\times[0,T]\times\dbR\times\dbR^d\times \cP_2(\dbR)\ts \cP_2(\dbR^d)\to \dbR$$ satisfy the following conditions: for $d\dbP\times dt$-almost all $(\o,t)\in\Om\ts[0,T]$,  we have 
	\begin{enumerate}[~~\,\rm (i)]
		\item the inequality 
		$$\hb{{\rm sgn}}(y)g(\omega,t,y,z,\mu_1,\mu_2)\leq \theta_t(\omega)+\beta|y|+\frac{\gamma}{2}|z|^2+\beta_0\cW_2(\mu_1,\delta_{\{0\}})
		+\gamma_0\cW_2(\mu_2,\delta_{\{0\}})^{1+\alpha}$$
		and 
		\item  either inequality 
		$$
		g(\omega,t,y,z,\mu_1,\mu_2)\leq
		-\frac{\tilde{\gamma}}{2}|z|^2+\theta_t(\omega)
		+\beta|y|+\beta_0  \cW_2(\mu_1,\delta_{\{0\}})
		+\gamma_0  \cW_2(\mu_2,\delta_{\{0\}})^{1+\alpha}
		$$
		or
		\begin{equation}\label{4.7.4}
			g(\omega,t,y,z,\mu_1,\mu_2)\geq \frac{\tilde{\gamma}}{2}|z|^2-\theta_t(\omega)
			-\beta|y|-\beta_0   \cW_2(\mu_1,\delta_{\{0\}})
			-\gamma_0   \cW_2(\mu_2,\delta_{\{0\}})^{1+\alpha}
		\end{equation}
	\end{enumerate}\eas
	for any $(y, z, \mu_1, \mu_2)\in\dbR\times\dbR^d\times \cP_2(\dbR)\ts \cP_2(\dbR^d)$.

	\begin{remark}\rm
		In \autoref{ass 3.2-11}, the condition (i) is some kind of one-sided linear growth condition of the generator $g$ with respect to the variable $y$, and the condition (ii) can be regarded as some kind of a strictly quadratic growth condition of the generator $g$ with respect to the variable $z$.
		A generator $g$ that satisfies Assumptions~\ref{ass 3.1} and \ref{ass 3.2-11} can still have a general growth in the variable $y$.  In addition, there are generators that satisfy all these conditions, but fail to satisfy the conditions in the work of Hibon, Hu, and Tang \cite{Hibon-Hu-Tang-17}.  See for instance, 
		\begin{align*} g(\omega,t,y,z,\mu_1,\mu_2): =-|z|^2+\cW_2(\mu_2, \d_{\{0\}})^\frac{5}{4}+e^{-2|y|}\cW_2(\mu_1, \d_{\{0\}})+y
		\end{align*}
		for any $(t, y, z, \mu_1, \mu_2)\in [0,T]\times \dbR\times\dbR^d\times \cP_2(\dbR)\ts \cP_2(\dbR^d)$. 
	\end{remark}

	We define the following constants for subsequent exposition.
	\begin{align}
		\ns\ds &L_3=\frac{(1-\alpha)\tilde{\gamma}\e_0}{4}\Big(\frac{1+\alpha}{2}
		\Big)^{\frac{1+\alpha}{1-\alpha}}
		\Big(\frac{2\gamma_0}{\tilde{\gamma}\e_0}
		\Big)^{\frac{2}{1-\alpha}}, \label{equ 3.34-1}\\
		\ns\ds& L_4=\frac{(1-\alpha)\tilde{\gamma}\e_0}{8}\Big(\frac{1+\alpha}{2}\Big)^{\frac{1+\alpha}{1-\alpha}}
		\Big(\frac{4\gamma_0}{\tilde{\gamma}}  \Big)^{\frac{2}{1-\alpha}}, \label{equ 3.34-2} \\
		\ns\ds &L_5=\[2(K_1+K_2)+2L_3T+4\e_0K_2+4L_4T\]e^{2(\beta+\beta_0)T}, \label{22.5.4.1}\\
		\ns\ds &L_6=\frac{2+(\beta+\beta_0)T}{\tilde{\gamma}}\Big(4L_5
		+16L_4T\Big)+\frac{4K_2}{\tilde{\gamma}}. \label{equ 3.34}
	\end{align}

	\begin{theorem}\label{th 3.5-1}\it 
		Under Assumptions~\ref{ass 3.1} and \ref{ass 3.2-11}, on the whole interval $[0,T]$, BSDE (\ref{MFBSDE}) admits a unique global solution $(Y,Z)\in S_\dbF^\infty(0,T;\dbR)\ts \cZ^2_\dbF(0,T;\dbR^d)$.
		Furthermore, there exist two positive constants $M_1$ and $M_2$, which depend only on the parameters $(K_1,K_2,T,\alpha, \tilde{\gamma},\gamma_0,\beta,\beta_0)$, such that
		\begin{align}\label{4.7.1}
			\|Y\|_{S_\dbF^\infty(0,T)}\leq  M_1  \q~
			\text{and}\q~  \|Z\|^2_{ \cZ^2_\dbF(0,T)}\leq M_2.
		\end{align} 
	\end{theorem}

	\begin{proof}
		
		The proof is divided into two steps. Firstly, we establish the estimate \rf{4.7.1}  on a small interval $[T-\kappa,T]$, where $\kappa\in(0,T]$ is a given constant. Secondly, we demonstrate the existence and uniqueness of global adapted solutions.

		{\it {Step 1.}} \label{Step1}  We show that for some given $\kappa\in(0,T]$, the solution $(Y,Z)$ to BSDE (\ref{MFBSDE}) on $[T-\kappa,T]$ satisfies
		\begin{equation}\label{equ 3.20}
			\|Y\|_{S_\dbF^\infty(T-\kappa,T)}\leq L_5\q~ \text{and}\
			\q~ \|Z\|^2_{ \cZ^2_\dbF(T-\kappa,T)}\leq L_6.
		\end{equation}
		For this, we let  $L$ be the local time of the process $Y$ at time $0$. Denote
		\begin{equation*}\label{3.20-1}
			\Psi(t,x)\deq\exp\Big\{ \gamma x+\gamma\int_0^t\(\theta_s+\beta|Y_s|+\beta_0\|Y_s\|_{L^2(\Om)}
			+\gamma_0\|Z_s\|^{1+\alpha}_{L^2(\Om)}\)ds\Big\}.\end{equation*}
		Applying It\^{o}-Tanaka's formula to $\Psi(t,|Y_t|)$, we have
		\begin{equation*}
			\begin{aligned}
				\ds d\Psi(t,|Y_t|)=&\ \Psi_t(t,|Y_t|)dt+\frac{1}{2}\Psi_{xx}(t,|Y_t|)|Z_t|^2dt\\
				\ns\ds &+\Psi_x(t,|Y_t|)\Big[-\sgn(Y_t)g(t,Y_t,Z_t,\mathbb{P}_{Y_t},\mathbb{P}_{Z_t})dt+\sgn(Y_t)Z_tdW_t+dL_t\Big].
			\end{aligned}
		\end{equation*}
		Notice that
		$$ \Psi_t(t,x)=\gamma\Psi(t,x)\Big(\theta_t+\beta|Y_t|+\beta_0\|Y_t\|_{L^2(\Om)}
		+\gamma_0\|Z_t\|^{1+\alpha}_{L^2(\Om)}\Big),$$
		and
		$$\Psi_x(t,x)=\gamma\Psi(t,x)\q\hb{and}\q\Psi_{xx}(t,x)=\gamma^2\Psi(t,x).$$
		Then from the condition (i) of \autoref{ass 3.2-11}, we have
		\begin{equation*}
			\begin{aligned}
				\ds d\Psi(t,|Y_t|)&=\gamma\Psi(t,|Y_t|)
				\Big[
				-\sgn(Y_t)g(t,Y_t,Z_t,\mathbb{P}_{Y_t},\mathbb{P}_{Z_t})+
				\theta_t+\beta|Y_t|+\beta_0\|Y_t\|_{L^2(\Om)}\\
				\ns\ds&\quad+\gamma_0\|Z_t\|^{1+\alpha}_{L^2(\Om)}
				+\frac{1}{2}\gamma|Z_t|^2\Big]dt+\gamma\Psi(t,|Y_t|)\sgn(Y_t)Z_tdW_t+\gamma\Psi(t,|Y_t|)dL_t\\
				\ns\ds &\geq \gamma\Psi(t,|Y_t|)\sgn(Y_t)Z_tdW_t.
			\end{aligned}
		\end{equation*}
		Integrating from $t$ to $T$ firstly and then taking conditional expectation $\mathbb{E}_t[\cdot]$, we obtain
		\begin{align*}
			\exp\big(\gamma|Y_t|  \big)
			\leq&\
			\dbE_t
			\exp\Big\{
			\gamma|\eta|+\gamma\int_t^T\(\theta_s+\beta|Y_s|+\beta_0\|Y_s\|_{L^2(\Om)}
			+\gamma_0\|Z_s\|^{1+\alpha}_{L^2(\Om)}\)ds
			\Big\}\\
			\ns\ds \les&\ \exp\Big\{
			\gamma(K_1+K_2)+\g\int_t^T(\beta+\beta_0)\|Y\|_{S_\dbF^\infty(s,T)}ds
			+\g\gamma_0\int_t^T\|Z_s\|^{1+\alpha}_{L^2(\Om)}ds
			\Big\},
		\end{align*}
		where we have used the condition (iii) of \autoref{ass 3.1}. Therefore, it yields that for $t\in[T-\kappa, T]$,
		\begin{equation}\label{equ 3.21}
			|Y_t|\leq (K_1+K_2)+\int_t^T(\beta+\beta_0)\|Y\|_{S_\dbF^\infty(s,T)}ds
			+\gamma_0\int_t^T\|Z_s\|^{1+\alpha}_{L^2(\Om)}ds.
		\end{equation}
		The last term in the right hand side of \rf{equ 3.21} is estimated below.
		Since we have from Young's inequality  that for positive numbers $a$ and $b$,
		\begin{equation}\label{equ 3.22}
			a b^{1+\a}=\left(\left(\frac{1+\a}{2}\right)^{\frac{1+\a}{1-\a}} a^{\frac{2}{1-\a}}\right)^{\frac{1-\a}{2}}\left(\frac{2}{1+\a} b^{2}\right)^{\frac{1+\a}{2}} \leq b^{2}+\frac{1-\a}{2}\left(\frac{1+\a}{2}\right)^{\frac{1+\a}{1-\a}} a^{\frac{2}{1-\a}}, 
		\end{equation}
		it follows for  a positive constant  $\e_0$ to be determined later (taking 
		$a=\frac{2\gamma_0}{\tilde{\gamma}\e_0}$ and $b=\|Z_s\|_{L^2(\Om)}$ in the preceding inequality)
		\begin{equation*}
			\begin{aligned}
				\gamma_0\|Z_s\|^{1+\alpha}_{L^2(\Om)}&=\frac{\tilde{\gamma}\e_0}{2}
				\Big[\frac{2\gamma_0}{\tilde{\gamma}\e_0}\|Z_s\|^{1+\alpha}_{L^2(\Om)}\Big]
				\leq \frac{\tilde{\gamma}\e_0}{2}\dbE|Z_s|^2+L_3.
			\end{aligned}
		\end{equation*}
		In view of inequality~\eqref{equ 3.21}, we arrive at the following inequality: for $t\in[T-\kappa,T]$,
		\begin{equation}\label{equ 3.25}
			|Y_t|\leq (K_1+K_2)+\int_t^T(\beta+\beta_0)\|Y\|_{S_\dbF^\infty(s,T)}ds
			+L_3T +\dbE\int_t^T\frac{\tilde{\gamma}\e_0}{2}|Z_s|^2ds.
		\end{equation}
		The last term in the right hand side of the last inequality is estimated below. We define the  stopping time
		$$\tau_k:=T\wedge \inf\Big\{s\in[t,T]: \int_t^s|Z_r|^2dr\geq k\Big\},$$
		with the convention that $\inf \emptyset=\infty$.
		Then the condition \rf{4.7.4} allows us to show that
		\begin{equation*}\label{equ 3.25-1}
			\begin{aligned}
				\ds &Y_t-Y_{\tau_k}=\int_t^{\tau_k}g(s,Y_s,Z_s, \dbP_{Y_s},\dbP_{Z_s})ds-\int_t^{\tau_k}Z_sdW_s\\
				\ns\ds&\geq\int_t^{\tau_k} \(
				\frac{\tilde{\gamma}}{2}|Z_s|^2-\theta_s(\omega)
				-\beta|Y_s|-\beta_0\|Y_s\|_{L^2(\Om)}
				-\gamma_0\|Z_s\|_{L^2(\Om)}^{1+\alpha}
				\)ds-\int_t^{\tau_k}Z_sdW_s.
			\end{aligned}
		\end{equation*}
		In view of the boundness of $\theta$, it follows that
		\begin{equation}\label{equ 3.25-2}
			\begin{aligned}
				\ds& \int_t^{\tau_k} \frac{\tilde{\gamma}}{2}|Z_s|^2ds\\
				&\leq Y_t-Y_{\tau_k}+
				\int_t^{\tau_k}\big(
				\theta_s(\omega)
				+\beta|Y_s|+\beta_0\|Y_s\|_{L^2(\Om)}
				+\gamma_0\|Z_s\|_{L^2(\Om)}^{1+\alpha}
				\big)ds+\int_t^{\tau_k}Z_sdW_s\\
				\ns\ds &\leq [2+(\beta+\beta_0)T]\|Y\|_{S_\dbF^\infty(t,T)}
				+K_2+\int_t^{\tau_k}
				\gamma_0\|Z_s\|_{L^2(\Om)}^{1+\alpha}ds+\int_t^{\tau_k}Z_sdW_s.
			\end{aligned}
		\end{equation}
		Multiplying $\e_0$ and taking the expectation on both sides of the above inequality, and using Fatou's lemma, we have
		\begin{equation}\label{equ 3.26}
			\begin{aligned}
				\dbE
				\int_t^T \frac{\tilde{\gamma}\e_0}{2}|Z_s|^2ds
				\leq
				\e_0[2+(\beta+\beta_0)T]\|Y\|_{S_\dbF^\infty(t,T)}
				+\e_0K_2
				+\dbE\int_t^T\e_0\gamma_0\|Z_s\|_{L^2(\Om)}^{1+\alpha}ds.
			\end{aligned}
		\end{equation}
		On the other hand, taking $a=\frac{4\gamma_0}{\tilde{\gamma}}$ and  $b=\|Z_s\|_{L^2(\Om)}$ in the inequality \eqref{equ 3.22}, it yields
		\begin{equation}\label{equ 3.27}
			\begin{aligned}
				\e_0\gamma_0\|Z_s\|_{L^2(\Om)}^{1+\alpha}
				&=\frac{\tilde{\gamma}\e_0}{4}
				\bigg[\frac{4\gamma_0}{\tilde{\gamma}}\|Z_s\|_{L^2(\Om)}^{1+\alpha}\bigg]
				\leq \dbE\Big[
				\frac{\tilde{\gamma}\e_0}{4}|Z_s|^2\Big]+L_4.
			\end{aligned}
		\end{equation}
		Combining \eqref{equ 3.26} and \eqref{equ 3.27}, we obtain
		
		\begin{equation}\label{equ 3.28}
			\begin{aligned}
				\dbE
				\int_t^T \frac{\tilde{\gamma}\e_0}{2}|Z_s|^2ds
				\leq 2\varepsilon_0[2+(\beta+\beta_0)T]\|Y\|_{S_\dbF^\infty(t,T)}
				+2\e_0K_2+2L_4T.
			\end{aligned}
		\end{equation}
		Hence, from (\ref{equ 3.25}) and (\ref{equ 3.28}), one gets
		\begin{equation*}
			\begin{aligned}
				\ds |Y_t|\leq&\ (K_1+K_2)+\int_t^T(\beta+\beta_0)\|Y\|_{S_\dbF^\infty(s,T)}ds+L_3T\\
				\ns\ds &\ + 2\varepsilon_0[2+(\beta+\beta_0)T]\|Y\|_{S_\dbF^\infty(t,T)}
				+2\e_0K_2+2L_4T.
			\end{aligned}
		\end{equation*}
		Setting 
		$$\e_0:=\frac{1}{4[2+(\beta+\beta_0)T]},$$
		it follows
		\begin{equation*}
			\begin{aligned}
				\|Y\|_{S_\dbF^\infty(t,T)}
				&\leq 2(K_1+K_2)+2L_3T
				+4\e_0K_2+4L_4T+2\int_t^T(\beta+\beta_0)\|Y\|_{S_\dbF^\infty(s,T)}ds.
			\end{aligned}
		\end{equation*}
		Using Gronwall's inequality, we have
		\begin{equation*}
			\begin{aligned}
				\|Y\|_{S_\dbF^\infty(t,T)}\leq L_5, \q~ \forall t\in[T-\kappa,T].
			\end{aligned}
		\end{equation*}
		In particular,
		\begin{equation}\label{equ 3.30}
			\begin{aligned}
				\|Y\|_{S_\dbF^\infty(T-\kappa,T)}\leq L_5.
			\end{aligned}
		\end{equation}
		In addition, taking the conditional expectation $\mathbb{E}_t[\cdot]$ on both sides of (\ref{equ 3.25-2}), similar to (\ref{equ 3.26}), we have
		\begin{equation}\label{equ 3.31}
			\begin{aligned}
				\dbE_t
				\int_t^T \frac{\tilde{\gamma}\e_0}{2}|Z_s|^2ds
				\leq
				\e_0[2+(\beta+\beta_0)T]\|Y\|_{S_\dbF^\infty(t,T)}
				+\e_0K_2
				+\int_t^T\e_0\gamma_0\|Z_s\|_{L^2(\Om)}^{1+\alpha}ds.
			\end{aligned}
		\end{equation}
		From (\ref{equ 3.27}) and (\ref{equ 3.31}), we see
		\begin{equation}\label{equ 3.322222}
			\begin{aligned}
				\dbE_t
				\int_t^T \frac{\tilde{\gamma}\e_0}{2}|Z_s|^2ds
				\leq
				\e_0[2+(\beta+\beta_0)T]\|Y\|_{S_\dbF^\infty(t,T)}
				+\e_0K_2+L_4T
				+\mathbb{E}\int_t^T\frac{\tilde{\gamma}\varepsilon_0}{4}|Z_s|^2ds.
			\end{aligned}
		\end{equation}
		In view of (\ref{equ 3.28}), we have
		\begin{equation*}\label{equ 3.33}
			\begin{aligned}
				\dbE_t
				\int_t^T \frac{\tilde{\gamma}\e_0}{2}|Z_s|^2ds
				\leq
				2\e_0[2+(\beta+\beta_0)T]\|Y\|_{S_\dbF^\infty(t,T)}
				+2\e_0K_2+2L_4T.
			\end{aligned}
		\end{equation*}
		Combining the definition of $\varepsilon_0$  and (\ref{equ 3.30}), we have
		\begin{equation*}\label{equ 3.35}
			\|Z\|^2_{ \cZ^2_\dbF(T-\kappa,T)}\leq L_6.
		\end{equation*}


			{\it {Step 2.}} We prove the existence and uniqueness of global adapted solutions.

			Note that the main idea in proving the existence and uniqueness result on the mean-field BSDE is twofold. Firstly, we employ \autoref{th 3.2}  to establish the unique solution of the mean-field BSDE (\ref{MFBSDE}) on small time intervals. Secondly, we utilize {\it Step 1} to demonstrate that the obtained solution $(Y,Z)$ satisfies the estimates \rf{equ 3.20} over the whole interval $[0, T]$. The detailed proof is structured into several steps. Let us proceed with the first step.

			%
			First, let $\kappa_0=\frac{\e^*}{2}$, where the positive constant $\e^*$ is given in  \autoref{th 3.2}. Since   $\|\eta\|\les K_1$,  \autoref{th 3.2} implies that the mean-field BSDE (\ref{MFBSDE}) has a unique adapted solution $(Y^{(1)},Z^{(1)})\in S_\dbF^\infty(T-\kappa_0,T)\ts\cZ^2_\dbF(T-\kappa_0,T)$ on the interval $[T-\kappa_0, T]$.
			In addition, {\it {Step 1}} implies that
			\begin{equation}\label{4.7.3}
				\|Y^{(1)}\|_{S_\dbF^\infty(T-\kappa_0,T)}\leq  L_5\q~
				\text{and}\q~
				\|Z^{(1)}\|^2_{ \cZ^2_\dbF(T-\kappa_0,T)}\leq L_6.
			\end{equation}
			%

			Second, we prove the existence and uniqueness result on the small interval $[T-2\kappa_0, T-\kappa_0]$. Note that here the terminal value is $Y^{(1)}_{T-\kappa_0}$, which is bounded by $L_5$ from \rf{4.7.3}.
			From \autoref{th 3.2} again we have that the mean-field BSDE (\ref{MFBSDE}) has a unique adapted solution
			$(Y^{(2)},Z^{(2)})\in S_\dbF^\infty(T-2\kappa_0, T-\kappa_0)\ts\cZ^2_\dbF(T-2\kappa_0, T-\kappa_0)$
			on the interval $[T-2\kappa_0, T-\kappa_0]$.  Define for $t\in [T-\kappa_0, T]$, 
			$$
			\left(\tilde{Y}_t, \tilde{Z}_t\right):=\left(Y^{(1)}_t, Z^{(1)}_t\right)\, \chi_{[T-\kappa_0, T]}(t)
			+\left(Y^{(2)}_t,Z^{(2)}_t \right)\, \chi_{[T-2\kappa_0, T-\kappa_0)}(t).$$
			Then $(\tilde{Y},\tilde{Z})$  is the unique solution to  
			the mean-field BSDE (\ref{MFBSDE}) on the interval $[T-2\kappa_0,T]$. 
			And similar to {\it Step 1}, we have
			$$\|Y\|_{S^\i_{\dbF}(T-2\kappa_0,T)}\les L_5\q \hbox{and}\q \|Z\|^2_{ \cZ^2_\dbF(T-2\kappa_0,T)}\leq L_6.$$

			Finally,  set $m\deq[\frac{T}{\kappa_0}]+1$.
			Repeating $m$ times the previous procedure, we  establish the existence and uniqueness of a solution $(Y,Z)$ for the mean-field BSDE (\ref{MFBSDE}) on the interval $[0,T]$. Furthermore, we have
			$$\|Y\|_{S^\i_{\dbF}(0,T)}\les L_5\q \hbox{and}\q \|Z\|^2_{ \cZ^2_\dbF(0,T)}\leq L_6.$$
		This completes the proof.
	\end{proof}
	

	In \autoref{th 3.5-1},   the condition (ii) of \autoref{ass 3.2-11}  is significant. 
	However, without this condition,  it is possible to obtain a unique global solution of BSDE \rf{MFBSDE} if $\gamma_0$ is sufficiently small and $g$ is bounded with respect to $Y$ and $\dbP_Y$.

	\bas{ass 3.1.2}\rm  The terminal value $\eta:\Om\rightarrow\dbR$ and the generator $g: \Omega\times[0,T]\times\dbR\times\dbR^d\times \cP_2(\dbR)\ts\cP_2(\dbR^{d})\to \dbR$ satisfy the following conditions:
	\begin{enumerate}[~~\,\rm (i)]
		\item For $d\dbP\times dt$-almost all $(\o,t)\in\Om\ts[0,T]$, 
		$$ |g(w,t,y,z,\mu_1,\mu_2)|\les \theta_t(\omega)+\frac{1}{2}\gamma |z|^2
		+\gamma_0\cW_2(\mu_2,\d_{\{0\}})^{1+\alpha}$$
		with  all $(y, z, \mu_1, \mu_2)\in\dbR\times\dbR^d\times \cP_2(\dbR)\ts\cP_2(\dbR^{d})$. 
		\medskip
		\item There is a positive constant $C$ such that for $d\dbP\times dt$-almost all  $(\o,t)\in\Om\ts[0,T]$, 
		$$\ba{ll}
		&|g(\o,t,y,z,\m_1,\mu_2)-g(\o,t,\bar y,\bz,\bar\mu_1,\bar\mu_2)| \\
		&\les  C\bigg\{ |y-\bar{y}|+(1+|z|+|\bar{z}|)\cd|z-\bar{z}|+\cW_2(\mu_1,\bar{\mu}_1)\\
		&\qq\q  +\Big(1+\cW_2(\mu_2,\d_{\{0\}})^{\alpha}
		+\cW_2(\bar{\mu}_2,\d_{\{0\}})^{\alpha}\Big)\cW_2(\mu_2,\bar{\mu}_2)\bigg\}
		\ea$$
		with  all $(y, z; \by,\bz;\mu_1, \mu_2; \bar{\m}_1, \bar{\m}_2)\in \left(\dbR\times\dbR^d\right)^2\times \left(\cP_2(\dbR)\ts\cP_2(\dbR^{d})\right)^2$. 
		\medskip
		\item  There are two positive constants $K_1$ and $K_2$ such that
		$$\|\eta\|_\infty\les K_1\q\hb{and}\q \bigg\|\int_0^T\theta_t(\omega)dt\bigg\|_\infty\les K_2.$$
	\end{enumerate}\eas

	\begin{theorem}\label{th 3.5-12}
		Under   \autoref{ass 3.1.2}, if  $\gamma_0$ is sufficiently small,  BSDE \rf{MFBSDE}
		has a unique global solution $(Y,Z)\in S_\dbF^\infty(0,T;\dbR)\ts \cZ^2_\dbF(0,T;\dbR^d)$.
		Furthermore, there exist two positive constants $\tilde M_1$ and $\tilde M_2$, which only depend on $(K_1,K_2,T,\alpha, {\gamma},\gamma_0,C)$,  such that 
		\begin{align}\label{4.7.1-1}
			\|Y\|_{S_\dbF^\infty(0,T)}\leq  \tilde M_1  \q~
			\text{and}\q~  \|Z\|^2_{ \cZ^2_\dbF(0,T)}\leq\tilde  M_2.
		\end{align} 
	\end{theorem}

	\begin{proof}
		For simplicity of presentation, we firstly consider the case that $g$ does not depend on $Y$ and $\dbP_Y$:
		\begin{equation}\label{bsde12.15}
			Y_t=\eta+\int_t^T g(s,Z_s,\dbP_{Z_s})ds-\int_t^TZ_sdW_s,\q t\in[0,T].
		\end{equation}
		We  construct a Cauchy sequence $\{(Y^n,Z^n), n\ge 0\}$ to prove the result. For this, we let $Z^1=0$ and for $n\ges 1$,
		\begin{equation}\label{bsde12.15.2}
			Y^{n+1}_t=\eta+\int_t^Tg(s,Z^{n+1}_s,\dbP_{Z^n_s})ds-Z^{n+1}_sdW_s, \q t\in[0,T].
		\end{equation}
		Then, according to \autoref{le 7.1}, BSDE (\ref{bsde12.15.2}) admits a unique solution
		$(Y^{n+1},Z^{n+1})\in S_\dbF^\infty(0,T;\dbR)\ts \cZ^2_\dbF(0,T;\dbR^d)$ for each $n\in\dbN$.
		Moreover,  
		similar to \rf{equ 7.2} and \rf{equ 7.3}, and recall the inequality \rf{equ 3.22}, we have  that (by taking $\theta=0$ for simplicity) for any $t\in[0,T]$,
		
		\begin{equation}\label{equ 3.3-2}
			\begin{aligned}
				\|Y^{n+1}\|_{S_\dbF^\infty(t,T)}\leq&\  \|\eta\|_\infty
				+\gamma_0T^{\frac{1-\alpha}{2}}\|Z^n\|^{1+\alpha}_{\cZ^2_\dbF(t,T)}\\
				\les &\ \|\eta\|_\infty
				+\gamma_0T^{\frac{1-\alpha}{2}}(C_\a+\|Z^n\|^{2}_{\cZ^2_\dbF(t,T)}),
			\end{aligned}
		\end{equation}
		and
		\begin{equation}\label{equ 3.6-2}
			\begin{aligned}
				& \|Z^{n+1}\|_{\cZ^2_\dbF(t,T)}^2
				\leq\frac{1}{\gamma^2}\exp\big(2\gamma\|\eta\|_\infty\big)
				+\frac{2\gamma_0}{\gamma}\exp\big(2\gamma\|Y^{n+1}\|_{S_\dbF^\infty(t,T)}\big)\cd
				T^{\frac{1-\a}{2}}\|Z^n\|^{1+\alpha}_{\cZ^2_\dbF(t,T)}\\
				&\	\leq
				\frac{2\gamma_0}{\gamma}\exp\Big\{2\gamma\(\|\eta\|_\infty
				+\gamma_0T^{\frac{1-\alpha}{2}}(C_\a+\|Z^n\|^{2}_{\cZ^2_\dbF(t,T)})  \)  \Big\}
				\cd
				T^{\frac{1-\a}{2}}(C_\a+\|Z^n\|^{2}_{\cZ^2_\dbF(t,T)})\\
				&\q	\ +\frac{1}{\gamma^2}\exp\big(2\gamma\|\eta\|_\infty\big),
			\end{aligned}
		\end{equation}
		where $C_\a$ is a constant only depending on $\a$. 
		We now prove by induction that for each $n$,
		\begin{equation}\label{12.11.1}
			\|Z^n\|^{2}_{\cZ^2_\dbF(t,T)}\les \frac{2}{\gamma^2}\exp\big(2\gamma\|\eta\|_\infty\big)\deq \tilde M_2.
		\end{equation}
		In fact,  if  the last inequality is true  for $n$,  we have  for  sufficiently small $\g_0$, 
		$$
		\frac{2\gamma_0}{\gamma}\exp\Big\{2\gamma\(\|\eta\|_\infty
		+\gamma_0T^{\frac{1-\alpha}{2}}(C_\a+\tilde M_2 ) \)  \Big\}\cd T^{\frac{1-\a}{2}}(C_\a+\tilde M_2)
		\les \frac{1}{2}\tilde M_2,
		$$
		and further in view of  \rf{equ 3.6-2}, we have
		\begin{equation*}
			\|Z^{n+1}\|^{2}_{\cZ^2_\dbF(t,T)}\les \tilde M_2, 
		\end{equation*}
		which is  the inequality \rf{12.11.1} with $n$ being replaced with  $n+1$. Hence,   the inequality \rf{12.11.1}  is all true for  $n\in\dbN$ by induction.  Moreover, we have from \rf{equ 3.3-2}, 
		\begin{equation}\label{equ 3.3-3}
			\begin{aligned}
				\|Y^{n+1}\|_{S_\dbF^\infty(t,T)}
				\les  \|\eta\|_\infty
				+\gamma_0T^{\frac{1-\alpha}{2}}(C_\a+\tilde M_2)\deq \tilde M_1.
			\end{aligned}
		\end{equation}
		Therefore, on the one hand, the estimate \rf{4.7.1-1} is true if BSDE \rf{bsde12.15} admits a solution. 
		On the other hand, from \autoref{th 3.2}, it is easy to see that 
		$\{(Y^n,Z^n)\}$  is a Cauchy sequence, which implies that BSDE \rf{bsde12.15} admits a  solution  on the interval $[T-\varepsilon,T]$. Repeating the above procedure in a finite number of steps, we obtain the existence result on the whole interval $[0, T]$. Moreover, the proof of uniqueness is similar to that of  \autoref{th 3.5}.
		
		Finally,   for the general mean-field BSDE \rf{MFBSDE}, since $g$ has a bounded growth in  $(y, \mu_1)$,  the proof is almost identical to the above.
	\end{proof}

	\begin{remark}\rm 
		In \autoref{th 3.5-12}, the smallness assumption on $\gamma_0$ is dispensed with  if the law term of $Z$ is additive and exhibits a quadratic growth.
	\end{remark}
	
	Let us consider the following mean-field BSDE with quadratic growth in both $Z$ and $\dbP_Z$:
	\begin{equation}\label{equ 3.200}
		Y_t=\eta+\int_{t}^T\left[g_1(s,Z_s)+g_2(s,\dbP_{Z_s})\right]ds-\int_{t}^TZ_s\, dW_s, \q t\in[0,T]. 
	\end{equation}
	Here, the two adapted fields  $g_1:\O\ts [0,T]\ts \dbR^d\ra \dbR$ and $g_2:\O\ts [0,T]\ts \cP_2(\dbR^d)\ra \dbR$ satisfy the following assumption.

	\bas{ass 3.11111}\rm   There is a  positive constant $C$ such that for $d\dbP\times dt$-almost all  $(\o,t)\in\Om\ts[0,T]$,  the following are satisfied:
	\begin{enumerate}[~~\,\rm (i)]
		\item  $|g_1(w,t,z)|\les \theta_t(\omega)+\frac{1}{2}\gamma|z|^2$ and $
		|g_2(w,t,\m_2)|\les  \theta_t(\omega) 
		+\gamma_0\cW_2(\mu_2,\d_{\{0\}})^{2}$ for $\forall (z, \mu_2)\in\mathbb{R}^d\times \mathcal{P}_2(\mathbb{R}^{d})$. 
		\medskip
		\item  
		$|g_1(t,z)-g_1(t,\bar{z})|\leq C(1+|z|+|\bar{z}|)|z-\bar{z}|$ for $\forall (z,\bar{z})\in \dbR^d\times \dbR^d.$ 
		\medskip
		\item   $\|\eta\|_\infty\les K_1\q\hb{and}\q \left\|\int_0^T\theta_tdt\right\|_\infty\les K_2$ for  two positive constants $K_1$ and $K_2$. 
	\end{enumerate}\eas

	\begin{proposition}\label{24.1.13}
		Let \autoref{ass 3.11111} be satisfied. Then the equation (\ref{equ 3.200}) admits 
		a unique global solution $(Y,Z)\in S_\dbF^\infty(0,T;\dbR)\ts \cZ^2_\dbF(0,T;\dbR^d)$. 	Furthermore, there exist two positive constants $\tilde M_1$ and $\tilde M_2$,  only depending on $(K_1,K_2,T, {\gamma},\gamma_0,C)$, such that
		\begin{align}\label{4.7.2}
			\|Y\|_{S_\dbF^\infty(0,T)}\leq  \tilde M_1  \q~
			\text{and}\q~  \|Z\|^2_{ \cZ^2_\dbF(0,T)}\leq\tilde  M_2.
		\end{align} 
	\end{proposition}
	\begin{proof}
		First,  \autoref{le 7.1}  implies that the following equation 
		\begin{equation*}
			\bar{Y}_t=\eta+\int_{t}^{T}g_1(s,Z_s)\, ds-\int_{t}^{T}Z_s\, dW_s,\q t\in[0,T],
		\end{equation*}
		admits a unique solution  $(\bar Y,Z)\in S_\dbF^\infty(0,T;\dbR)\ts \cZ^2_\dbF(0,T;\dbR^d)$ such that  
		$$	\|\bar Y\|_{S_\dbF^\infty(0,T)}\leq  \tilde M_1  \q~
		\text{and}\q~  \|Z\|^2_{ \cZ^2_\dbF(0,T)}\leq\tilde  M_2.$$
		Setting  
		\begin{equation*}
			Y_t:=\bar{Y}_t+\int_t^Tg_2(s,\dbP_{Z_s})\, ds,\q t\in[0,T],
		\end{equation*}
		we see that the pair $(Y,Z)$ is a solution of BSDE (\ref{equ 3.200}). Moreover, we have
		\begin{align*}
			|Y_t|& \les \|\bar Y\|_{S_\dbF^\infty(0,T)}
			+\int_t^T \left(|\theta_s(\omega)|+\gamma_0 \dbE |Z_s|^2 \right)\, ds	\les \tilde M_1 +K_2+\tilde M_2.
		\end{align*}
		Thus,   $\|Y\|_{S_\dbF^\infty(0,T)}\leq \tilde M_1 +K_2+\tilde M_2$.
		
		We now consider  the uniqueness. Let $(Y^i,Z^i),\  i=1,2$ be two solutions to the equation (\ref{equ 3.200}).
		Define 
		\begin{equation*}
			\bar{Y}^i_t:=Y^i_t-\int_t^Tg_2(s,\dbP_{Z^i_s})\, ds, \quad t\in [0,T]; \quad i=1,2.
		\end{equation*}
		Then, the pair of processes $(\bar{Y}^i,{Z}^i )$ satisfy the following equation 
		\begin{equation*}
			\bar{Y}^i_t=\eta+\int_{t}^{T}g_1(s,Z^i_s)\, ds-\int_{t}^{T}Z^i_s\, dW_s,\q t\in[0,T]; \quad i=1,2.
		\end{equation*}
		Since  BSDE \rf{equ 7.1} has a unique solution, we have  $\bar{Y}^1=\bar{Y}^2$ and  $Z^1=Z^2$ on $[0,T]$. Thus ${Y}^1={Y}^2$  on $[0,T]$. The proof is complete.
	\end{proof}

\vspace{-3mm}

\section{Comparison Theorem}\label{Sec4}

In this section, we present a comparison theorem for the solutions of the mean-field BSDE (\ref{MFBSDE}) with quadratic growth. To simplify the presentation, we will refer to BSDE \rf{MFBSDE} as the BSDE \rf{MFBSDE} with parameters $(g, \eta)$, which indicates that the generator is $g$ and the terminal value is $\eta$. 
%
%
\begin{theorem}[Comparison principle]\label{th 4.1}
Let the parameters $(g,\eta)$ and $(\bar g,\bar{\eta})$ satisfy
\autoref{ass 3.2}.
Denote by $(Y,Z)$ and $(\bar{Y},\bar{Z})$ the global adapted solutions of mean-field BSDE
(\ref{MFBSDE}) with parameters $(\eta, g)$ and $(\bar{\eta},\bar{g})$, respectively.
In addition, suppose that
\begin{enumerate}[~~\,\rm (i)]
	\item  One of both generators $g$ and $\bar g$ does not depend on $\mu_2$.
	\item  The other one of both generators $g$ and $\bar g$ is nondecreasing with respect to
	$\mu_1$ in the following sense: there exists a positive constant
	$K$ such that
	for every $(\o,t)\in\Om\ts[0,T]$, $y\in \dbR, z\in \dbR^d$, $\xi,\bar{\xi}\in L^2_{\sF_t}(\Om;\dbR)$ and $\zeta\in L^2_{\sF_t}(\Om;\dbR^d)$,
	$$
	|g(\o,t,y,z,\dbP_{\xi},\dbP_{\zeta})-g(\o,t,y,z,\dbP_{\bar{\xi}},\dbP_{\zeta})|\leq
	K \|(\xi-\bar{\xi})^+\|_{L^2(\Om)}.
	$$
\end{enumerate}
Then, if $\eta\leq \bar{\eta}$ and $g\leq  \bar{g}$,
$\dbP$-a.s., we have that for every $t\in[0,T]$,
$$Y_t\leq \bar{Y}_t,\q \dbP\hb{-a.s.}$$
\end{theorem}

Before presenting the proof of the comparison theorem, we would like to provide some remarks.

\begin{remark} \rm
If either the generator $g$ depends on $\dbP_Z$ or $g$ decreases with respect to $\dbP_Y$, then the comparison theorem will not hold.  Counterexamples illustrating this can be found in Buckdahn, Li, and Peng \cite[Example 3.1, Example 3.2]{BLP}. 
\end{remark}

\begin{remark}\rm 
Under the scenario where the mean-field term in the generator appears as an expectation, Buckdahn, Li, and Peng \cite[Example 3.2 and Theorem 3.2]{BLP} impose a condition called $\mathbf{Condition \ (L)}$ for the comparison property of mean-field BSDEs, which can be restated as follows:
$$
\begin{array}{ll}
	&\mathbf{Condition \ (L)}\q
	\hb{the generator $g$ (or $\bar g$) is nondecreasing in  the mean-field term $y'$.}
\end{array}
$$
One way to interpret this condition is that if the derivative of $g$ (or $\bar g$) with respect to the mean-field term $y'$ exists, it should not be less than zero. 
In the context of measure-dependence, assumption $\mathrm{(ii)}$ in \autoref{th 4.1} serves as an equivalent assumption to $\mathbf{Condition \ (L)}$.

\end{remark}

\begin{remark} \rm 
Here are  sufficient conditions for  $\mathrm{(ii)}$ in \autoref{th 4.1}.
Assume that the function $g(\o, t,y,z,\cdot,\mu_2)$ is differentiable with respect to
$\mu_1$ for every $(\o,t,y,z,\mu_2)\in\Om\ts[0,T]\times\dbR\times \dbR^d\times \mathcal{P}_2(\dbR^d)$.
If there is a positive constant $K$ such that for every $(\o,t,y,z,\mu_1,\mu_2, a)\in\Omega\times [0,T]\times \dbR\times \dbR^d\times\mathcal{P}_2(\dbR)\times \mathcal{P}_2(\dbR^d)\times \dbR$,
$$
0\leq \partial_{\mu_1}g(s,y,z,\mu_1, \mu_2;a)\leq K,
$$
then the assumption $\mathrm{(ii)}$ in \autoref{th 4.1} holds true. For more details,  see Li, Liang, and Zhang \cite[Remark 2.2]{LLZ}.
\end{remark}

\begin{proof}[Proof of \autoref{th 4.1}]
Without loss of generality, we assume that $g$ satisfies the item (i) and $\bar g$ satisfies the item (ii).
Set for every $s\in[t,T]$,
\begin{align*}
	&\delta Y_s= Y_s-\bar{Y}_s,\q~ \delta Z_s=Z_s-\bar{Z}_s,\q~ \delta \eta=\eta-\bar{\eta},\\
	&\delta g(s)= g(s,\bar{Y}_s,\bar{Z}_s,\dbP_{Y_s})-\bar{g}(s,\bar{Y}_s,\bar{Z}_s,\dbP_{Y_s},\dbP_{\bar{Z}_s}).
\end{align*}
Then,
\begin{equation}\label{equ 4.1}
	\begin{aligned}
		\delta Y_t=&\delta \eta+\int_t^T \big[g(s,Y_s,Z_s,\dbP_{{Y_s}})
		-\bar{g}(s,\bar{Y}_s,\bar{Z}_s,\dbP_{\bar{Y}_s},\dbP_{\bar{Z}_s})\big]ds\\
		&-\int_t^T\delta Z_sdW_s,\q 0\les t\les T.
	\end{aligned}
\end{equation}
On the generator of BSDE \rf{equ 4.1}, we have that for $s\in[t,T]$,
\begin{equation*}\label{equ 4.2}
	\begin{aligned}
		&g(s,Y_s,Z_s,\dbP_{Y_s})-\bar{g}(s,\bar{Y}_s,\bar{Z}_s,\dbP_{\bar{Y}_s},\dbP_{\bar{Z}_s})\\
		&=\d_yg(s)\delta Y_s+\d_zg(s)\delta Z_s+\d_\mu\bar{g}(s)\|(\d Y_s)^+\|_{L^2(\Om)}+\d g(s),
	\end{aligned}
\end{equation*}
where
$$\d_yg(s)=\left\{\ba{ll}
\frac{g(s,Y_s,Z_s,\dbP_{Y_s})-g(s,\bar{Y}_s,Z_s,\dbP_{Y_s})}{Y_s-\bar{Y}_s},&\quad \text{if }\ Y_s\neq\bar{Y}_s;\\
0                                                                    ,&\quad \text{if }\ Y_s=\bar{Y}_s,
\ea\right.$$
$$\d_zg(s)=\left\{\ba{ll}
\frac{(g(s,\bar{Y}_s,Z_s,\dbP_{Y_s})-g(s,\bar{Y}_s,\bar{Z}_s,\dbP_{Y_s}))(Z_s-\bar{Z}_s)

}{|Z_s-\bar{Z}_s|^2},&\quad \text{if }\ Z_s\neq\bar{Z}_s;\\
0                                                                    ,&\quad \text{if }\ Z_s=\bar{Z}_s,
\ea\right.$$
and
$$\d_\mu\bar{g}(s)=\left\{\ba{ll}
\frac{\bar{g}(s,\bar{Y}_s,\bar{Z}_s,\dbP_{Y_s},\dbP_{\bar{Z}_s})
	-\bar{g}(s,\bar{Y}_s,\bar{Z}_s,\dbP_{\bar{Y}_s}, \dbP_{\bar{Z}_s})}
{\|(\delta Y_s)^+\|_{L^2(\Om)}},&\quad \text{if }\ \|(\delta Y_s)^+\|_{L^2(\Om)}\neq0;\\
0  ,&\quad \text{if }\ \|(\delta Y_s)^+\|_{L^2(\Om)}=0.
\ea\right.$$
From \autoref{ass 3.2}, the item $\mathrm{(ii)}$ of \autoref{th 4.1}  and  \autoref{th 3.5}, we have
\begin{align}\label{22.4.27.1} 
	& |\d_yg(s)|\leq K,\q~ |\d_\mu\bar{g}(s)|\leq K,\q~ 0\les t\les s\les T,
\end{align}
and
\begin{align*}
	&|\d_zg(s)|\leq \phi(M_1)(1+|Z_s|+|\bar{Z}_s|),\q~ 0\les t\les s\les T,
\end{align*}
where $M_1$ is given in (\ref{4.7.0}).
Moreover, since both $Z$ and $\bar{Z}$ belong to the space $\cZ^2_{\dbF}[0,T]$,  we have $\d_zg\in \cZ^2_{\dbF}[0,T]$. Define that for $ t\in [0,T]$,
$$\widetilde{W}_t=W_t-\int_0^t\d_zg(r)dr\q  \hb{and}\q d\dbQ =\sE(\d_z g\cd W)_0^Td\dbP.$$
%
%
Then,  $\dbQ$ is a probability measure equivalent to $\dbP$ and  $\widetilde{W}$ is a Brownian motion under $\dbQ$.
So BSDE (\ref{equ 4.1})  can be rewritten as
\begin{equation*}\left\{
	\begin{aligned}
		d\delta Y_t=&\ -\Big[
		\big(\d_yg(t)\delta Y_t+\d_\mu\bar{g}(t)\|(\d Y_t)^+\|_{L^2(\Om)}+\d g(t)
		\Big]dt+\delta Z_td\wW_t,\q0\les t< T;\\
		\delta Y_T=&\ \delta \eta.
	\end{aligned}\right.
\end{equation*}
Applying It\^{o}'s formula to $e^{\int_t^s\delta_yg(r)dr}\delta Y_s$, we have
\begin{equation*}\label{equ 4.3}
	\begin{aligned}
		\delta Y_t&=\mathbb{E}^\mathbb{Q}_t[e^{\int_t^T\delta_yg(r)dr}\delta \eta]
		+\mathbb{E}^\mathbb{Q}_t\[\int^T_te^{\int_t^s\delta_yg(r)dr}\(\delta_\mu \bar{g}(s)||(\delta Y_s)^+||_{L^2(\Omega)}
		+\delta g(s)\)ds\].
	\end{aligned}
\end{equation*}
Now, since $\delta \eta\leq0$  and $\delta g(s)\leq0$, we have
\begin{equation*}\label{equ 4.4}
	\begin{aligned}
		\delta Y_t&\leq
		\mathbb{E}^\mathbb{Q}_t\[\int^T_te^{\int_t^s\delta_yg(r)dr}\delta_\mu \bar{g}(s)||(\delta Y_s)^+||_{L^2(\Omega)}ds\].
	\end{aligned}
\end{equation*}
According to the fact (\ref{22.4.27.1}),  using H\"{o}lder's inequality, we have
\begin{equation*}\label{equ 4.5}
	\begin{aligned}
		\delta Y_t&\leq e^{KT}K\sqrt{T}
		\Big\{\int^T_t\mathbb{E}^\mathbb{P}[((\delta Y_s)^+)^2]ds\Big\}^\frac{1}{2}.
	\end{aligned}
\end{equation*}
Hence,
\begin{equation*}\label{equ 4.6}
	\begin{aligned}
		((\delta Y_t)^+)^2&\leq e^{2KT}K^2T\int^T_t\mathbb{E}^\mathbb{P}[((\delta Y_s)^+)^2]ds.
	\end{aligned}
\end{equation*}
Taking the expectation $\mathbb{E}^\mathbb{P}$ on both sides of the last
inequality, we have the desired result  from  Gronwall's lemma.
\end{proof}

\begin{remark} \rm 
\autoref{th 4.1} generalizes a related result by Hao, Wen, and Xiong \cite{Hao-Wen-Xiong-22}, where the generator depends on the expectations of $(Y,Z)$. Furthermore, it also generalizes the comparison theorem of Buckdahn, Li, and Peng \cite{BLP}, where the generator $g$ grows linearly with respect to $Z$.   %
\end{remark}

\begin{remark} \rm As in the proof of  \autoref{th 4.1}, it can be verified that \autoref{th 4.1} remains valid under  Assumptions~\ref{ass 3.1} and  \ref{ass 3.2-11}.
\end{remark}


\section{Particle systems}\label{Sec7}

In this section, we  study the particle systems for the mean-field BSDE \rf{MFBSDE} with quadratic growth in two situations: (i) the generator is bounded in the law of $Z$; (ii)  the generator is unbounded in the law of $Z$. We will conduct a detailed analysis of the convergence of the particle systems and provide the convergence rate.

Let  $\{\eta^i;1\leq i\leq N\}$ be $N$ independent copies of $\eta$, and
let $\{W^j;1\leq j\leq N\}$ be $N$ independent $d$-dimensional Brownian motions.
Denote by $(Y^i,Z^{i,j})$ and $(\bar{Y}^i,\bar{Z}^i)$ the adapted solutions to the following BSDEs
\begin{align}
Y^i_t&=\eta^i+\int_t^Tg(s,Y^i_s,Z^{i,i}_s, \nu_s^N,\mu^N_s )ds-\int_t^T\sum\limits_{j=1}^NZ^{i,j}_sdW^j_s,\q t\in[0,T],
\label{8.2}
\end{align}
and
\begin{align}
\bar{Y}^i_t=\eta^i+\int_t^Tg(s,\bar{Y}^i_s,\bar{Z}^{i}_s, \bar{\nu}_s,\bar{\mu}_s)ds-\int_t^T\bar{Z}^{i}_sdW^i_s,\q t\in[0,T], \label{8.2.2}
\end{align}
respectively, where for each $i,j=1,...,N$, $Z^{i,j}$ is a $1\times d$ matrix, and
\begin{equation}\label{24.1.24}
\nu^N_s\deq \frac{1}{N}\sum\limits_{i=1}^N\delta_{Y^i_s},\q
\mu^N_s\deq \frac{1}{N}\sum\limits_{i=1}^N\delta_{Z^{i,i}_s},\q
\bar{\nu}_s\deq \mathbb{P}_{\bar{Y}^i_s},\q
\bar{\mu}_s\deq \mathbb{P}_{\bar{Z}^i_s}.
\end{equation}
%
%
We shall show that the pair $(Y^i,Z^{i,j})$ is close to the pair $(\bar{Y}^i,\bar{Z}^i)$. For simplicity, we denote that for  $i,j=1,...,N$,
\begin{equation}\label{22.8.13.1}
\Delta Y^i=Y^i-\bar{Y}^i,\q
\Delta Z^{i,j}=Z^{i,j}-\bar{Z}^{i,j}\q\hb{with}\q
\bar{Z}^{i,j}=
\left\{\ba{ll}
\bar{Z}^{i},\q\ i=j;\\
0,\qq i\neq j.
\ea\right.
\end{equation}

\subsection{The case that the generator is bounded in the law of $Z$}

Here, in order to better express our ideas, we prefer to present the following assumption, which is slightly stronger than \autoref{ass 3.2}.
\bas{ass 3.3-0}\rm  For $i=1,2,\cdot\cdot\cdot,N$, there exists a positive constant $K$ such that the terminal value $\eta^i:\Om\rightarrow\dbR$ and the generator $g: \Omega\times[0,T]\times\dbR\times\dbR^d\times \cP_2(\dbR)\times \cP_2(\dbR^{d})\rightarrow \dbR$, adapted the filtration of
$\dbF^i$, where $\dbF^i=\{\sF^i_t\}_{t\ges0}$ is the natural filtration of $W^i$ augmented by all the $\dbP$-null sets in $\sF^i$,
satisfy the following conditions:
\begin{enumerate}[~~\,\rm (i)]
\item $d\dbP\times dt$-a.e. $(\o,t)\in\Om\ts[0,T]$, for every $(y, z,\nu,\mu)\in \mathbb{R}\ts\mathbb{R}^d\ts\mathcal{P}_2(\mathbb{R})\ts\mathcal{P}_2(\mathbb{R}^{d}),$
\begin{equation*}
	|g(w,t,y,z,\nu,\mu)|\les \theta_t(\omega)+K|y|+\frac{\gamma}{2}|z|^2+K\cW_2(\nu,\d_{\{0\}}).
\end{equation*}

\item   $d\dbP\times dt$-a.e. $(\o,t)\in\Om\ts[0,T]$, for every $y,\by\in \dbR, z,\bz\in \dbR^d$ and $\nu,\bar{\nu}\in \cP_2(\dbR)$,
$\m,\bar{\m}\in \cP_2(\dbR^{d})$,
$$\ba{ll}
& |g(\o,t,y,z,\nu,\m)-g(\o,t,\bar y,\bz,\bar\nu,\bar\mu)| \les K\big[|y-\bar{y}|+\cW_2(\nu,\bar{\nu})
+\cW_2(\mu,\bar{\mu})\big]\\
&\q + \phi\big(|y|\vee|\bar y|\vee\cW_2(\nu,\d_{\{0\}})\vee\cW_2(\bar \nu,\d_{\{0\}})\big) \cdot \big[(1+|z|+|\bar{z}|)|z-\bar{z}|\big].
\ea$$
\item  There are two positive constants $K_1$ and $K_3$ such that
$$\max_{1\les i\les N}\|\eta^i\|_\infty\les K_1\q\hb{and}\q \bigg\|\int_0^T|\theta_t(\omega)|^2dt\bigg\|_\infty\les K_3.$$
\end{enumerate}\eas

Similar to \rf{4.7.1},  the solution $(Y^i, Z^{i,j})$ of BSDE (\ref{8.2}) has the following property.

\begin{proposition}\label{22.8.29.1} \it 
Under \autoref{ass 3.3-0}, there is a positive constant $C$ depending only on $(K,K_1,K_3,T,\g)$, such that the solution $(Y^i, Z^{i,j})$ of BSDE (\ref{8.2}) admits the following estimate: for each $i,j=1,...,N$,
\begin{equation}\label{5.111}
	||Y^i||_{S_\dbF^\infty(0,T)}\leq C\q\hbox{and}\q   ||Z^{i,j}||_{\cZ^2_\dbF(0,T)}\leq C.
\end{equation}
\end{proposition}
\begin{proof}
	It is easy to check that under \autoref{ass 3.3-0}, the generator $g$ satisfies \cite[(H1)-(H4)]{FHT}. According to \cite[Theorem 2.4]{FHT},
	the equation (\ref{8.2}) admits a unique solution $(Y,Z)\in S_\dbF^\infty(0,T;\dbR)\ts \cZ^2_\dbF(0,T;\dbR^d)$. Next, we show that 
	the bounds of $||Y^i||_{S_\dbF^\infty(0,T)} $ and  $||Z^{i,j}||_{\cZ^2_\dbF(0,T)}$ do not depend on $N$.
	BSDE (\ref{8.2}) can be  written as follows:
	\begin{equation*}
		\begin{aligned}
			Y^i_t&=\eta^i-\int_t^T\sum\limits_{j=1}^NZ^{i,j}_sdW_s^j +\int_t^T\Big[g(s,Y^i_s,Z^{i,i}_s, \nu_s^N,\mu_s^N)-g(s,0,Z_s^{i,i}, \delta_{\{0\}},\mu_s^N)\\
			&\quad+g(s,0,Z_s^{i,i}, \delta_{\{0\}},\mu_s^N) - g(s,0,0,\delta_{\{0\}},\mu_s^N)  +g(s,0,0,\delta_{\{0\}},\mu_s^N)\Big]ds\\
			&=\eta^i-\int_t^T\sum\limits_{j=1}^NZ^{i,j}_sd\widetilde{W}_s^j +\int_t^T\Big[g(s,Y^i_s,Z^{i,i}_s, \nu_s^N,\mu_s^N)-g(s,0,Z_s^{i,i}, \delta_{\{0\}},\mu_s^N)\\
			& \q +g(s,0,0,\delta_{\{0\}},\mu_s^N)\Big]ds,
		\end{aligned}
	\end{equation*}
	where
	$\widetilde{W}^j_t\deq W^j_t-\int_0^t\Gamma(s)ds,\ 0\leq t\leq T$
	is a Brownian motion under $\dbQ^j$ with
	\begin{align*}
		& d\dbQ^j =\sE(\Gamma\cdot W^j)_0^T d\dbP \q \hb{and}\q \Gamma(s)\leq \phi(0)(1+|Z_s^{i,i}|),\\
		&g(s,0,Z_s^{i,i}, \delta_{\{0\}}, \mu_s^N) - g(s,0,0,\delta_{\{0\}}, \mu_s^N)=\Gamma(s)Z_s^{i,i}.
	\end{align*}
	Now, applying It\^{o}'s formula to $|Y^i_t|^2$,  we have
	\begin{equation}
		\begin{aligned}
			&|Y^i_t|^2+\mathbb{E}^{\dbQ^j}_t\Big[\int_t^T\sum\limits_{j=1}^N|Z^{i,j}_s|^2ds\Big]
			=\mathbb{E}^{\dbQ^j}_t[|\eta^i|^2]\\
			&
			+\mathbb{E}^{\dbQ^j}_t\bigg[ \int_t^T2Y_s^i \Big(g(s,Y^i_s,Z^{i,i}_s, \nu_s^N,\mu_s^N)-g(s,0,Z_s^{i,i}, \delta_{\{0\}},\mu_s^N)+
			g(s,0,0,\delta_{\{0\}},\mu_s^N)\Big)ds\bigg].
		\end{aligned}
	\end{equation}
	Thanks to \autoref{ass 3.3-0}, we have
	\begin{equation}\label{equ 4.9}
		\begin{aligned}
			|Y^i_t|^2 &\leq \mathbb{E}^{\dbQ^j}_t[|\eta^i|^2]
			+\mathbb{E}^{\dbQ^j}_t\bigg[ \int_t^T2|Y_s^i| \Big(K |Y^i_s|+K\cW_2(\nu^N_s, \delta_{\{0\}})+\theta_s\Big)ds\bigg]\\
			& \leq \mathbb{E}^{\dbQ^j}_t[|\eta^i|^2]
			+\mathbb{E}^{\dbQ^j}_t\bigg[ \int_t^T\Big((K^2+2K+1)|Y_s^i|^2+\cW^2_2(\nu^N_s, \delta_{\{0\}})+|\theta_s|^2  \Big)ds\bigg]\\
			& \leq K_1^2+K_3+\mathbb{E}^{\dbQ^j}_t\bigg[ \int_t^T \Big((K^2+2K+1)|Y^i_s|^2+ \frac{1}{N} \sum\limits_{i=1}^N|Y^i_s|^2\Big)ds\bigg].
		\end{aligned}
	\end{equation}
	Summing over $i$ on both sides of (\ref{equ 4.9}), using Gronwall's inequality, we have
	$$||Y^i||_{S_\dbF^\infty(0,T)}\leq C,\q i=1,\cdot\cdot\cdot,N,$$
	where $C$ is a positive constant depending only on $(K,K_1,K_3,T)$.
	Next, we prove that the term $Z^{i,j}\cdot W^j$ is a BMO martingale. For this, we let
	\begin{equation*}
		\Phi(x)=\frac{1}{\gamma^2}\big[\exp(\gamma|x|)-\gamma|x|-1\big].
	\end{equation*}
	Then, for BSDE (\ref{8.2}), applying It\^{o}'s formula to $\Phi(Y^i_t)$, we have
	\begin{equation*}\label{22.9.20.1}
		\begin{aligned}
			\Phi(Y^i_t)
			=&\ \Phi(Y^i_T)+\int_{t}^{T}\Phi'(Y^i_s)g\big(s,Y^i_s,Z^i_s,\nu_s^N, \mu_s^N \big)ds \\
			&   -\int_{t}^{T}\Phi'(Y^i_s)\sum\limits_{j=1}^NZ^{i,j}_sdW^j_s
			-\frac{1}{2}\int_{t}^{T}\Phi''(Y^i_s)\sum\limits_{j=1}^N|Z^{i,j}_s|^2ds \nn\\
			\les &\ \Phi(\eta^i)+\int_{t}^{T}  |\Phi'(Y^i_s)|
			\Big(\theta_s+K\big[|Y^i_s|+ \cW_2(\nu_s^N, \delta_{\{0\}})\big]\Big)ds
			\nn\\
			& -\int_{t}^{T}\Phi'(Y^i_s)\sum\limits_{j=1}^NZ^{i,j}_sdW^j_s +\frac{1}{2}\int_{t}^{T}\(\gamma|\Phi'(Y^i_s)| -\Phi''(Y^i_s)\)\sum\limits_{j=1}^N|Z^{i,j}_s|^2ds. \nn
		\end{aligned}
	\end{equation*}
	Taking the conditional expectation $\mathbb{E}_t[\cd]$ on both sides, we have from \autoref{ass 3.2}
	\begin{align*}
		& \Phi(Y^i_t)+\frac{1}{2}\mathbb{E}_t\bigg[\int_{t}^{T}\sum\limits_{j=1}^N|Z^{i,j}_s|^2ds \bigg]   \\
		&\les  \Phi(\eta^i)+\mathbb{E}_t\bigg[\int_{t}^{T}|\Phi'(Y^i_s)|
		\Big(\theta_s+K\big[|Y^i_s|+ \cW_2(\nu_s^N, \delta_{\{0\}})\big]\Big)ds\bigg]\\
		&  \les \Phi(K_1)+|\Phi'(C)|\mathbb{E}_t\bigg[\int_{t}^{T}
		\Big(\theta_s+K |Y^i_s|+ K\Big\{\frac{1}{N}\sum\limits_{i=1}^N|Y^i_s|^2\Big\}^\frac{1}{2} \Big)ds\bigg]\\
		& \les  \Phi(K_1)+|\Phi'(C)| \Big(\sqrt{K_3 T}+2KCT\Big).
	\end{align*}
	Hence,
	\begin{equation*}
		\|Z^{i,j}\|^2_{ \cZ^2_\dbF(0,T)}
		=\|Z^{i,j}\cdot W^j\|^2_{BMO(\mathbb{P})}\leq2\Phi(K_1)+2|\Phi'(C)| \big(\sqrt{K_3 T}+2KCT\big),
	\end{equation*}
	which implies that $Z^{i,j}\cdot W^j$ is a BMO martingale and thus (\ref{5.111}) holds.
\end{proof}

The particle systems \rf{8.2}-\rf{8.2.2} have the following convergence.

\begin{theorem}\label{th 8.1}\it 
	Under  \autoref{ass 3.3-0}, for any $p\geq2$, there exist two constants $q_0, q'_0>1$ and  a positive constant
	$C$, depending only on $(K,K_1,K_3,T,\g,\phi(\cd),p,q_0,q'_0)$, such that, for $i=1,\cds,N,$
	\begin{align*}
		\mathrm{(i)}\quad  &\mathbb{E}\Big[\frac{1}{N}\sum\limits_{i=1}^N\Big\{\sup\limits_{t\in[0,T]}|\Delta Y_t^i|^p
		+\(\int_0^T \sum\limits_{j=1}^N|\Delta Z^{i,j}_t|^2dt\)^\frac{p}{2}\Big\}\Big] \nn\\
		&\leq C\mathbb{E}\Big[\int_0^T \mathcal{W}^{pq_0q'_0}_2(\nu_t^N,\bar{\nu}_t)dt+ \int_0^T \mathcal{W}^{pq_0q'_0}_2(\mu_t^N,\bar{\mu}_t)dt\Big]^\frac{1}{q_0q'_0},\nn\\
		\mathrm{(ii)}\quad &
		\mathbb{E}\Big\{\sup\limits_{t\in[0,T]}|\Delta Y_t^i|^p+\(\int_0^T \sum\limits_{j=1}^N
		|\Delta Z^{i,j}_t|^2dt\)^\frac{p}{2}\Big\}\\
		&\leq C\mathbb{E}\Big[\int_0^T \mathcal{W}^{pq_0q'_0}_2(\nu_t^N,\bar{\nu}_t)dt+ \int_0^T \mathcal{W}^{pq_0q'_0}_2(\mu_t^N,\bar{\mu}_t)dt\Big]^\frac{1}{q_0q'_0}.\nn
	\end{align*}
\end{theorem}

\begin{proof}
	From BSDE \rf{8.2} and the notation \rf{22.8.13.1}, we see that the pair $(\D Y^i,\D Z^{i,j})$ satisfies the following equation
	\begin{equation*}\label{22.8.29.2}
		\left\{\ba{ll}
		-d\Delta Y_t^i=\big[g(t,Y_t^i,Z_t^{i,i},\nu_t^N,\mu_t^N)
		-g(t,\bar{Y}_t^i,\bar{Z}_t^{i},\bar{\nu}_t, \bar{\mu}_t)\big]dt
		-\sum\limits_{j=1}^N\Delta Z_t^{i,j}dW^j_t,\quad t\in[0,T];\\
		\q \Delta Y^i_T=0,
		\ea\right.
	\end{equation*}
	where
	$$
	\begin{aligned}
		&g(t,Y_t^i,Z_t^{i,i},\nu_t^N,\mu_t^N)-g(t,\bar{Y}_t^i,\bar{Z}_t^{i},\bar{\nu}_t, \bar{\mu}_t)=J_{1,t}+J_{2,t},
	\end{aligned}
	$$
	with
	$$J_{1,t}\deq g(t,Y_t^i,Z_t^{i,i},\nu_t^N,\mu_t^N)-g(t,\bar{Y}_t^i,Z_t^{i,i},\bar{\nu}_t, \bar{\mu}_t),$$
	and
	$$J_{2,t}\deq g(t,\bar{Y}_t^i,Z_t^{i,i},\bar{\nu}_t, \bar{\mu}_t)-g(t,\bar{Y}_t^i,\bar{Z}_t^{i,i},\bar{\nu}_t, \bar{\mu}_t).$$
	On the term $J_{2,t}$, by  \autoref{ass 3.3-0}, there is a process $\Gamma^i$ such that
	$$J_{2,t}=\Gamma^i_t \Delta Z_t^{i,i},$$
	where
	\begin{equation*}\label{22.9.9.1}
		|\Gamma^i_t|\leq \phi(|\bar{Y}^i_t|\vee
		||\bar{Y}^i_t||_{L^2(\Omega)})(1+|Z^{i,i}_t|+|\bar{Z}^{i,i}_t|).
	\end{equation*}
	By  \rf{4.7.0} and \rf{5.111}, for each $i$,  we have $||\G^i||_{\cZ^2_\dbF(0,T)}\les C$ for a constant $C$ depending only on  $(K,K_1,K_3,T)$.
	Now, thanks to Girsanov theorem, we have
	\begin{equation*}
		-d\Delta Y_t^i=J_{1,t}dt-\sum\limits_{j=1}^N\Delta Z_t^{i,j}d\widetilde{W}^{j,i}_t,\quad t\in[0,T].
	\end{equation*}
	Here $\wt W^{j,i}$ depends on $i$, which means one only changes the $i$-th component of $(W^1,\cd\cd\cd,W^N)$,
	i.e., 
	$$\widetilde{W}^{j,i}_t=
	\left\{\ba{ll}
	\ds W^j_t-\int_0^t\Gamma^i_rdr,\q\ j=i;\\
	\ns\ds W^j_t,\qq\qq\q\ \ \  j\neq i.
	\ea\right.$$
	Then $(\wt W^{j,i})_{1\leq j\leq N}$ is an N-dimensional $\mathbb{Q}^i$-Brownian motion with
	$d\mathbb{Q}^i=\sE^i(\Gamma\cdot W^j)_0^Td\mathbb{P}$, where $\G=(0,\cd\cd\cd,0,\G^i,0,\cd\cd\cd,0).$
	From Briand et al. \cite{Briand-Delyon-Hu-Pardoux-Stoica-03}, it follows for fixed $i$,
	\begin{equation}\label{22.9.9.11111}
		\begin{aligned}
			&\dbE^{\dbQ^i}\[\sup_{t\in[0,T]}|\Delta Y_t^i|^p+\(\int_0^T\sum_{j=1}^N|\Delta Z_t^{i,j}|^2dt \)^{\frac{p}{2}}\]\\
			&\leq C\mathbb{E}^{\dbQ^i}\Big[\int_0^T \mathcal{W}^p_2(\nu_t^N,\bar{\nu}_t)dt+ \int_0^T \mathcal{W}^p_2(\mu_t^N,\bar{\mu}_t)dt\Big].
		\end{aligned}
	\end{equation}
	Notice that $\|\G\cd W^j\|_{BMO}$ and $\|\G\cd W^j\|_{{BMO}_q}, q>2$ are equivalent,  from \autoref{22.9.1.1} there exists a constant $p_0>1$
	such that $$\mathbb{E}\Big[(\sE^i(\Gamma\cdot W^j)_0^T)^{p_0}\]\leq C_{p_0}.$$
	For this $p_0$, by summing over $i$ on both sides of (\ref{22.9.9.11111}) and then applying H\"{o}lder inequality, we have for $p\geq2$,
	\begin{equation*}
		\begin{aligned}
			&\frac{1}{N}\sum_{i=1}^N\bigg\{\dbE^{\dbQ^i}\[\sup_{t\in[0,T]}|\Delta Y_t^i|^p+\(\int_0^T\sum_{j=1}^N|\Delta Z_t^{i,j}|^2dt \)^{\frac{p}{2}}\]\bigg\}\\
			&\leq \frac{C}{N}\sum_{i=1}^N\mathbb{E}^{\dbQ^i}\Big[\int_0^T \mathcal{W}^p_2(\nu_t^N,\bar{\nu}_t)dt+ \int_0^T \mathcal{W}^p_2(\mu_t^N,\bar{\mu}_t)dt\Big]\\
			&\leq \frac{C}{N}\sum_{i=1}^N\mathbb{E}\Big[\sE^i(\Gamma\cdot W^j)_0^T\cd\(\int_0^T \mathcal{W}^p_2(\nu_t^N,\bar{\nu}_t)dt+ \int_0^T \mathcal{W}^p_2(\mu_t^N,\bar{\mu}_t)dt\)\Big]\\
			&\leq \frac{C}{N}\sum_{i=1}^N
			\mathbb{E}\Big[(\sE^i(\Gamma\cdot W^j)_0^T)^{p_0}\]^\frac{1}{p_0}
			\mathbb{E}\Big[\(\int_0^T \mathcal{W}^p_2(\nu_t^N,\bar{\nu}_t)dt+ \int_0^T \mathcal{W}^p_2(\mu_t^N,\bar{\mu}_t)dt\)^{q_0}\Big]^\frac{1}{q_0}\\
			&\leq \frac{C}{N}\sum_{i=1}^N\mathbb{E}\Big[\(\int_0^T \mathcal{W}^p_2(\nu_t^N,\bar{\nu}_t)dt+ \int_0^T \mathcal{W}^p_2(\mu_t^N,\bar{\mu}_t)dt\)^{q_0}\Big]^\frac{1}{q_0}\\
			&\leq C\mathbb{E}\Big[\int_0^T \mathcal{W}^{pq_0}_2(\nu_t^N,\bar{\nu}_t)dt+ \int_0^T \mathcal{W}^{pq_0}_2(\mu_t^N,\bar{\mu}_t)dt\Big]^\frac{1}{q_0},
		\end{aligned}
	\end{equation*}
	where $q_0=\frac{p_0}{p_0-1}>1$.

	On the other hand,  it is easy to see that $d\mathbb{P}=\sE^i(-\Gamma\cdot \widetilde{W}^{j,i})_0^T d\mathbb{Q}^i.$
	Similar to the above argument, for fixed $i$,
	thanks to \autoref{22.9.1.1} there exists a constant $p'_0>1$ such that $$\dbE^{\dbQ^i}\[\(\sE^i(-\Gamma\cdot \widetilde{W}^{j,i})_0^T\)^{p'_0}\]\leq C_{p'_0}.$$
	For the above $p'_0$, Cauchy-Schwarz inequality and H\"{o}lder inequality allow us to show that for any $p\geq2$,
	\begin{equation*}
		\begin{aligned}
			&\frac{1}{N}\sum_{i=1}^N\bigg\{\dbE\[\sup_{t\in[0,T]}|\Delta Y_t^i|^p+\(\int_0^T\sum_{j=1}^N|\Delta Z_t^{i,j}|^2dt \)^{\frac{p}{2}}\]\bigg\}\\
			&\leq\frac{1}{N}\sum_{i=1}^N\bigg\{\dbE^{\dbQ^i}\[\sE^i(-\Gamma\cdot \widetilde{W}^{j,i})_0^T\cd
			\(\sup_{t\in[0,T]}|\Delta Y_t^i|^p+\(\int_0^T\sum_{j=1}^N|\Delta Z_t^{i,j}|^2dt \)^{\frac{p}{2}}\)\]\bigg\}\\
			&\leq\frac{C}{N}\sum_{i=1}^N
			\bigg\{\dbE^{\dbQ^i}\[\(\sE^i(-\Gamma\cdot \widetilde{W}^{j,i})_0^T\)^{p'_0}\]\bigg\}^{\frac{1}{p'_0}}\\
			&\q \cdot   \bigg\{\dbE^{\dbQ^i}\[
			\(\sup_{t\in[0,T]}|\Delta Y_t^i|^{pq'_0}+\(\int_0^T\sum_{j=1}^N|\Delta Z_t^{i,j}|^2dt \)^{\frac{pq'_0}{2}}\)\]\bigg\}^{\frac{1}{q'_0}}\\
			&\leq\frac{C}{N}\sum_{i=1}^N\bigg\{\dbE^{\dbQ^i}\[
			\(\sup_{t\in[0,T]}|\Delta Y_t^i|^{pq'_0}+\(\int_0^T\sum_{j=1}^N|\Delta Z_t^{i,j}|^2dt \)^{\frac{pq'_0}{2}}\)\]\bigg\}^{\frac{1}{q'_0}}\\
			&\leq C\bigg\{\frac{1}{N}\sum_{i=1}^N\dbE^{\dbQ^i}\[
			\sup_{t\in[0,T]}|\Delta Y_t^i|^{pq'_0}+\(\int_0^T\sum_{j=1}^N|\Delta Z_t^{i,j}|^2dt \)^{\frac{pq'_0}{2}}\]\bigg\}^{\frac{1}{q'_0}},
		\end{aligned}
	\end{equation*}
	where $q'_0=\frac{p'_0}{p'_0-1}>1$.

	Finally, combining the above equalities we arrive at
	\begin{equation*}
		\begin{aligned}
			&\frac{1}{N}\sum_{i=1}^N\bigg\{\dbE\[\sup_{t\in[0,T]}|\Delta Y_t^i|^p+\(\int_0^T\sum_{j=1}^N|\Delta Z_t^{i,j}|^2dt \)^{\frac{p}{2}}\]\bigg\}\\
			&\leq C\mathbb{E}\Big[\int_0^T \mathcal{W}^{pq_0q'_0}_2(\nu_t^N,\bar{\nu}_t)dt+ \int_0^T \mathcal{W}^{pq_0q'_0}_2(\mu_t^N,\bar{\mu}_t)dt\Big]^\frac{1}{q_0q'_0}.
		\end{aligned}
	\end{equation*}
	Thanks to the exchangeability of the $(Y_t,Z_t)$, we have the item (ii).
\end{proof}

\begin{remark}\label{re 8.1}\rm
	In BSDE \rf{MFBSDE},
	if the generator $g$ does not depend on the law of $Z$, then the inequalities appeared in \autoref{th 8.1} become that, for any $p\geq2$ and for 
	$i=1,\cds,N,$
	\begin{equation*}
		\begin{aligned}
			\mathrm{(i)}\quad &\mathbb{E}\Big[\frac{1}{N}\sum\limits_{i=1}^N\Big\{\sup\limits_{t\in[0,T]}|\Delta Y_t^i|^p
			+\(\int_0^T \sum\limits_{j=1}^N|\Delta Z^{i,j}_t|^2dt\)^\frac{p}{2}\Big\}\Big]\\
			&\leq C
			\mathbb{E}\Big[\int_0^T \mathcal{W}^{pq_0q'_0}_2(\nu_t^N,\bar{\nu}_t)dt\Big]^\frac{1}{q_0q'_0},\\
			\mathrm{(ii)}\quad &
			\mathbb{E}\Big[\sup\limits_{t\in[0,T]}|\Delta Y_t^i|^p+\(\int_0^T \sum\limits_{j=1}^N
			|\Delta Z^{i,j}_t|^2dt\)^\frac{p}{2}\Big]
			\leq C
			\mathbb{E}\Big[\int_0^T \mathcal{W}^{pq_0q'_0}_2(\nu_t^N,\bar{\nu}_t)dt\Big]^\frac{1}{q_0q'_0}.
		\end{aligned}
	\end{equation*}
\end{remark}

\begin{lemma}\label{le 8.1} \it 
	Let  \autoref{ass 3.3-0}   hold true and assume that $g$ does not depend on the law of $Z$. Then, for $p>2$, there
	exist two constants $q_1, q'_1>1$ such that
	$$
	\sup_{t\in[0,T]}\dbE\[ \cW_2^p(\nu_t^N,\bar{\nu}_t)\]\leq CN^{-\frac{1}{2q_1q'_1}}.
	$$
\end{lemma}
\begin{proof}
	Let $(\bar{Y}^i,\bar{Z}^i)$ be the solution of (\ref{8.2.2}), and let $(\widetilde{Y}^i,\widetilde{Z}^i), i=1,2,\cd\cd\cd, N$ be i.i.d. copies of  $(\bar{Y}^i,\bar{Z}^i)$
	such that
	\begin{equation}\label{8.1.000}
		\widetilde{Y}^i_t=\eta^i+\int_t^Tg(s,\widetilde{Y}^i_s,\widetilde{Z}^{i}_s, \bar{\nu}_s)ds-\int_t^T\widetilde{Z}^{i}_sdW^i_s,\q t\in[0,T].
	\end{equation}
	%
	%
	%
	%
	%
	%
	%
	Then
	\begin{equation}\label{8.1.001}\begin{aligned}
			\widetilde{Y}^i_t-Y^i_t&=\int_t^Tg(s,\widetilde{Y}^i_s,\widetilde{Z}^{i}_s,\bar{\nu}_s)-g(s,Y^i_s,\widetilde{Z}^{i}_s,\bar{\nu}_s)
			+g(s,Y^i_s,\widetilde{Z}^{i}_s, \bar{\nu}_s)-g(s,Y^i_s,Z^{i,i}_s, \bar{\nu}_s)\\
			&\q+g(s,Y^i_s,Z^{i,i}_s,\bar{\nu}_s)-g(s,Y^i_s,Z^{i,i}_s, \nu^N_s)ds
			-\int_t^T\sum_{j=1}^N\(\d_{ij}\widetilde{Z}^{i}_s-Z^{i,j}_s \)dW^j_s,
	\end{aligned}\end{equation}
	where $\d_{ij}=1,$ if $i=j$; or else, it equals to $0$.
	Since
	$$\left\{ \begin{aligned}
		&g(s,Y^i_s,\widetilde{Z}^{i}_s, \bar{\nu}_s)-g(s,Y^i_s,Z^{i,i}_s, \bar{\nu}_s)=\Theta^i_s(\widetilde{Z}^{i}_s-Z^{i,i}_s);\\
		&|\Theta^i_s|\leq\phi(|Y_s^i|\vee\|Y_s^i\|_{L^2(\O)})(1+|\widetilde{Z}^{i}_s|+|Z^{i,i}_s|),
	\end{aligned}\right.$$
	so for  given $\bar{\nu}_s$,  the equation (\ref{8.1.000}) is a standard BSDE with quadratic growth. Thanks to Hu and Tang \cite[Theorem 2.3]{Hu-Tang-16}, $|\widetilde{Z}^{i}|$ belongs to $\cZ^2_{\dbF}[0,T]$. Hence, according to \autoref{22.8.29.1}, one can know that
	$\Theta^i$ also belongs to $\cZ^2_{\dbF}[0,T]$. Define $\widehat{W}^{j,i}(s)=W^j(s)-\int_0^s\Theta^i_rdr$. Then
	$(\widehat{W}^{j,i})_{1\leq j\leq N}$ is an $N$-dimensional Brownian motion under the probability $\widehat{\dbQ}^i$, which is defined by $d\widehat{\dbQ}^i=\sE^i(\Theta\cdot W^j)_0^T d\mathbb{P}$ with $\Th=(0,\cd\cd\cd,0,\Th^i,0,\cd\cd\cd,0).$
	Consequently, (\ref{8.1.001}) can be written as
	\begin{equation*}\label{8.1.002}\begin{aligned}
			\widetilde{Y}^i_t-Y^i_t
			&=\int_t^Tg(s,\widetilde{Y}^i_s,\widetilde{Z}^{i}_s,\bar{\nu}_s)-g(s,Y^i_s,\widetilde{Z}^{i}_s,\bar{\nu}_s)
			+g(s,Y^i_s,Z^{i,i}_s,\bar{\nu}_s)-g(s,Y^i_s,Z^{i,i}_s, \nu^N_s)ds\\
			&\q-\int_t^T\sum_{j=1}^N\(\d_{ij}\widetilde{Z}^{i}_s-Z^{i,j}_s \)d\widehat{W}^{j,i}_s.
	\end{aligned}\end{equation*}
	Thanks to Briand et al. \cite[Proposition 3.2]{Briand-Delyon-Hu-Pardoux-Stoica-03}, we have for fixed $i$ and for $p>2$,
	\begin{equation*}\label{8.1.003}\begin{aligned}
			\dbE^{\widehat{\dbQ}^i}\[\sup_{t\in[0,T]}|\widetilde{Y}^i_t-Y^i_t|^p\]\leq C\dbE^{\widehat{\dbQ}^i}\[\int_0^T\cW_2^p(\nu_s^N,\bar{\nu}_s)ds\].
	\end{aligned}\end{equation*}
	Hence,   Cauchy-Schwarz inequality allows to show that for $t\in[0,T]$,
	\begin{equation*}\label{8.1.004}\begin{aligned}
			&\dbE^{\widehat{\dbQ}^i}\bigg[\cW_2^p( \nu_t^N,\widetilde{\nu}_t^N)\bigg]
			\leq\dbE^{\widehat{\dbQ}^i}\bigg[ \bigg( \frac{1}{N}\sum_{i=1}^N|Y^i_t-\widetilde{Y}^i_t|^2\bigg)^\frac{p}{2}\bigg] \\     &
			\leq\frac{1}{N}\sum_{i=1}^N\dbE^{\widehat{\dbQ}^i}\bigg[ |Y^i_t-\widetilde{Y}^i_t|^p\bigg]
			\leq \frac{C}{N}\sum_{i=1}^N \dbE^{\widehat{\dbQ}^i} \[\int_0^T\cW_2^p(\nu_s^N,\bar{\nu}_s)ds\],
	\end{aligned}\end{equation*} 
	where   $\widetilde{\nu}^N_t$ is defined similarly to ${\nu}^N_t$ as mentioned in  \rf{24.1.24}, but with the substitution of ${Y}^i_t$ by $\widetilde{Y}^i_t$.
	By summing over $i$ on both sides of the past inequality, we have, for $t\in[0,T]$,
	\begin{equation*}\label{8.1.005}\begin{aligned}
			\frac{1}{N}\sum_{i=1}^N\dbE^{\widehat{\dbQ}^i}\bigg[\cW_2^p( \nu_t^N,\widetilde{\nu}_t^N)\bigg]
			\leq C\frac{1}{N}\sum_{i=1}^N\dbE^{\widehat{\dbQ}^i}\[\int_0^T\cW_2^p(\nu_s^N,\bar{\nu}_s)ds\].
	\end{aligned}\end{equation*}
	From this and the triangle inequality, we know, for $t\in[0,T]$,
	\begin{equation*}\label{8.1.005}\begin{aligned}
			&\frac{1}{N}\sum_{i=1}^N\dbE^{\widehat{\dbQ}^i}\bigg[\cW_2^p(\nu_t^N, \bar{\nu}_t)\bigg]
			\leq \frac{1}{N}\sum_{i=1}^N\dbE^{\widehat{\dbQ}^i}\bigg[\(\cW_2(\nu_t^N,\widetilde{\nu}_t^N)+\cW_2(\widetilde{\nu}_t^N, \bar{\nu}_t)\)^p\bigg]\\
			&\leq2^p\frac{1}{N}\sum_{i=1}^N\dbE^{\widehat{\dbQ}^i}\bigg[(\cW_2^p(\nu_t^N,\widetilde{\nu}_t^N)+\cW_2^p(\widetilde{\nu}_t^N, \bar{\nu}_t))\bigg]\\
			&\leq C\frac{1}{N}\sum_{i=1}^N\dbE^{\widehat{\dbQ}^i}\[\int_0^T\cW_2^p(\nu_s^N,\bar{\nu}_s)ds\]
			+2^p\frac{1}{N}\sum_{i=1}^N\dbE^{\widehat{\dbQ}^i}\bigg[\cW_2^p(\widetilde{\nu}_t^N, \bar{\nu}_t))\bigg].
	\end{aligned}\end{equation*}
	Using Gronwall's inequality, we deduce, for $t\in[0,T]$,
	\begin{equation}\label{8.1.006}\begin{aligned}
			\frac{1}{N}\sum_{i=1}^N\dbE^{\widehat{\dbQ}^i}\[\cW_2^p(\nu_t^N, \bar{\nu}_t)\]\leq C \frac{1}{N}\sum_{i=1}^N\dbE^{\widehat{\dbQ}^i}\[\cW_2^p(\widetilde{\nu}_t^N, \bar{\nu}_t)\].
	\end{aligned}\end{equation}
	%

	On the one hand, similar to \autoref{th 8.1}, thanks to \autoref{22.9.1.1}, there exists a constant $p_1>1$ such that
	$$\max_{1\les i\les N}\mathbb{E}\Big[(\sE^i(\Theta\cdot W^j)_0^T)^{p_1}\]\leq C_{p_1}.$$
	Define $q_1=\frac{p_1}{p_1-1}$. Clearly, $q_1>1.$ From H\"{o}lder inequality,  we have, for $t\in[0,T]$,
	\begin{align}\label{8.1.007}
		&\frac{1}{N}\sum_{i=1}^N\dbE^{\widehat{\dbQ}^i}\[\cW_2^p(\widetilde{\nu}_t^N, \bar{\nu}_t)\]
		\leq\frac{1}{N}\sum_{i=1}^N\dbE\[ \sE^i(\Theta\cdot W^j)_0^T\cd   \cW_2^p(\widetilde{\nu}_t^N, \bar{\nu}_t)\]\nn \\
		&\leq\frac{1}{N}\sum_{i=1}^N\bigg(
		\Big\{\mathbb{E}\Big[(\sE^i(\Theta\cdot W^j)_0^T)^{p_1}\] \Big\}^\frac{1}{p_1}
		\cd \Big\{\mathbb{E}\Big[  \cW_2^{pq_1}(\widetilde{\nu}_t^N, \bar{\nu}_t)     \] \Big\}^\frac{1}{q_1}\bigg)\\
		&\leq C\Big\{\mathbb{E}\Big[  \cW_2^{pq_1}(\widetilde{\nu}_t^N, \bar{\nu}_t)     \] \Big\}^\frac{1}{q_1}. \nn
	\end{align}
	On the other hand, clearly, $d\mathbb{P}=\sE^i((-\Theta)\cdot \widehat{W}^{j,i})_0^T d\widehat{\dbQ}^i$. For fixed $i$, from \autoref{22.9.1.1} there exists a constant $p'_1>1$ such that
	$$\max_{1\les i\les N}\mathbb{E}^{\widehat{\dbQ}^i}
	\Big[(\sE^i((-\Theta)\cdot \widehat{W}^{j,i})_0^T)^{p'_1}\]\leq C_{p'_1}.$$
	Define $q'_1=\frac{p'_1}{p'_1-1}$. Obviously, $q'_1>1.$ Then one has, for $t\in[0,T]$,
	\begin{equation}\label{8.1.008}\begin{aligned}
			&\dbE\[\cW_2^p(\nu_t^N, \bar{\nu}_t)\]=\frac{1}{N}\sum_{i=1}^N\dbE\[\cW_2^p(\nu_t^N, \bar{\nu}_t)\]
			=\frac{1}{N}\sum_{i=1}^N\dbE^{\widehat{\dbQ}^i}\[ \sE^i((-\Theta)\cdot \widehat{W}^{j,i})_0^T\cd\cW_2^p(\nu_t^N, \bar{\nu}_t)\]\\
			&\leq\frac{1}{N}\sum_{i=1}^N\bigg(
			\Big\{\dbE^{\widehat{\dbQ}^i}\Big[(\sE^i((-\Theta)\cdot \widehat{W}^{j,i})_0^T)^{p'_1}\] \Big\}^\frac{1}{p'_1}
			\cd \Big\{\mathbb{E}^{\widehat{\dbQ}^i}\Big[  \cW_2^{pq'_1}(\nu_t^N, \bar{\nu}_t)     \] \Big\}^\frac{1}{q'_1}\bigg)\\
			&\leq C\frac{1}{N}\sum_{i=1}^N\Big\{\mathbb{E}^{\widehat{\dbQ}^i}\Big[  \cW_2^{pq'_1}(\nu_t^N, \bar{\nu}_t)\] \Big\}^\frac{1}{q'_1}.
	\end{aligned}\end{equation}
	According to (\ref{8.1.006})-(\ref{8.1.008}), we have from Cauchy-Schwarz inequality, for $t\in[0,T]$,
	\begin{equation}\label{8.1.009}\begin{aligned}
			&\dbE\[\cW_2^p(\nu_t^N, \bar{\nu}_t)\]
			\leq C\bigg\{\frac{1}{N}\sum_{i=1}^N\mathbb{E}^{\widehat{\dbQ}^i}\Big[  \cW_2^{pq'_1}(\nu_t^N, \bar{\nu}_t)\] \bigg\}^\frac{1}{q'_1}\\
			&\leq C\bigg\{\frac{1}{N}\sum_{i=1}^N\dbE^{\widehat{\dbQ}^i}\[\cW_2^{pq'_1}(\widetilde{\nu}_t^N, \bar{\nu}_t)\] \bigg\}^\frac{1}{q'_1}
			\leq C\bigg\{\dbE\[\cW_2^{pq_1q'_1}(\widetilde{\nu}_t^N, \bar{\nu}_t)\] \bigg\}^\frac{1}{q_1q'_1}.
	\end{aligned}\end{equation}

	Next, we show for $t\in[0,T]$ and $p>2$, $\dbP$-a.s.,
	\begin{equation}\label{8.1.0091}\begin{aligned}
			\cW_2^{pq_1q'_1}(\widetilde{\nu}_t^N, \bar{\nu}_t)\leq\cW_{pq_1q'_1}^{pq_1q'_1}(\widetilde{\nu}_t^N, \bar{\nu}_t).
	\end{aligned}\end{equation}
	In fact, from the definition of $p$-Wasserstein metric, for arbitrary $\e>0$, there exists a $\rho_0\in \cH$, where $\cH$ is
	the set of probability measures $\rho$ on $\dbR^n\ts \dbR^n$ with the first and second marginal laws $\widetilde{\nu}_t^N$ and $\bar{\nu}_t$,
	such that
	$$\begin{aligned}
		\int_{\dbR^{2n}}|x-y|^{pq_1q'_1}\rho_0(dx,dy)\leq\cW_{pq_1q'_1}^{pq_1q'_1}(\widetilde{\nu}_t^N, \bar{\nu}_t)+\e.
	\end{aligned}$$
	From H\"{o}lder inequality, we have for $t\in[0,T]$ and $p>2$, $\dbP$-a.s.,
	\begin{align*}
		&\cW_2^{pq_1q'_1}(\widetilde{\nu}_t^N, \bar{\nu}_t)=\inf_{\rho\in\cH}\bigg\{\int_{\dbR^{2n}}|x-y|^2\rho(dx,dy)\bigg\}^{\frac{pq_1q'_1}{2}}\\
		&\leq\bigg\{\int_{\dbR^{2n}}|x-y|^2\rho_0(dx,dy)\bigg\}^{\frac{pq_1q'_1}{2}}\\
		&\leq\int_{\dbR^{2n}}|x-y|^{pq_1q'_1}\rho_0(dx,dy)\leq\cW_{pq_1q'_1}^{pq_1q'_1}(\widetilde{\nu}_t^N, \bar{\nu}_t)+\e.
	\end{align*}
	Letting $\varepsilon\rightarrow0$ we have $\dbP$-a.s., $\cW_2^{pq_1q'_1}(\widetilde{\nu}_t^N, \bar{\nu}_t)\leq\cW_{pq_1q'_1}^{pq_1q'_1}(\widetilde{\nu}_t^N, \bar{\nu}_t).$

	Combining (\ref{8.1.009}) and (\ref{8.1.0091}) we have from  Fournier and Guillin \cite[Theorem 1]{Fournier-Guillin-15},
	\begin{equation*}\label{8.1.1111111}\begin{aligned}
			&\dbE\[\cW_2^p(\nu_t^N, \bar{\nu}_t)\]\leq C\bigg\{\dbE\[\cW_{pq_1q'_1}^{pq_1q'_1}(\widetilde{\nu}_t^N, \bar{\nu}_t)\] \bigg\}^\frac{1}{q_1q'_1}\leq CN^{-\frac{1}{2q_1q'_1}},\q t\in[0,T].
	\end{aligned}\end{equation*}
	This completes the proof.
\end{proof}

As an immediate consequence of \autoref{le 8.1} and \autoref{re 8.1}, we have the following result concerning the rate of convergence.

\begin{theorem}\label{th 8.2-1} \it 
	Let   \autoref{ass 3.3-0} be in force and assume that the generator $g$ is independent of the law of $Z$.
	Then, for $p\geq2$, there exist four constants $q_0,q'_0,q_1,q'_1>1$ and
	a positive constant  $C$, depending only
	on $(T,K, K_1,K_3,\phi(\cd),\g,p,q_0,q'_0,q_1,q'_1)$,   such that, for $i=1,\cds,N,$
	\begin{equation*}
		\begin{aligned}
			\mathrm{(i)}\quad  &\mathbb{E}\Big[\frac{1}{N}\sum\limits_{i=1}^N\Big\{\sup\limits_{t\in[0,T]}|\Delta Y_t^i|^p
			+\(\int_0^T \sum\limits_{j=1}^N|\Delta Z^{i,j}_t|^2dt\)^\frac{p}{2}\Big\}\Big]
			\leq CN^{-\frac{1}{2q_0q'_0q_1q'_1}},\\
			\mathrm{(ii)}\quad &
			\mathbb{E}\Big[\sup\limits_{t\in[0,T]}|\Delta Y_t^i|^p+\(\int_0^T \sum\limits_{j=1}^N
			|\Delta Z^{i,j}_t|^2dt\)^\frac{p}{2}\Big]
			\leq CN^{-\frac{1}{2q_0q'_0q_1q'_1}}.
		\end{aligned}
	\end{equation*}
\end{theorem}


%

\begin{remark}\rm
In \autoref{th 8.2-1}, if the generator $g$ depends on the law of $Z$, the convergence rate relies on estimating $\mathbb{E}\big [\mathcal{W}^{pq_0q'_0}_2 (\mu_t^N,\bar{\mu}_t)\big]$. However, this estimation is still an open problem at the moment. The difficulty arises from the fact that, in general, the process $Z$ does not possess uniform H\"older continuity in time:
$$
\mathbb{E}[|Z_t-Z_s|^2]\leq C|t-s|, \quad 0\leq s\leq t\leq T,
$$
which is true for the process $Y$ as a crucial  fact  used in the proof of \autoref{le 8.1} (see also Lauriere and Tangpi \cite[Lemma 3.2]{Lauriere-Tangpi-22} and Briand et al.
\cite[Lemma 1]{Brand-Cardaliaguet-Chaudru-Hu-2020}).

%
%
\end{remark}

\subsection{The case that the generator is unbounded in the law of $Z$}

We make the following assumption.
\bas{ass 3.3}\rm  For $i=1,2,\cdot\cdot\cdot,N$, there exists a positive constant $K$ such that the terminal value $\eta^i:\Om\rightarrow\dbR$ and the generator $g: \Omega\times[0,T]\times\dbR\times\dbR^d\times \cP_2(\dbR)\times \cP_2(\dbR^{d})\rightarrow \dbR$, adapted the filtration of
$\dbF^i$, where $\dbF^i=\{\sF^i_t\}_{t\ges0}$ is the natural filtration of $W^i$ augmented by all the $\dbP$-null sets in $\sF^i$,
satisfy the following conditions:
\begin{enumerate}[~~\,\rm (i)]
\item $d\dbP\times dt$-a.e. $(\o,t)\in\Om\ts[0,T]$, for every $(y, z,\nu,\mu)\in \mathbb{R}\ts\mathbb{R}^d\ts\mathcal{P}_2(\mathbb{R})\ts\mathcal{P}_2(\mathbb{R}^{d}),$
\begin{equation*}
	|g(w,t,y,z,\nu,\mu)|\les \theta_t(\omega)+K|y|+\frac{\gamma}{2}|z|^2+K\cW_2(\nu,\d_{\{0\}})+\gamma_0\cW_2(\mu,\d_{\{0\}}). 
\end{equation*}
\item $d\dbP\times dt$-a.e. $(\o,t)\in\Om\ts[0,T]$, for every $y,\by\in \dbR, z,\bz\in \dbR^d$ and $\nu,\bar{\nu}\in \cP_2(\dbR)$,
$\m,\bar{\m}\in \cP_2(\dbR^{d})$,
$$\ba{ll}
& |g(\o,t,y,z,\nu,\m)-g(\o,t,\bar y,\bz,\bar\nu,\bar\mu)| \les K\big[|y-\bar{y}|+\cW_2(\nu,\bar{\nu})
+\cW_2(\mu,\bar{\mu})\big]\\
&\q + \phi\big(|y|\vee|\bar y|\vee\cW_2(\nu,\d_{\{0\}})\vee\cW_2(\bar \nu,\d_{\{0\}})\big) \cdot \big[(1+|z|+|\bar{z}|)|z-\bar{z}|\big].

\ea$$
\item  There are two positive constants $K_1$ and $K_2$ such that
$$\max_{1\les i\les N}\|\eta^i\|_\infty\les K_1\q\hb{and}\q \bigg\|\int_0^T\theta_t(\omega)dt\bigg\|_\infty\les K_2.$$
\end{enumerate}\eas

\begin{proposition}\label{22.8.29.1111} \it 
Let \autoref{ass 3.2-11} with $\a=1$ and \autoref{ass 3.3} hold true, there is a positive constant $C$, depending only on  $(K_1,K_2,T,\g, \tilde{\gamma},\gamma_0,\beta,\beta_0),$  such that the solution $(Y^i, Z^{i,j})$ of BSDE (\ref{8.2}) admits the following estimate: for each $i,j=1,...,N$,
\begin{equation}\label{22.8.26-1}
	||Y^i||_{S_\dbF^\infty(0,T)}\leq C\q\hbox{and}\q  ||Z^{i,j}||_{\cZ^2_\dbF(0,T)}\leq C.
\end{equation}
\end{proposition}
\begin{proof}
Under \autoref{ass 3.2-11} with $\a=1$ and \autoref{ass 3.3}, the generator satisfies \cite[(H1)-(H5)]{FHT}. According to \cite[Theorem 2.5]{FHT},
the equation (\ref{8.2}) admits a unique solution $(Y,Z)\in S_\dbF^\infty(0,T;\dbR)\ts \cZ^2_\dbF(0,T;\dbR^d)$. Next, we show that 
the bounds of $||Y^i||_{S_\dbF^\infty(0,T)}$ and  $||Z^{i,j}||_{\cZ^2_\dbF(0,T)}$ do not depend on $N$.
Similar to (\ref{equ 3.25}) and (\ref{equ 3.322222}), we have, 
for $t\in[T-\kappa,T]$,
\begin{equation}\label{5.5}
	\begin{aligned}
		|Y^i_t|  &\leq (K_1+K_2)+L_3T+\int_t^T\(\beta\|Y^i\|_{S_\dbF^\infty(s,T)} \\
		& \q  
		+\beta_0  \cW_2(\nu_s^N,\d_{\{0\}})   
		+\frac{\tilde{\gamma}\e_0}{2} \cW_2(\mu_s^N,\d_{\{0\}})^2\)ds\\
		&\leq (K_1+K_2)+L_3T+\int_t^T\(\beta\|Y^i\|_{S_\dbF^\infty(s,T)} 
		+ \beta_0 \Big\{\frac{1}{N}\sum_{i=1}^N\|Y^i\|^2_{S_\dbF^\infty(s,T)}\Big\}^\frac{1}{2} \\
		&\q + \frac{\tilde{\gamma}\e_0}{2}\frac{1}{N}\sum_{i=1}^N|Z^{i,i}_s|^2\)ds,
	\end{aligned}
\end{equation}
and

\begin{equation} \label{5.6}
	\begin{aligned}
		&\int_t^T\frac{\tilde{\gamma}\e_0}{2}|Z^{i,i}_s|^2ds\\
		&\leq \e_0(2+T\b)\|Y^i\|_{S_\dbF^\infty(t,T)}+\e_0K_2+ \e_0T\b_0 \Big\{\frac{1}{N}\sum_{i=1}^N\|Y^i\|^2_{S_\dbF^\infty(t,T)}\Big\}^\frac{1}{2}\\
		&\q +\int_t^T\frac{\tilde{\gamma}\e_0}{4} \frac{1}{N}\sum_{i=1}^{N}   |Z^{i,i}_s|^2ds+L_4T+\e_0\int_t^T\sum_{j=1}^{N}Z^{i,j}_sdW^j_s,
	\end{aligned}
\end{equation}
where $L_3$ and $L_4$ are defined in (\ref{equ 3.34-1}) and (\ref{equ 3.34-2}), respectively.
Summing over $i$ on both sides of the last inequality, we have from Gronwall's inequality that 

\begin{equation} \label{5.7}
	\begin{aligned}
		&\int_t^T\frac{\tilde{\gamma}\e_0}{2} \frac{1}{N}\sum_{i=1}^{N} |Z^{i,i}_s|^2ds\\
		&\leq 2\e_0(2+T\b)\frac{1}{N}\sum_{i=1}^{N} \|Y^i\|_{S_\dbF^\infty(t,T)}+2\e_0K_2+ 2\e_0T\b_0 \Big\{\frac{1}{N}\sum_{i=1}^N\|Y^i\|^2_{S_\dbF^\infty(t,T)}\Big\}^\frac{1}{2}\\
		&\q +2L_4T+2\e_0\int_t^T\frac{1}{N}\sum_{i=1}^{N}\sum_{j=1}^{N}Z^{i,j}_sdW^j_s.
	\end{aligned}
\end{equation}
Combining (\ref{5.5}) and (\ref{5.7}),  we have

\begin{equation*} 
	\begin{aligned}
		|Y^i_t|  &\leq (K_1+K_2)+L_3T+\int_t^T\(\beta\|Y^i\|_{S_\dbF^\infty(s,T)} 
		+ \beta_0 \Big\{\frac{1}{N}\sum_{i=1}^N\|Y^i\|^2_{S_\dbF^\infty(s,T)}\Big\}^\frac{1}{2}\)ds \\
		&\q +2\e_0(2+T\b)\frac{1}{N}\sum_{i=1}^{N} \|Y^i\|_{S_\dbF^\infty(t,T)}+2\e_0K_2+ 2\e_0T\b_0 \Big\{\frac{1}{N}\sum_{i=1}^N\|Y^i\|^2_{S_\dbF^\infty(t,T)}\Big\}^\frac{1}{2}\\
		&\q +2L_4T+2\e_0\int_t^T\frac{1}{N}\sum_{i=1}^{N}\sum_{j=1}^{N}Z^{i,j}_sdW_s^j.
	\end{aligned}
\end{equation*}
Taking the conditional expectation $\dbE_t[\cd]$, we have
\begin{equation}\label{5.8}
	\begin{aligned}
		\|Y^i\|_{S_\dbF^\infty(t,T)}  &\leq (K_1+K_2)+L_3T+2L_4T+ 2\e_0K_2\\
		&\q +\int_t^T\(\beta\|Y^i\|_{S_\dbF^\infty(s,T)} 
		+ \beta_0 \Big\{\frac{1}{N}\sum_{i=1}^N\|Y^i\|^2_{S_\dbF^\infty(s,T)}\Big\}^\frac{1}{2}\)ds \\
		&\q +2\e_0(2+T\b)\frac{1}{N}\sum_{i=1}^{N} \|Y^i\|_{S_\dbF^\infty(t,T)}+ 2\e_0T\b_0 \Big\{\frac{1}{N}\sum_{i=1}^N\|Y^i\|^2_{S_\dbF^\infty(t,T)}\Big\}^\frac{1}{2}, 
	\end{aligned}
\end{equation}
which, combining  Cauchy-Schwarz inequality, implies that
\begin{equation*}\label{5.9}
	\begin{aligned}
		\|Y^i\|^2_{S_\dbF^\infty(t,T)}  &\leq 5[(K_1+K_2)+L_3T+2L_4T+ 2\e_0K_2]^2\\
		&\q +5\beta^2 T\int_t^T\|Y^i\|^2_{S_\dbF^\infty(s,T)}ds 
		+5\beta_0^2 T \int_t^T\frac{1}{N}\sum_{i=1}^N\|Y^i\|^2_{S_\dbF^\infty(s,T)} ds \\
		&\q +20\e_0^2 [(2+T\b)^2+ T^2\b_0^2]\frac{1}{N}\sum_{i=1}^N\|Y^i\|^2_{S_\dbF^\infty(t,T)}.
	\end{aligned}
\end{equation*}
Summing over $i$ on both sides of the last inequality, we see 
\begin{equation*} 
	\begin{aligned}
		\frac{1}{N}\sum_{i=1}^{N}\|Y^i\|^2_{S_\dbF^\infty(t,T)}  &\leq 5[(K_1+K_2)+L_3T+2L_4T+ 2\e_0K_2]^2\\
		&\q +5T(\beta^2+\beta_0^2) \int_t^T\frac{1}{N}\sum_{i=1}^{N}\|Y^i\|^2_{S_\dbF^\infty(s,T)}ds \\
		&\q +20\e_0^2 [(2+T\b)^2+ T^2\b_0^2]\frac{1}{N}\sum_{i=1}^N\|Y^i\|^2_{S_\dbF^\infty(t,T)}.
	\end{aligned}
\end{equation*}
Choosing $\e_0$ such that $20\e_0^2 [(2+T\b)^2+ T^2\b_0^2]=\frac{1}{2}$,  and using Gronwall lemma, we have 
\begin{equation}\label{5.11}
	\begin{aligned}
		\frac{1}{N}\sum_{i=1}^{N}\|Y^i\|^2_{S_\dbF^\infty(t,T)}  &\leq C_1,
	\end{aligned}
\end{equation}
where
\begin{equation*} 
	\begin{aligned}
		C_1=10[(K_1+K_2)+L_3T+2L_4T+ 2\e_0K_2]^2e^{10T^2(\beta^2+\beta_0^2) }.
	\end{aligned}
\end{equation*}
Insert  (\ref{5.11}) into (\ref{5.8}),  we have 
\begin{equation*} 
	\begin{aligned}
		\|Y^i\|_{S_\dbF^\infty(t,T)}  &\leq (K_1+K_2)+L_3T+2L_4T+ 2\e_0K_2+\b_0T\sqrt{C_1}+2\e_0T\b_0\sqrt{C_1} \\
		&\q +\int_t^T\beta\|Y^i\|_{S_\dbF^\infty(s,T)} ds +2\e_0(2+T\b)\frac{1}{N}\sum_{i=1}^{N} \|Y^i\|_{S_\dbF^\infty(t,T)}.
	\end{aligned}
\end{equation*}
Summing over $i$ on both sides of the last inequality,  from the choice of $\e_0$ and  Gronwall inequality, we have 
\begin{equation}\label{5.12}
	\begin{aligned}
		\frac{1}{N}\sum_{i=1}^{N} \|Y^i\|_{S_\dbF^\infty(t,T)}  &\leq C_2,
	\end{aligned}
\end{equation}
where
\begin{equation*} 
	\begin{aligned}
		C_2=2[(K_1+K_2)+L_3T+2L_4T+ 2\e_0K_2+\b_0T\sqrt{C_1}+2\e_0T\b_0\sqrt{C_1}]  e^{2\b T}.
	\end{aligned}
\end{equation*}
Finally, according to  (\ref{5.8}), (\ref{5.11}) and (\ref{5.12}),  we obtain from Gronwall lemma, 
\begin{equation*}
	\begin{aligned}
		\|Y^i\|_{S_\dbF^\infty(t,T)}  &\leq C_3, \quad i=1,2,\cds N, 
	\end{aligned}
\end{equation*}
where
\begin{equation*}
	\begin{aligned}
		C_3&=[(K_1+K_2)+L_3T+2L_4T+ 2\e_0K_2+(1+2\e_0)\b_0T\sqrt{C_1}+2\e_0(2+T\b)C_2]e^{\b T}. 
	\end{aligned}
\end{equation*}

Now we prove that $Z^{i,j}\cdot W^j$ is a BMO martingale. Define
\begin{equation*}
	\Phi(x)=\frac{1}{\gamma^2}\big[\exp(\gamma|x|)-\gamma|x|-1\big].
\end{equation*}
Again, for BSDE (\ref{8.2}), applying It\^{o}'s formula to $\Phi(Y^i_t)$, we have
\begin{equation*}\label{22.9.20.1}
	\begin{aligned}
		\Phi(Y^i_t)
		=&\ \Phi(Y^i_T)+\int_{t}^{T}\Phi'(Y^i_s)g\big(s,Y^i_s,Z^i_s,\nu_s^N, \mu_s^N \big)ds \\
		&   -\int_{t}^{T}\Phi'(Y^i_s)\sum\limits_{j=1}^NZ^{i,j}_sdW^j_s
		-\frac{1}{2}\int_{t}^{T}\Phi''(Y^i_s)\sum\limits_{j=1}^N|Z^{i,j}_s|^2ds \nn\\
		\les &\ \Phi(\eta^i)+\int_{t}^{T}  |\Phi'(Y^i_s)|
		\Big(\theta_s+K\big[|Y^i_s|+ \cW_2(\nu_s^N, \delta_{\{0\}})\big]+\gamma_0\cW_2(\mu_s^N,\d_{\{0\}})\Big)ds
		\nn\\
		& -\int_{t}^{T}\Phi'(Y^i_s)\sum\limits_{j=1}^NZ^{i,j}_sdW^j_s +\frac{1}{2}\int_{t}^{T}\(\gamma|\Phi'(Y^i_s)| -\Phi''(Y^i_s)\)\sum\limits_{j=1}^N|Z^{i,j}_s|^2ds. \nn
	\end{aligned}
\end{equation*}
Taking the conditional expectation $\mathbb{E}_t[\cd]$ on both sides, we have
\begin{equation*}
	\begin{aligned}
		& \Phi(Y^i_t)+\frac{1}{2}\mathbb{E}_t\bigg[\int_{t}^{T}\sum\limits_{j=1}^N|Z^{i,j}_s|^2ds \bigg]   \\
		&\les \ \Phi(\eta^i)+\mathbb{E}_t\bigg[\int_{t}^{T}|\Phi'(Y^i_s)|
		\Big(\theta_s+K\big[|Y^i_s|+ \cW_2(\nu_s^N, \delta_{\{0\}})\big]+\gamma_0\cW_2(\mu_s^N,\d_{\{0\}})\Big)ds\bigg]\\
		&\les \ \Phi(K_1)+|\Phi'(C_3)|\mathbb{E}_t\bigg[\int_{t}^{T}
		\Big(\theta_s+K |Y^i_s|+ K\Big\{\frac{1}{N}\sum\limits_{i=1}^N|Y^i_s|^2\Big\}^\frac{1}{2}\\
		&\hskip 4.2cm
		+\g_0 \Big\{\frac{1}{N}\sum\limits_{i=1}^N|Z^{i,i}_s|^2\Big\}^\frac{1}{2} \Big)ds\bigg].\\
	\end{aligned}
\end{equation*}
According to (\ref{equ 3.22}), we have
\begin{equation*}
	\begin{aligned}
		&|\Phi'(C_3)|\g_0 \Big\{\frac{1}{N}\sum\limits_{i=1}^N|Z^{i,i}_s|^2\Big\}^\frac{1}{2} 
		=\frac{1}{4}\Big[4|\Phi'(C_3)|\g_0 \Big\{\frac{1}{N}\sum\limits_{i=1}^N|Z^{i,i}_s|^2\Big\}^\frac{1}{2}\Big] \\
		&\leq \frac{1}{4}\Big\{\frac{1}{N}\sum\limits_{i=1}^N|Z^{i,i}_s|^2\Big\}+\frac{1}{16}\(4|\Phi'(C_3)|\g_0\)^2. 
	\end{aligned}
\end{equation*}
Combining the last two inequalities, we have
\begin{equation*}
	\begin{aligned}
		\frac{1}{2}\mathbb{E}_t\bigg[\int_{t}^{T}\sum\limits_{j=1}^N|Z^{i, j}_s|^2ds \bigg]
		\les &\ C_4+\frac{1}{4}\mathbb{E}_t\bigg[\int_{t}^{T}\frac{1}{N}\sum\limits_{i=1}^N|Z^{i,i}_s|^2ds \bigg],
	\end{aligned}
\end{equation*}
where $$C_4=\Phi(K_1)+|\Phi'(C_3)| \Big(K_2+ KC_3+K\sqrt{C_1}  \Big)+\frac{1}{16}\(4|\Phi'(C_3)|\g_0\)^2.$$
Summing over $i$ on both sides of the above inequality, we get
\begin{equation*}
	\begin{aligned}
		\mathbb{E}_t\bigg[\int_{t}^{T}\frac{1}{N}\sum\limits_{i=1}^N|Z^{i,i}_s|^2ds \bigg]\les   4C_4.
	\end{aligned}
\end{equation*}
Consequently, 
\begin{equation*}
	\begin{aligned}
		\mathbb{E}_t\bigg[\int_{t}^{T}\sum\limits_{j=1}^N|Z^{i, j}_s|^2ds \bigg]
		\les &\ 4C_4.
	\end{aligned}
\end{equation*}
Then, we have
\begin{equation*}
	\begin{aligned}
		\mathbb{E}_t\bigg[\int_{t}^{T} |Z^{i, j}_s|^2ds \bigg]
		\les &\ 4C_4, \quad i, j=1,2, \cds, N.
	\end{aligned}
\end{equation*}
Hence, for $t\in[T-\kappa,T],$
\begin{equation*}
	\|Z^{i,j}\|^2_{ \cZ^2_\dbF(t,T)}
	=\|Z^{i,j}\cdot W^j\|^2_{BMO([t,T])}\leq4C_4. 
\end{equation*}
Finally,  similar to the proof of \autoref{th 3.5}, we have (\ref{22.8.26-1}).
\end{proof}

Note that the proofs of convergence and convergence rate for the particle systems \rf{8.2}-\rf{8.2.2} only rely on the uniformly Lipschitz continuity of $g$ and do not require the boundedness assumption on $g$. Therefore, the proofs under \autoref{ass 3.2-11} and \autoref{ass 3.3} can be considered the same as those under \autoref{ass 3.3-0}. Thus, we can state the results without including the proof.

\begin{theorem}\label{th 8.5}\it 
	Under \autoref{ass 3.2-11} with $\a=1$ and \autoref{ass 3.3}, for any $p\geq2$, there exist two constants $q_0, q'_0>1$ and a constant $C>0$, depending only on $(K,K_1,K_2,T,\phi(\cd),\g,p,q_0,q'_0)$, such that, for $i=1,\cds, N,$
	\begin{align*}
		\mathrm{(i)}\quad  &\mathbb{E}\Big[\frac{1}{N}\sum\limits_{i=1}^N\Big\{\sup\limits_{t\in[0,T]}|\Delta Y_t^i|^p
		+\(\int_0^T \sum\limits_{j=1}^N|\Delta Z^{i,j}_t|^2dt\)^\frac{p}{2}\Big\}\Big] \nn\\
		&\leq C\mathbb{E}\Big[\int_0^T \mathcal{W}^{pq_0q'_0}_2(\nu_t^N,\bar{\nu}_t)dt+ \int_0^T \mathcal{W}^{pq_0q'_0}_2(\mu_t^N,\bar{\mu}_t)dt\Big]^\frac{1}{q_0q'_0},\nn\\
		\mathrm{(ii)}\quad &
		\mathbb{E}\Big\{\sup\limits_{t\in[0,T]}|\Delta Y_t^i|^p+\(\int_0^T \sum\limits_{j=1}^N
		|\Delta Z^{i,j}_t|^2dt\)^\frac{p}{2}\Big\}\\
		&\leq C\mathbb{E}\Big[\int_0^T \mathcal{W}^{pq_0q'_0}_2(\nu_t^N,\bar{\nu}_t)dt+ \int_0^T \mathcal{W}^{pq_0q'_0}_2(\mu_t^N,\bar{\mu}_t)dt\Big]^\frac{1}{q_0q'_0}.\nn
	\end{align*}
\end{theorem}

\begin{theorem}\label{th 8.2-1} \it 
	Let \autoref{ass 3.2-11} with $\alpha=1$ and \autoref{ass 3.3} be in force, and assume that the generator $g$ is independent of the law of $Z$. 
	Then, for any $p\geq2$, there exist four constants $q_0,q'_0,q_1,q'_1>1$ and a constant   $C>0$, depending only on $(K,K_1,K_2,T,\phi(\cd),\g,p,q_0,q'_0,q_1,q'_1)$, such that, for $i=1,\cds,N,$
	\begin{equation*}
		\begin{aligned}
			\mathrm{(i)}\quad  &\mathbb{E}\Big[\frac{1}{N}\sum\limits_{i=1}^N\Big\{\sup\limits_{t\in[0,T]}|\Delta Y_t^i|^p
			+\(\int_0^T \sum\limits_{j=1}^N|\Delta Z^{i,j}_t|^2dt\)^\frac{p}{2}\Big\}\Big]
			\leq CN^{-\frac{1}{2q_0q'_0q_1q'_1}},\\
			\mathrm{(ii)}\quad &
			\mathbb{E}\Big[\sup\limits_{t\in[0,T]}|\Delta Y_t^i|^p+\(\int_0^T \sum\limits_{j=1}^N
			|\Delta Z^{i,j}_t|^2dt\)^\frac{p}{2}\Big]
			\leq CN^{-\frac{1}{2q_0q'_0q_1q'_1}}.
		\end{aligned}
	\end{equation*}
\end{theorem}


\section{Applications to PDE}\label{Sec5}

The master equation can be regarded as a PDE in the Wasserstein space, whose state variable refers to the distribution of some state process. However, finding the viscosity solution of the master equation is always a challenging topic. This is due to the Wasserstein space lacking local compactness, which is a necessary element in viscosity theory (see Wu and Zhang \cite{WZ}). 
Furthermore, studying general mean-field BSDEs and the associated master equations with quadratic growth adds further complexity to the problem. In addition to quadratic growth, the nonlinearity of the coefficient $g$ in mean-field BSDEs can lead to time-inconsistency of the value function. The regularity of the value function under weak formulation is also more involved, even for the Lipschitz case.
Due to these challenges, instead of considering the distribution of the state process, it may be more appropriate to consider situations where the coefficients depend on the expectation of the state process. This alternative approach takes into account the expectations of the state process, rather than its distribution.

In detail, for any initial pair $(t,x)\in[0,T]\ts\dbR^n$ and frozen $x_0\in\dbR^n$, we consider the following mean-field forward-backward stochastic differential equation:
\begin{equation}\label{FBSDE2}
\left\{\begin{aligned}
	dX^{t,x}_s&=\dbE'\big[b(s,(X^{0,x_0}_s)',X^{t,x}_s)]ds
	+\dbE'[\sigma(s,(X^{0,x_0}_s)',X^{t,x}_s)\big]dW_s,\\
	-dY^{t,x}_s&=\dbE'\big[g(s,(X^{0,x_0}_s)',X^{t,x}_s,(Y^{0,x_0}_s)',Y^{t,x}_s,Z^{t,x}_s)\big]ds
	- Z^{t,x}_sdW_s,\q~ t\les s\les T, \\
	X^{t,x}_t&=x\in\dbR^n, \q~ Y^{t,x}_T=\dbE'\big[\Phi((X^{0,x_0}_T)',X^{t,x}_T)\big],
\end{aligned}\right.
\end{equation}
where the coefficients $b:[0,T]\times \dbR^n\times \dbR^n\rightarrow \dbR^n$ and
$\sigma:[0,T]\times \dbR^n\times \dbR^n\rightarrow \dbR^{n\times d}$
satisfy the following condition.
\bas{ass 5.2-1}\rm
There is a positive constant $C$ such that for any $t\in[0,T]$, $ x',x,\bar x',\bar x\in\dbR^n$,
$$|b(t,x',x)|+|\sigma(t,x',x)|\leq C(1+|x|+|x'|),$$
and
$$|b(t, x',x)-b(t,\bar x',\bar x)|+|\sigma(t,x',x)-\sigma(t,\bar x',\bar x)|
\leq C(|x'-\bar x'|+|x-\bar x|).$$
\eas
Then under \autoref{ass 5.2-1}, it is well known that the forward equation of \rf{FBSDE2} admits a unique adapted solution, denoted by $\{X^{t,x}_s;t\les s\les T\}$, in the space $S^2_{\dbF}(t,T;\dbR^n)$. Besides, the following result holds (see Buckdahn, Li, and Peng \cite{BLP}).
\begin{lemma}\label{pro 5.2}\it 
For every $p\geq2$, there exists a positive constant $C_p$ such that for any $t\in[0,T]$, $\d\in[0,T-t]$, and $x,x'\in \dbR^n$,
\begin{align*}
	& \dbE_t\Big[\sup\limits_{s\in[t,T]}|X^{t,x}_s-X^{t,x'}_s|^p\Big]\leq C_p|x-x'|^p,\\
	& \dbE_t\Big[\sup\limits_{s\in[t,T]}|X^{t,x}_s|^p\Big]\leq C_p(1+|x|^p),\\
	& \dbE_t\Big[\sup\limits_{s\in[t,t+\d]}|X^{t,x}_s-x|^p\Big]\leq C_p(1+|x|^p)\delta^\frac{p}{2}.
\end{align*}
\end{lemma}
On the other hand, for the coefficients $g:[0,T]\times \dbR^n\times \dbR^n\times \dbR\times \dbR\times \dbR^d\rightarrow   \dbR$ and $\Phi:\dbR^n\times \dbR^n\rightarrow\dbR$ of the backward equation of \rf{FBSDE2}, we present the following assumption.
\bas{ass 5.3}\rm
There exist positive constants $C$, $C_0$ and $C_1$ such that for any $t\in[0,T], x',x,\bar x',\bar x\in \dbR^n$, $y',y,\bar y',\bar y\in \dbR$ and $\bar z,z\in \dbR^d$, the generator $g$ is differential with respect to $x$, $y$, $z$, and
\begin{align*}
&|g(t,x',x,y',y,z)|\leq C(1+|z|^2)\q\hb{and}\q|\Phi(x',x)|\leq C,\nn\\
&|g(t,x',x,y,y',z)-g(t,\bar{x}',\bar{x},\bar{y},\bar{y}',\bar{z})|\\
&\ \leq  C(|x'-\bar{x}'|+|x-\bar{x}|+|y'-\bar{y}'|+|y-\bar{y}|)
+C(1+|z|+|\bar z|)|z-\bar{z}|,\nn\\
%
%
%
%
%
&g(t,x',x,y',y,z)\hb{ is continuous with respect to }t,\nn\\
&g(t,x',x,y',y,z)\hb{ is nondecreasing in }y'.\nn
\end{align*}
\eas

\begin{remark}\label{re 5.4}\rm
Under \autoref{ass 5.3}, it is easy to check that  for any $t\in[0,T], x',x\in \dbR^n$, $y',y\in \dbR$, $z',z\in \dbR^d$, and any $\e>0$, the following conditions, appeared in Kobylanski \cite{Kobylanski00}, hold:
\begin{align*}
	&|\frac{\partial g}{\partial x}(t,x',x,y',y,z)|\leq C(1+|z|)\leq C\(\frac{3}{2}+|z|^2\),\\
	&|\frac{\partial g}{\partial y}(t,x',x,y',y,z)|\leq C(1+|z|)\leq C+\frac{C^2}{4\varepsilon}+\varepsilon|z|^2,\\
	&|\frac{\partial g}{\partial z}(t,x',x,y',y,z)|\leq C(1+|z|).
\end{align*}
\end{remark}

Due to that \autoref{ass 5.3} is stronger then \autoref{ass 3.2}, then combining \autoref{th 3.5}, it is not hard to verify that
the backward equation of  \rf{FBSDE2} admits a unique global solution in $S_{\dbF}^\infty(t,T;\dbR)\times  \cZ^2_\dbF(t,T;\dbR^d)$. In other words, we have the following result.

\begin{proposition}\label{th 5.5}\it 
Let \autoref{ass 5.3} hold, then the backward equation of \rf{FBSDE2}
possesses a unique global solution, denoted by $\{(Y^{t,x}_s,Z^{t,x}_s);t\les s\les T\}$, in the space
$S_{\dbF}^\infty(t,T;\dbR)\times  \cZ^2_\dbF(t,T;\dbR^d)$.
\end{proposition}

\begin{proof}
On the one hand, from \autoref{th 3.5}, we see that for
$(t,x)=(0,x_0)$, the backward equation of \rf{FBSDE2} possesses a global solution
$\{(Y^{0,x_0}_s,Z^{0,x_0}_s);0\les s\les T\}$ in the space
$S_{\dbF}^\infty(0,T;\dbR)\times  \cZ^2_\dbF(0,T;\dbR^d)$.
Moreover, $\|Y^{0,x_0}\|_{S_\dbF^\infty(0,T)}$ and $ \|Z^{0,x_0}\|^2_{ \cZ^2_\dbF(0,T)}$
are bounded.
On the other hand, once knowing $(Y^{0,x_0},Z^{0,x_0})$, we can define that for any
$s\in[t,T], y\in \dbR$ and $z\in \dbR^d$,
$$
g^\#(s,X_s^{t,x},y,z)=\dbE'[g(s,(X_s^{0,x_0})',X_s^{t,x},(Y_s^{0,x_0})',y,z)],\quad
\Phi^\#(x)=\dbE'[\Phi((X_T^{0,x_0})',x)].
$$
Then,
it is easy to check that $g^\#$  and $\Phi^\#$ also satisfy the condition  (6)  of Hibon, Hu, and Tang \cite{Hibon-Hu-Tang-17}.
Hence,  \cite[Lemma 2.1]{Hibon-Hu-Tang-17}  (see also  Hu and Tang \cite[Theorem 2.3]{Hu-Tang-16}) implies that BSDE \rf{BSDE} with generator $g^\#$ and terminal value $\Phi^\#$ has a unique
global  solution $(Y,Z)\in S_{\dbF}^\infty(t,T;\dbR)\times  \cZ^2_\dbF(t,T;\dbR^d)$.
In other words, the backward equation of \rf{FBSDE2} admits a unique global solution
$\{(Y^{t,x}_s,Z^{t,x}_s);t\les s\les T\}$ in the space
$S_{\dbF}^\infty(t,T;\dbR)\times  \cZ^2_\dbF(t,T;\dbR^d)$.
\end{proof}

We  define the value function as follows:
\begin{equation}\label{equ 5.2-1}
u(t,x)\deq Y^{t,x}_t,\q~ t\in[0,T],\ x\in \dbR^n.
\end{equation}
Note that $u(t,x)$ is continuous with respect to $(t,x)$. In fact, on the one hand,
\autoref{pro 5.2} deduces that the flow $(t,x)\mapsto X^{t,x}_s$ is continuous.
On the other hand, the stability result  (see Kobylanski \cite{Kobylanski00}) implies that $(t,x,s)\mapsto Y^{t,x}_s$ is continuous  under \autoref{ass 5.3}.
Thus, in particular, the deterministic function $u(t,x)=Y^{t,x}_t$ is certainly continuous with respect to $(t,x)$.

Now, we would like to connect the forward-backward system \rf{FBSDE2} with quadratic growth to the following nonlocal partial differential equation:
\begin{equation}\label{equ 5.3}
\left\{
\begin{aligned}
	&\frac{\partial v(t,x)}{\partial t}+\cL v(t,x)\\
	& +\dbE\Big[g\(t,X_t^{0,x_0},x,v(t,X_t^{0,x_0}),v(t,x),
	Dv(t,x)\dbE\big[\sigma(t,X_t^{0,x_0},x)\big]^\intercal\)\Big]=0,\\
	&\qq (t,x)\in[0,T)\times\dbR^n,\\
	&v(T,x)=\dbE\big[\Phi(X_T^{0,x_0},x)\big],\q~ x\in\dbR^n,
\end{aligned}
\right.
\end{equation}
where
\begin{align*}
\cL v(t,x)=&\ \frac{1}{2}\mathrm{tr}\Big(
\dbE\big[\sigma(t,X_t^{0,x_0},x)\big]
\dbE\big[\sigma(t,X_t^{0,x_0},x)\big]^\intercal
D^2v(t,x)\Big)\\
& + Dv(t,x)\dbE\big[b(t,X_t^{0,x_0},x)\big]^\intercal.
\end{align*}
It should be pointed out that in \rf{equ 5.3}, the terminology nonlocal means
$$
\begin{aligned}
&\dbE\Big[g\(t,X_t^{0,x_0},x,v(t,X_t^{0,x_0}),v(t,x),
Dv(t,x)\dbE\big[\sigma(t,X_t^{0,x_0},x)\big]^\intercal\)\Big]\\
& =\int_{\mathbb{R}^n}g\(t,x',x,v(t,x'),v(t,x), Dv(t,x)\dbE\big[\sigma(t,X_t^{0,x_0},x)\big]^\intercal \)\dbP_{X_t^{0,x_0}}(dx'),
\end{aligned}
$$
where $X^{0,x_0}$ is the solution of the forward equation of \rf{FBSDE2} with the initial pair $(0,x_0)$.

In the following,  we prove that the value function $u$ defined in \rf{equ 5.2-1} is the viscosity solution of PDE \rf{equ 5.3}. For this, we extend the approach of
Buckdahn, Li, and Peng \cite{BLP} developed in the framework of mean-field BSDEs with linear growth to that of quadratic growth.
First, we recall the definition of a viscosity solution of PDE \rf{equ 5.3}. For more details about the viscosity solutions,  we refer the reader to Crandall, Ishii, and Lions \cite{Grandall-Ishii-Lions-92}.

Note that for Euclidean space $\dbH$, denote by
$C_{p}(\dbH)$ the set of continuous function on $\dbH$, who grows at most like a polynomial function of the variable $x$ at infinity; and denote by
$C_{l, b}^{k}(\dbH)$ the set of functions of class $C^{k}$ on $\dbH$, whose partial derivatives of order less than or equal to $k$ are bounded.
\begin{definition}[Viscosity solution]\rm
A real-valued continuous function $v\in C_p([0,T]\ts\dbR^n)$ is called
\begin{enumerate}[~~\,\rm (i)]
	\item a viscosity subsolution to PDE (\ref{equ 5.3}), if
	$v(T,x)\leq \dbE\big[\Phi(X_T^{0,x_0},x)\big]$ for all $x\in\dbR^n$, in addition,
	if for arbitrary $\psi\in  C^3_{l,b}([0,T]\times \dbR^n)$ and $(t^*,x^*)\in[0,T)\times \dbR^n$
	such that $v-\psi$ attains its local maximum at $(t^*,x^*),$
	\begin{equation*}
		\begin{aligned}
			&\frac{\partial \psi(t^*,x^*)}{\partial t}+\cL  \psi(t^*,x^*)\\
			& +\dbE\Big[g\(t^*,X_{t^*}^{0,x_0},x^*,v(t^*,X_{t^*}^{0,x_0}),v(t^*,x^*),
			\dbE[\sigma(t^*,X_{t^*}^{0,x_0},x^*)]^\intercal D \psi(t^*,x^*)\)\Big]\geq0.
		\end{aligned}
	\end{equation*}
	\item a viscosity supersolution to PDE (\ref{equ 5.3}), if
	$v(T,x)\ges \dbE\big[\Phi(X_T^{0,x_0},x)\big]$ for all $x\in\dbR^n$, in addition,
	if for arbitrary $\psi\in C^3_{l,b}([0,T]\times \dbR^n)$ and $(t^*,x^*)\in[0,T)\times \dbR^n$
	such that $v-\psi$ attains its local minimum at $(t^*,x^*),$
	\begin{equation*}
		\begin{aligned}
			&\frac{\partial  \psi(t^*,x^*)}{\partial t}+\cL  \psi(t^*,x^*)\\
			& +\dbE\Big[g\(t^*,X_{t^*}^{0,x_0},x^*,v(t^*,X_{t^*}^{0,x_0}),v(t^*,x^*),
			\dbE[\sigma(t^*,X_{t^*}^{0,x_0},x^*)]^\intercal D \psi(t^*,x^*)\)\Big]\leq0.
		\end{aligned}
	\end{equation*}
	\item a viscosity  solution to PDE (\ref{equ 5.3}),
	if it is both a viscosity subsolution and a viscosity supersolution.
\end{enumerate}
\end{definition}

\begin{lemma} \label{le 5.5}\it 
Under \autoref{ass 5.2-1}  and \autoref{ass 5.3}, for any $t\in[0,T]$
and $\xi\in L_{\sF_t}^2(\Om;\dbR^n)$, we have
\begin{equation}\label{equ 5.4-1}
	u(t,\xi)=Y^{t,\xi}_t.
\end{equation}
\end{lemma}

\begin{proof}
The proof  is essentially an adaptation of Proposition 4.7 of Peng \cite{Peng-97}, so we  sketch it.
First, we assert that (\ref{equ 5.4-1}) holds true for a simple random variable $\xi\in L_{\sF_t}^2(\Om;\dbR^n)$ with a form below
\begin{equation}\label{22.5.12.1}
	\xi=\sum\limits_{j=1}^Mx_j\chi_{B_j},
\end{equation}
where $\{ B_j\}_{j=1}^M$ is a finite partition of  $(\Omega, \sF_t)$ and
$x_j\in \mathbb{R}^n$ with $j=1,2,\cdot\cdot\cdot, M.$
For each $x_j$, let $X^{t,x_j}$ be the solution of the following mean-field SDE
\begin{equation*}
	X^{t,x_j}_s=x_j+\int_t^s \dbE'\big[b(r,(X^{0,x_0}_r)',X^{t,x_j}_r)\big]dr
	+\int_t^s\dbE'\big[\sigma(r,(X^{0,x_0}_r)',X^{t,x_j}_r)\big]dW_r, \ t\les s\les T,
\end{equation*}
and  let $(Y^{t,x_j},Z^{t,x_j})$ be the solution of the mean-field BSDE
\begin{equation*}
	\begin{aligned}
		Y^{t,x_j}_s&=\dbE'\big[\Phi((X^{0,x_0}_T)',X^{t,x_j}_T)\big]
		+\int_t^s\dbE'\big[g(r,(X^{0,x_0}_r)',X^{t,x_j}_r,(Y^{0,x_0}_r)',Y^{t,x_j}_r,Z^{t,x_j}_r)\big]dr\\
		&\quad- \int_t^sZ^{t,x_j}_rdW_r,\q~ t\les s\les T.
	\end{aligned}
\end{equation*}
Multiplying $\chi_{A_j}$ on both sides of the above two equations and then summing up with respect to $j$,  note that $\sum_j\psi(x_j)\chi_{A_j}=\psi(\sum_jx_j\chi_{A_j})$, one gets
\begin{equation*}
	\begin{aligned}
		\sum\limits_{j=1}^M\chi_{A_j}X^{t,x_j}_s
		=&\ \sum\limits_{j=1}^M\chi_{A_j}x_j+\int_t^s \dbE'\[b(r,(X^{0,x_0}_r)',\sum\limits_{j=1}^M\chi_{A_j}X^{t,x_j}_r)\]dr\\
		&+\int_t^s\dbE'\[\sigma(r,(X^{0,x_0}_r)',
		\sum\limits_{j=1}^M\chi_{A_j}X^{t,x_j}_r)\]dW_r, \q~t\les s\les T,
	\end{aligned}
\end{equation*}
and for $t\les s\les T$,
\begin{equation*}
	\begin{aligned}
		\sum\limits_{j=1}^M\chi_{A_j}Y^{t,x_j}_s
		=&\ \dbE'\[\Phi\((X^{0,x_0}_T)',\sum\limits_{j=1}^M\chi_{A_j}X^{t,x_j}_T\)\]
		-\int_t^s\sum\limits_{j=1}^M\chi_{A_j}Z^{t,x_j}_rdW_r\\
		+\int_t^s&\dbE'\[g\(r,(X^{0,x_0}_r)',\sum\limits_{j=1}^M\chi_{A_j}X^{t,x_j}_r,(Y^{0,x_0}_r)',
		\sum\limits_{j=1}^M\chi_{A_j}Y^{t,x_j}_r,
		\sum\limits_{j=1}^M\chi_{A_j}Z^{t,x_j}_r\)\]dr.
	\end{aligned}
\end{equation*}
Then from the existence and uniqueness of the mean-field BSDE with quadratic growth (see \autoref{th 3.5}), we conclude that for $t\les s\les T$,
\begin{equation*}
	X^{t,\xi}_s=\sum\limits_{j=1}^M\chi_{A_j}X^{t,x_j}_s,\q\
	Y^{t,\xi}_s=\sum\limits_{j=1}^M\chi_{A_j}Y^{t,x_j}_s,\q\
	Z^{t,\xi}_s=\sum\limits_{j=1}^M\chi_{A_j}Z^{t,x_j}_s.
\end{equation*}
Finally, from the definition of $u(t,x)$, it yields
$$
Y^{t,\xi}_t=\sum\limits_{j=1}^M\chi_{A_j}Y^{t,x_j}_t=\sum\limits_{j=1}^M\chi_{A_j}u(t,x_j)
=u(t,\sum\limits_{j=1}^M x_j\chi_{A_j})=u(t,\xi).
$$
Now, for given $\xi\in L^2_{\sF_t}(\Omega;\mathbb{R}^n)$, there exists a sequence of simple variables $\{\xi_j\}$ admitting the form \rf{22.5.12.1} that converges to $\xi$ in $L^2_{\sF_t}(\Omega;\mathbb{R}^n)$.
Since the flow $(t,x)\mapsto X^{t,x}_s$  and $(t,x,s)\mapsto Y^{t,x}_s$ are continuous,
$u(t,x)$ is continuous with respect to $(t,x)$ and note the fact $Y^{t,\xi_j}_t=u(t,\xi_j)$, we have
\begin{equation*}
	\begin{aligned}
		&\mathbb{E}\big[|Y^{t,\xi}_t-u(t,\xi)|^2\big]
		= \mathbb{E}\big[|Y^{t,\xi}_t-Y^{t,\xi_j}_t
		+Y^{t,\xi_j}_t-u(t,\xi_j)+u(t,\xi_j)-u(t,\xi)|^2\big] \\
		&\leq 2\mathbb{E}\big[|Y^{t,\xi}_t-Y^{t,\xi_j}_t|^2\big]
		+\mathbb{E}\big[|u(t,\xi_j)-u(t,\xi)|^2\big] \rightarrow0\q
		\text{as}\q  j\rightarrow\infty.
	\end{aligned}
\end{equation*}
Hence, we have (\ref{equ 5.4-1}).
\end{proof}

Based on the above preparation, now we state the main result of this section.

\begin{theorem}[Feynman-Kac formula]\label{th 5.7}\it 
Under \autoref{ass 5.2-1}  and \autoref{ass 5.3}, the value function $u$ defined in (\ref{equ 5.2-1}) is the unique viscosity solution of PDE (\ref{equ 5.3}).
\end{theorem}

\begin{proof}
From  \autoref{le 5.5}  and the uniqueness of the mean-field forward-backward SDE (\ref{FBSDE2}) with the initial pair $(t,x)=(0,x_0)$, we have that
\begin{equation}\label{equ 5.4-8}
	Y^{0,x_0}_t=Y^{t,X^{0,x_0}_t}_t=u(t,X^{0,x_0}_t),\q~ t\in[0,T].
\end{equation}
Based on the value of $X^{0,x_0}$ and $Y^{0,x_0}$, one can define that
\begin{equation}\label{equ 5.9}
	\begin{aligned}
		\tilde{b}(t,x)&=\dbE\big[b(t,X_t^{0,x_0},x)\big],\quad \tilde{\sigma}(t,x)=\dbE\big[\sigma(t,X_t^{0,x_0},x)\big],\\
		\tilde{g}(t,x,y,z)&=\dbE\big[g(t,X_t^{0,x_0},x, Y_t^{0,x_0},y,z)\big],\quad \tilde{\Phi}(x)=\dbE\big[\Phi(X_T^{0,x_0},x)\big].
	\end{aligned}
\end{equation}
By \autoref{re 5.4}, it is easy to check that the parameters  $(\tilde{b}, \tilde{\sigma},\tilde{g},\tilde{\Phi})$ satisfy the assumptions $\mathrm{(H4)}$ and $\mathrm{(H5)}$ of Kobylanski \cite{Kobylanski00}.
In view of  Theorem 3.8 of \cite{Kobylanski00},  the function $u$ is a viscosity solution to the following PDE
\begin{equation*}\label{equ 5.10}
	\left\{
	\begin{aligned}
		&\frac{\partial u(t,x)}{\partial t}+
		\frac{1}{2}\mathrm{tr}\big(
		\tilde{\sigma}(t,x)
		\tilde{\sigma}(t,x)^\intercal
		D^2u(t,x)\big)
		+\tilde{b}(t,x)^\intercal Du(t,x)\\
		&\quad
		+\tilde{g}\big(t,x,u(t,x), \tilde{\sigma}(t,x)^\intercal Du(t,x)\big)=0,\q (t,x)\in[0,T)\times\dbR^n;\\
		&u(T,x)=\tilde{\Phi}(x),\q x\in\dbR^n.
	\end{aligned}
	\right.
\end{equation*}
Finally, from the definitions of the parameters $(\tilde{b}, \tilde{\sigma},\tilde{g},\tilde{\Phi})$ and (\ref{equ 5.4-8}), we see that $u$ is also a viscosity solution to PDE (\ref{equ 5.3}).

Next, we prove the uniqueness of PDE (\ref{equ 5.3}).
Let  both  $u^1$ and $u^2$ be the viscosity solutions of PDE (\ref{equ 5.3}).
For any $t\in[0,T]$, $x\in\dbR^n$, $y\in\dbR$ and $z\in\dbR^d$, we set
\begin{align*}
	g^1(t,x,y,z)=\dbE\[g\(t,X_t^{0,x_0},x,u^1(t,X_t^{0,x_0}),y,z\)\],\\
	g^2(t,x,y,z)=\dbE\[g\(t,X_t^{0,x_0},x,u^2(t,X_t^{0,x_0}),y,z\)\].
\end{align*}
Then, for $i=1,2$, $u^i$ is a viscosity solution to the following PDE
\begin{equation*}\label{equ 5.12}
	\left\{
	\begin{aligned}
		&\frac{\partial u^i(t,x)}{\partial t}+
		\frac{1}{2}\mathrm{tr}\big(
		\tilde{\sigma}(t,x)
		\tilde{\sigma}(t,x)^\intercal
		D^2u^i(t,x)\big)
		+\tilde{b}(t,x)^\intercal Du^i(t,x)\\
		&\quad
		+g^i\big(t,x,u^i(t,x), \tilde{\sigma}(t,x)^\intercal Du^i(t,x)\big)=0,\q (t,x)\in[0,T)\times\dbR^n;\\
		&u(T,x)=\tilde{\Phi}(x),\q x\in\dbR^n,
	\end{aligned}
	\right.
\end{equation*}
where $\tilde{b}, \tilde{\sigma}, \tilde{\Phi}$ is given in (\ref{equ 5.9}).
Again, thanks to Kobylanski \cite{Kobylanski00}, $u^i$ admits the following
probabilistic interpretation:
\begin{equation}\label{equ 5.13}
	u^i(t,x)=Y^{t,x,i}_t, \q~ (t,x)\in[0,T]\times \dbR^n,
\end{equation}
where $(Y^{t,x,i},Z^{t,x,i})$ is the solution to the following quadratic BSDE
\begin{equation}\label{equ 5.14}
	\left\{
	\begin{aligned}
		-dY^{t,x,i}_s&=g^i(s,X^{t,x}_s,Y^{t,x,i}_s,Z^{t,x,i}_s)ds-Z^{t,x,i}_sdW_s,\q\ s\in[t,T];\\
		Y_T^{t,x,i}&=\tilde{\Phi}(x),\q\ x\in\dbR^n.
	\end{aligned}
	\right.
\end{equation}
Now, we let $(t,x)=(0,x_0)$ in (\ref{equ 5.13}).
Similar to \autoref{le 5.5}, from the continuity of $u^i$ and the uniqueness of the solution to BSDE (\ref{equ 5.14}), one has
\begin{align*}
	&u^i(t,X_t^{0,x_0})=Y^{0,x_0,i}_t,\q\  t\in[0,T].
\end{align*}
Finally, from the definitions of $(g^i,\tilde{b}, \tilde{\sigma}, \tilde{\Phi})$,
we have that the pair $(Y^{t,x,i},Z^{t,x,i})$ is the adapted solution to the following  quadratic BSDE
\begin{equation*}\label{equ 5.15}
	\left\{
	\begin{aligned}
		-dY^{t,x,i}_s&=\dbE'\big[g^i(s, (X^{0,x_0}_s)', X^{t,x}_s, (Y^{0,x_0,i}_s)',Y^{t,x,i}_s,Z^{t,x,i}_s)\big]ds-Z^{t,x,i}_sdW_s,\q s\in[t,T];\\
		Y_T^{t,x,i}&=\dbE'\big[\Phi((X_T^{0,x_0})',X_T^{t,x})\big],\q x\in\dbR^n.
	\end{aligned}
	\right.
\end{equation*}
Then \autoref{th 5.5} implies that
$$Y_s^{t,x,1}=Y_s^{t,x,2},\q\ t\les s\les T.$$
In particular, as $s=t$, one gets that
$$u^1(t,x)=Y_t^{t,x,1}=Y_t^{t,x,2}=u^2(t,x), \q \forall (t,x)\in[0,T]\ts\dbR^n.$$
This completes the proof.
\end{proof}

\begin{remark}\rm
\autoref{th 5.7} establishes the relation between the solution of mean-field BSDEs and the viscosity solution of the nonlocal PDEs with quadratic growth, which extends the related result of Kobylanski \cite{Kobylanski00} to the mean-field framework and extends the nonlinear Feynman-Kac formula of Buckdahn, Li, and Peng \cite{BLP} with linear growth to that of quadratic growth.
\end{remark}

\section{Concluding Remark}\label{Sec6}

We initiate the study of general mean-field BSDEs \rf{MFBSDE} and  give the existence, uniqueness, and comparison results for one-dimensional mean-field BSDEs \rf{MFBSDE} with quadratic growth and with bounded terminal value by introducing some new ideas.
Besides, we obtain the convergence of particle systems for the mean-field BSDE \rf{MFBSDE} with quadratic growth and give the rate of  convergence when generator $g$ is independent of the law of $Z$. However, the convergence rate is still an open problem if the generator $g$ depends on the law of $Z$.
Finally, in this framework, when the generator $g$ depends on the expectation of the state process $(Y,Z)$, we used the mean-field BSDEs \rf{MFBSDE} to prove the existence and uniqueness of the viscosity solutions of the nonlocal PDEs \rf{equ 5.3}, which extend the nonlinear Feynman-Kac formula of Buckdahn, Li, and Peng \cite{BLP} to that of quadratic growth.
On the other hand, when the generator $g$ depends on the  distribution of $Z$,  the existence and uniqueness of viscosity solutions of related PDEs are intriguing and challenging, and remains to be studied in the future.
%
%

%


\end{document}